\documentclass[12pt]{amsart}
\usepackage{enumerate,amssymb,amsfonts,latexsym,psfrag,psfig}
\usepackage{amscd,amsmath}
\usepackage[dvips]{graphicx}

\newtheorem{thm}{Theorem}[section]
\newtheorem{cor}[thm]{Corollary}
\newtheorem{lem}[thm]{Lemma}
\newtheorem{prop}[thm]{Proposition}
\theoremstyle{remark}
\newtheorem{rem}{Remark}[section]

\newtheorem{clm}[thm]{Claim}

\theoremstyle{definition}
\newtheorem{defn}{Definition}[section]

\numberwithin{equation}{section}
\numberwithin{figure}{section}

 \newcommand{\pichere}[2]
{\begin{center}\includegraphics[width=#1\textwidth]{#2}\end{center}}

\font\nt=cmr7

\def\note#1
{\marginpar
{\nt $\leftarrow$
\par
\hfuzz=20pt \hbadness=9000 \hyphenpenalty=-100 \exhyphenpenalty=-100
\pretolerance=-1 \tolerance=9999 \doublehyphendemerits=-100000
\finalhyphendemerits=-100000 \baselineskip=6pt
#1}\hfuzz=1pt}

\newcommand{\QED}{\rlap{$\sqcup$}$\sqcap$\smallskip}

\renewcommand{\Im}{\operatorname{Im}}

\newcommand{\di}{\partial}

\newcommand{\ra}{\rightarrow}

\def\ssk{\smallskip}
\def\msk{\medskip}

\def\nin{\noindent}

\def\sm{\smallsetminus}

\newcommand{\diam}{\operatorname{diam}}
\newcommand{\dist}{\operatorname{dist}}

\renewcommand{\mod}{\operatorname{mod}}

\newcommand{\tl}{\tilde}

\newcommand{\orb}{\operatorname{orb}}

\newcommand{\id}{\operatorname{id}}

\newcommand{\tg}{\operatorname{tg}}

\newcommand{\transverse}
 {\kern .7em\makebox[0pt][c]{\raisebox{.2ex}{$|$}}\kern -.6em\cap}

\newcommand{\tangent}
 {\kern .7em\makebox[0pt][c]{\raisebox{.77ex}{$--$}}\kern -.6em\cap}

\newcommand{\eps}{{\varepsilon}}

\newcommand{\De}{{\Delta}}
\newcommand{\de}{{\delta}}

\newcommand{\La}{{\Lambda}}
\newcommand{\si}{{\sigma}}

\newcommand{\Om}{{\Omega}}
\newcommand{\om}{{\omega}}

\newcommand{\BB}{{\mathcal B}}
\newcommand{\CC}{{\mathcal C}}

\newcommand{\II}{{\mathcal I}}
\newcommand{\FF}{{\mathcal F}}
\newcommand{\GG}{{\mathcal G}}

\newcommand{\HH}{{\mathcal H}}

\newcommand{\OO}{{\mathcal O}}
\newcommand{\PP}{{\mathcal P}}

\renewcommand{\SS}{{\mathcal S}}

\newcommand{\UU}{{\mathcal U}}

\newcommand{\XX}{{\mathcal X}}

\newcommand{\N}{{\mathbb N}}

\newcommand{\R}{{\mathbb R}}

\newcommand{\X}{{\mathbb X}}
\newcommand{\Z}{{\mathbb Z}}

\newcommand{\bh}{{\bf h}}

\newcommand{\bv} {{\bf v  }}
\newcommand{\bw} {{\bf w  }}

\def\Bg{{\mathbf{g}}}

\def\Bh{{\mathbf{h}}}
\def\Bv{{\mathbf{v}}}
\def\Bw{{\mathbf{w}}}

\def\Bde{{\boldsymbol{\de}}}

\def\bB{{\mathbf{B}}}
\def\bQ{{\mathbf{Q}}}

\def\BS{{\mathbf{S}}}
\def\BDe{{\boldsymbol{\Delta}}}
\def\BGa{{\boldsymbol{\Gamma}}}

\def\BPhi{{\boldsymbol{\BPhi}}}
\def\B0{{\mathbf{0}}}

\newcommand{\Jac}{\operatorname{Jac}}

\newcommand{\Dom}{\operatorname{Dom}}

\newcommand{\Domain}{\operatorname{Dom}}

\def\diff{{\text{diff}}}
\def\aff{{\text{aff}}}

\catcode`\@=12

\def\Empty{}
\newcommand\oplabel[1]{
  \def\OpArg{#1} \ifx \OpArg\Empty {} \else
  	\label{#1}
  \fi}
		
\newcommand{\comm}[1]{}
\newcommand{\comment}[1]{}

\begin{document}

\bigskip\bigskip

\title[H\'enon renormalization]{
 Probabilistic  Universality in two-dimensional Dynamics\\}

\author{M. Lyubich, M. Martens}

\address {SUNY at Stony Brook}


\date{\today}

\begin{abstract} 
In this paper we continue to explore infinitely renormalizable H\'enon maps with small  Jacobian. 
It was shown in \cite{CLM} that contrary to the one-dimensional intuition,  
the Cantor attractor of such a map is non-rigid and the conjugacy 
with the one-dimensional Cantor attractor is at most  $1/2$-H\"older.
Another formulation of this phenomenon is that  the scaling structure of the H\'enon Cantor attractor 
differs from its one-dimensional counterpart. However, in this paper we prove 
that the weight assigned by the canonical 
invariant measure to these bad spots tends to zero on microscopic scales. 
This phenomenon is called  {\it Probabilistic Universality}. 
It implies, in particular, that the Hausdorff dimension of the canonical measure is universal. 
In this way,  universality and rigidity phenomena  of one-dimensional dynamics 
assume a probabilistic nature in the two-dimensional  world.
\end{abstract}

\maketitle

\def\IMSmarkvadjust{0 pt}
\def\IMSmarkhadjust{0 pt}
\def\IMSmarkhpadding{0 pt}
\def\IMSpubltext{Published in modified form:}
\def\SBIMSMark#1#2#3{
 \font\SBF=cmss10 at 10 true pt
 \font\SBI=cmssi10 at 10 true pt
 \setbox0=\hbox{\SBF \hbox to \IMSmarkhpadding{\relax}
                Stony Brook IMS Preprint \##1}
 \setbox2=\hbox to \wd0{\hfil \SBI #2}
 \setbox4=\hbox to \wd0{\hfil \SBI #3}
 \setbox6=\hbox to \wd0{\hss
             \vbox{\hsize=\wd0 \parskip=0pt \baselineskip=10 true pt
                   \copy0 \break%
                   \copy2 \break%
                   \copy4 \break}}
 \dimen0=\ht6   \advance\dimen0 by \vsize \advance\dimen0 by 8 true pt
                \advance\dimen0 by -\pagetotal
	        \advance\dimen0 by \IMSmarkvadjust
 \dimen2=\hsize \advance\dimen2 by .25 true in
	        \advance\dimen2 by \IMSmarkhadjust

%
%
  \openin2=publishd.tex
  \ifeof2\setbox0=\hbox to 0pt{}
  \else 
     \setbox0=\hbox to 3.1 true in{
                \vbox to \ht6{\hsize=3 true in \parskip=0pt  \noindent  
                {\SBI \IMSpubltext}\hfil\break
                \input publishd.tex 
                \vfill}}
  \fi
  \closein2
  \ht0=0pt \dp0=0pt
 \ht6=0pt \dp6=0pt
 \setbox8=\vbox to \dimen0{\vfill \hbox to \dimen2{\copy0 \hss \copy6}}
 \ht8=0pt \dp8=0pt \wd8=0pt
 \copy8
 \message{*** Stony Brook IMS Preprint #1, #2. #3 ***}
}

\SBIMSMark{2011/2}{June 2011}{}

\setcounter{tocdepth}{1}
\tableofcontents

\section{Introduction}

Renormalization ideas have played a central role in Dynamics since 
the discovery of the Universality and Rigidity phenomena by Feigenbaum \cite{F}, and independently by Coullet and Tresser \cite{CT},
in the mid 1970s. Roughly speaking, it means that different systems in the same ``universality class''
have the same small scale geometry. In the one-dimensional setting this phenomenon has been viewed from many angles
(statistical physcis,  geometric function theory, Teichm\"uller theory, hyperbolic geometry, infinite-dimensional complex geometry) 
and by now has been fully and rigorously justified, see \cite{Ep}, \cite{FMP}, \cite{L}, \cite{Lan}, \cite{Ma2}, \cite{McM}, \cite{S} and references therein. 

In \cite{CT} Coullet and Tresser also conjectured that these phenomena would also be valid  in higher dimensional systems, even in infinite dimensional situations. Indeed, computer and physical experiments that followed suggested that universality and rigidity hold in much more general context. 
The simplest test case for it is the dissipative H\'enon family which can be viewed as a small perturbation of the one-dimensional  
quadratic  family. However, it was shown in \cite{CLM} that Universality and Rigidity break down already in this case. 
This puts in question the relevance of one-dimensional models for higher dimensional problems.

In this paper we provide a resolution of this unsatisfactory situation:
namely, we show that for dissispative H\'enon maps,  small scale universality is actually valid in {\it probabilistic} sense,
almost everywhere with respect to the canonical invariant measure.  {\it Probabilistic universality} and {\it probabilistic rigidity} phenomena may be valid for higher dimensional (including infinite dimensional) systems which are contracting in all but one direction.

\comm{In this sense, Universality and Rigidity phenomena may be valid for higher dimensional 
(including infinite dimensional systems) which are contracting in all but one direction.}

\msk
Let us now  formulate our results  more precisely. 
We consider a class of dissipative  H\'enon-like maps on the unit box $B^0=[0,1]\times [0,1]$  of form
\begin{equation}\label{Henon family intro}
   F(x,y)= ( f(x)-\eps(x,y),x ),
\end{equation}
where $f(x)$ is a unimodal map with non-degenerate critical point and   $\eps$ is small.
It maps $B^0$ on a slightly thickened parabola $x=f(y)$.  
Such a map is called {\it renormalizable} if there exists a smaller box $B^1\subset B^0$
around the tip of of the parabola which is mapped into itself under $F^2$. 
The {\it renormalization} for $F$ is the map $RF= \Psi^{-1}\circ F^2\circ \Psi$,
where $\Psi: B^0\ra B^1$ is an explicit  non-linear change of variable  (``rescaling'')
that brings $F^2$ to the normal form of type (\ref{Henon family intro}). 

If $RF$ is in turn renormalizable then $F$ is called {\it twice renormalizble}, etc. 
In this paper we will be concerned with {\it infinitely renormalizable} H\'enon-like maps. 
Such a map  admits a nest of $2^n$-periodic boxes  $B^0\supset B^1\supset B^2\supset\dots$ shrinking to the {\it tip}
$\tau$ of $F$.  The $n^{th}$-renormalization cycle is the orbit $\BB^n=\{B^n_i= f^i (B^n), i=0,1,\dots 2^n-1\}$. We obtain a hierarchy of such cycles
shrinking to the  {\it Cantor attractor} 
$$
\OO_F= \bigcap_{n=0}^\infty \bigcup_{i=0}^{2^n-1} B^n_i
$$ 
on which $F$ acts as the dyadic adding machine. 
In particular, the dynamics on $\OO_F$ is uniquely ergodic, so we obtain a canonical invariant measure $\mu$ supported on $\OO_F$. 
We define the {\it average Jacobian} of $F$ as follows:
$$
       b_F=\exp \int_{\OO_F} \log \Jac F d\mu .
$$

\bigskip

Consider a strongly dissipative infinitely renormalizable H\'enon-like map.
The geometry of a piece $B\in \BB^n$ can be very different from the geometry of the corresponding piece $I$ of the one-dimensional renormalization fixed point $f_*$. The pieces of the one-dimensional system are small intervals.
Take a piece $B\in \BB^n$ and the two pieces $B_1, B_2\in \BB^{n+1}$ with $B_1, B_2\subset B$. Let $I, I_1, I_2$ be the corresponding pieces of $f_*$. The piece $B$ of $F$ has {\it $\epsilon-$precision} if after one 
simultaneous  rescaling and translation $A:\Bbb{R}^2\to \Bbb{R}^2$ we have
that the (Hausdorff) distance between $I$ and $A(B)$, $ I_1$ and $A(B_1)$, 
$I_2$ and $A(B_2)$ is at most $\epsilon\cdot \diam(I)$. The triples $B_1, B_2\subset B$ and $I_1,I_2\subset I$ are geometrical almost the same. 

Collect the pieces  of the $n^{th}-$cycle with $\epsilon-$precision in
$$
\SS_n(\epsilon)=\{B\in \BB^n| B \text{ has } \epsilon-{precision}\}. 
$$

\begin{defn} The geometry of the Cantor attractor $\OO_F$ of a dissipative infinitely renormalizable H\'enon-like map is probabilistically universal 
if there exists $\theta<1$ such that
$$
\mu(\SS_n(\theta^n))\ge 1-\theta^n.
$$
\end{defn}

\begin{thm}(Probabilistic universality) The geometry of the Cantor attractor of a strongly dissipative infinitely renormalizable H\'enon-like map is probabilistically universal.
\end{thm}

\begin{defn}  The Cantor attractor $\OO_F$ of a dissipative infinitely renormalizable H\'enon-like map is probabilistically rigid  if the conjugation $h:\OO_F\to \OO_{f_*}$ to the attractor $\OO_{f_*}$ of the one-dimensional renormalization fixed point $f_*$ has the following property. 
There exist $\beta>0$, and a sequence $X_1\subset X_2\subset X_3\subset \cdots \subset \OO_F$ such that
$
h:X_N\to h(X_N)\subset \OO_{f_*}
$
is $(1+\beta)$-differentiable, 
and 
$
\mu(X_N)\to 1.
$
\end{defn}

\begin{thm}(Probabilistic Rigidity) The Cantor attractor of a dissipative infinitely renormalizable H\'enon-like  map is probabilistically rigid.
\end{thm}

The Cantor attractor $\OO_F$ is not part of a smooth curve, see \cite{CLM}. However, large parts of it, the sets 
$$
X_N=\bigcap_{k\ge N} \SS_n(\theta^n)
$$
where $\theta<1$ is close enough to $1$ satisfy 

\begin{thm} Each set $X_N\subset \OO_F$ is part of a smooth $C^{1+\beta}-$curve.
\end{thm}

Let $\mu_*$ be the invariant measure on $\OO_{f_*}$, the attractor of the one-dimensional renormalization fixed point. A consequence of probabilistic rigidity is 

\begin{thm} The Hausdorff dimension is universal
$$
HD_{\mu}(\OO_F)=HD_{\mu_*}(\OO_{f_*}).
$$
\end{thm}

The theory of universality and rigidity became a probabilistic geometrical theory for H\'enon dynamics.

\bigskip
We prove the above results by introducing the so-called {\it pushing-up} machinery. This method locates the pieces in the $n^{th}$-renormalization cycle that have exponential precision.  The difficulty is that the orbit between two such good pieces may pass through poor pieces,
so one cannot recover all good pieces by simple iteration  of the original map.
Instead, the pushing-up machinery relates pieces in the same renormalization cycle 
by means of the diffeomorphic rescalings  built into the notion of renormalization. The distortion of these rescalings can be controlled if the two pieces under consideration, viewed from an appropriate scale, do not  lie ``too deep" (in the sense precisely defined below) .
 This machinery might have applications beyond the present situation.

For the reader's convenience, the pushing-up machinery will be informally outlined in \S \ref{out}. 
Also more special notations are collected in the Nomenclature.
For a survey on H\'enon renormalization see \cite{LM2}.  For early experiments and results on H\'enon renormalization
see \cite{CEK}, \cite{Cv}, and \cite{GST}.
 
\bigskip

{\bf Acknowledgment.} We thank all the institutions and foundations that have supported us in the course of this work: 
Simons Mathematics and Physics  Endowment,  Fields Institute, NSF, NSERC, University of Toronto.  In fall 2005, when M. Feigenbaum saw the negative results of \cite{CLM}, he made computer experiments that suggested that the universal scaling of the attractor is violated very rarely. 
Our paper provides a rigorous justification of Feigenbaum's experiments and conjectures.  We also thank C. Tresser for many valuable renormalization discussions, and R. Schul for interesting comments on \cite{J}.

\section{Outline}\label{out}

\subsection{Infinitely renormalizable H\'enon-like maps}

We will start with  outlining the set-up developed in \cite{CLM,LM1} --
see \S \ref{prelim} for details.

We consider a class $\HH=\HH(\bar\eps)$ of H\'enon-like maps of the form
$$
       F\colon (x, y) \mapsto (f(x)-\eps(x,y), x),
$$
acting on the unit box $B^0 =[0,1]\times [0,1]$,
where $f(x)$ is a unimodal map subject of certain regularity assumptions,
and $\|\eps\|< \bar\eps$ is small (for an appropriate norm).  
If the unimodal map $f$ is renormalizable
then the renormalization $F_1=RF\in \HH$ is defined as $(\Psi^1_0)^{-1} \circ (F^2|_{B^1}) \circ \Psi^1_0$, 
where $B^1$ is a certain box around the {\it tip}, a point which plays the role of the ``critical value'',
and $\Psi^1_0 : \Dom(F_1)\ra B^1$ is an explicit {\it non-linear} change of variables.

Inductively, we can define $n$ times renormalizable maps for any $n\in \N$,
and consequently, {\it infinitely renormalizable} H\'enon-like maps.
For such a map  the $n$-fold renormalization $F_n= R^n F\in \HH$
is obtained as $ (\Psi^n_0)^{-1} \circ (F^{2^n}|_{B^n}) \circ \Psi^n_0$,
where $B^n$ is an appropriate {\it renormalization box},  
$\Psi^n_0: \Dom(F_n)\ra B^n$ is a non-linear change of variables.

These boxes $B^n$ form a nest around the {\it tip} of $F$: 
$$
   B^0\supset B^1\supset\dots\supset B^n\supset\dots \ni \tau  
$$

 Taking the iterates $F^k B^n$, $k=0,1,\dots, 2^n-1$,
we obtain a family $\BB^n$ of $2^n$ pieces $\{B^n_\om\}$, called the {\it $n^{th}$ renormalization level},  that can be naturally 
labelled by strings $\om\in \{c,v\}^n$ in two symbols, $c$ and $v$,
with $B^n_{v^n}\equiv B^n$. See \S \ref{prelim} for details. Then 
$$
  \OO_F = \bigcap_n \bigcup_\om B^n_\om 
$$ 
is an attracting Cantor set on which $F$ acts as the adding
machine. This Cantor set carries a unique invariant measure $\mu$.
This allows us to introduce  
the most important geometric parameter attached to $F$, its {\it average Jacobian}
$$ 
b_F= \exp \int_{\OO_F} \log \Jac F \, d\mu.
$$
Usually, we will denote the average Jacobian with $b$. 

The size of the boxes decays exponentially:
\begin{equation}\label{sigma}
     \diam B^n_\om \leq C \si^n,
\end{equation}
where $\si\in (0,1)$ is the universal scaling factor
(coming from one-dimensional dynamics) while $C=C(\bar \eps)$ depends only on the geometry of  $F$.

A surprising phenomenon discovered in \cite{CLM} is that unlike its one-dimensional counterpart,
the Cantor set $\OO_F$ {\it does not have universal geometry}:  it essentially depends on the average Jacobian $b$.
However, the difference appears only in scale of order $b$: if all the pieces $ B^n_\om$ of level $n$ are much bigger than $ b$ 
then the geometry of the pieces $B^n_\om$ is controlled by one-dimensional dynamics:
the pieces  are aligned along the parabola $x=f(y)$ with thickness of order $b$.
According to (\ref{sigma}), this happens whenever 
\begin{equation}\label{safe scales for F}
  \alpha \si^n \geq  b
\end{equation}
 with sufficienty small (absolute) $\alpha>0$, i.e., when 
\begin{equation}\label{safe scales}
    n \leq  c |\log b| - A, \quad {\mathrm {where}}\ c = \frac 1{|\log\si|}, \ A= \frac {\log \alpha} {\log \si}.
\end{equation}
We will call these levels {\it safe}.

\subsection{Random walk model}\label{random walk}

\newcommand{\depth}{\operatorname{depth}}

To any point $x\in \OO\equiv \OO_F$ we can assign its {\it depth} 
$$
  \depth (x)\equiv  k(x)= \sup \{ k: \ x\in B^k\}\in \N\cup \{\infty\}. 
$$
Here the tip is the only point of infinite depth. If $\depth(x)=k$ then $x\in E^k\equiv B^{k}\setminus B^{k+1}$ (see Figure \ref{figE} and \ref{figEsch}).
\begin{figure}[htbp]
\begin{center}
\psfrag{E0}[c][c] [0.7] [0] {\Large $E^0$}
\psfrag{E1}[c][c] [0.7] [0] {\Large $E^1$}
\psfrag{E2}[c][c] [0.7] [0] {\Large $E^2$}
\psfrag{F}[c][c] [0.7] [0] {\Large $F$}
\psfrag{F2}[c][c] [0.7] [0] {\Large $F^2$}
\psfrag{F4}[c][c] [0.7] [0] {\Large $F^4$}
\psfrag{B1v}[c][c] [0.7] [0] {\Large $B^1$} 
\psfrag{B2vv}[c][c] [0.7] [0] {\Large $B^2$} 
\pichere{0.6}{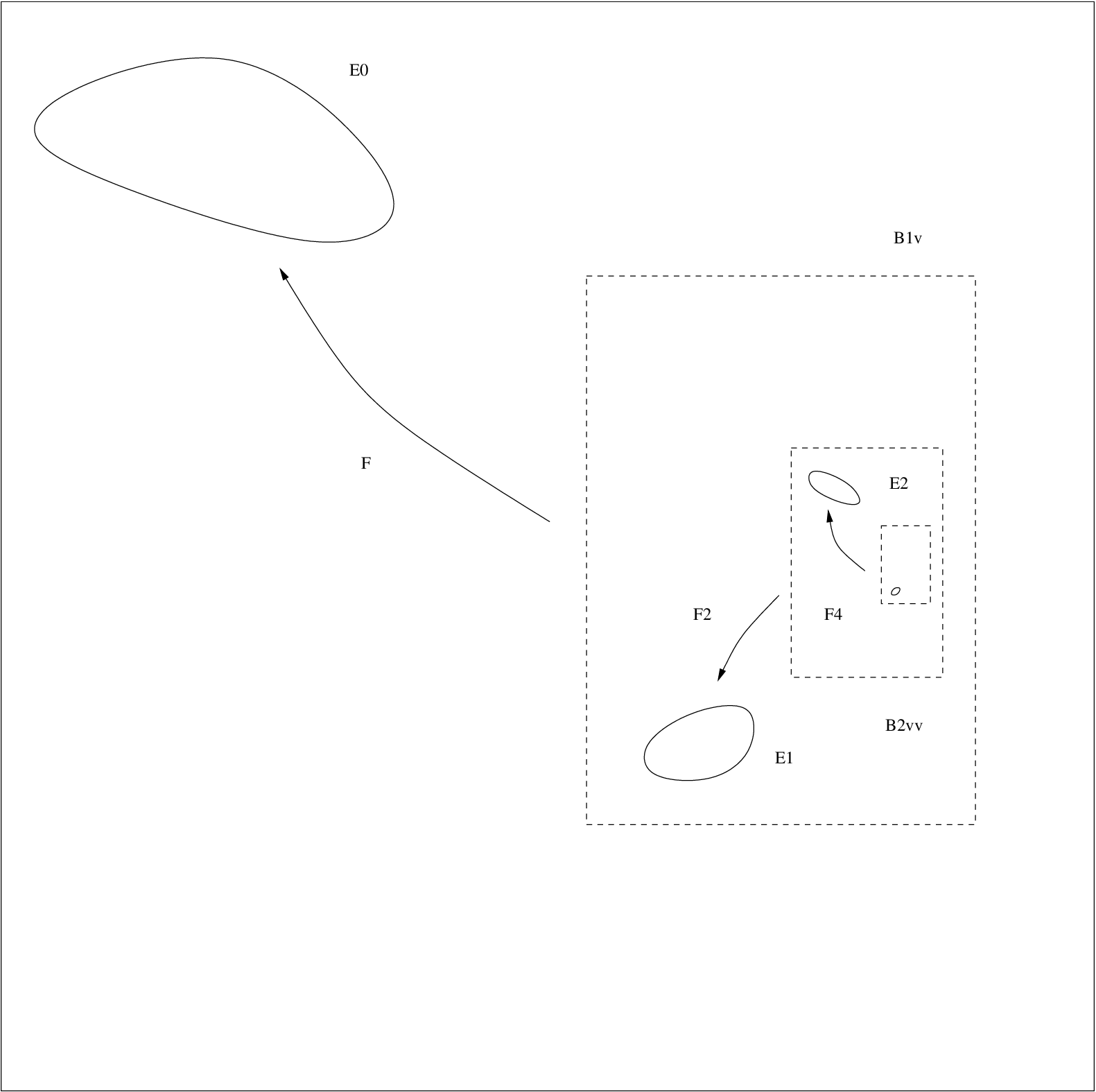}
\caption{} 
 \label{figE}
\end{center}
\end{figure}
We say that a point $x\in \OO$ is {\it combinatorially closer to $\tau$ than $y\in \OO$}
if $k(x)>k(y)$.  
We will now encode  any point $x\in \OO$ by its {\it closest approaches} to $\tau$
in  {\it backward} time.  Namely, let us consider the backward orbit $\{F^{-t} x\}_{t=0}^\infty$,
and mark the moments $t_m$ ($m=0,1,\dots$) of closest approaches, i.e., at the moment $t_m$ the point $x_m:= F^{-t_m} x$ is combinatorially
closer to $\tau$ than all previous points $F^{-t} x$, $t=0,1,\dots, t_m-1$. Since the dynamics of $F$ on $\OO$ is  the adding machine,
this is an infinite sequence of moments for any $x\not\in \orb (\tau)$.
If  $x=F^t(\tau)$, we terminate the code at the moment $t$.
 Let 
$$
    k_m(x) = k(x_m),\ m=0,1,\dots,
$$
be the sequence of the corresponding depths. 
Obviously, both sequences, $\bar t= \{t_m\}$ and $\bar k= \{k_m\}$ are {\it strictly increasing}.

For any depth $k$, let us consider the {\it first return map} (see Figure \ref{figE} and \ref{figEsch}).
$$ 
  G_k: B^{k+1} \ra B^k, \quad G_k= F^{2^k},
$$
and the {\it first landing map in backward time} 
$$
    L_k: \bigcup_{m=0}^{2^k-1} F^m (B^k) \ra B^k,\quad L_k(x) = F^{-m} x,\ \mbox{for} \ x\in  F^m (B^k).
$$  
Then we have by definition:
$$
                  x_m = G_{k_m(x)} (x_{m+1}),\quad  x_{m} = L_{k_m(x)} (x)  
$$ 
 
Let $\Sigma$ stand for the space of strictly increasing sequences $\bar k= \{k_m\}$ of symbols $k_m\in \N\cup \{\infty\}$
that terminate at moment $m$ if and only if $k_m=\infty$. Endow $\Sigma$ with a weak topology and the measure $\nu$ 
corresponding to the following {\it random walk} on $\N$: the probability of jumping from $k\in \N$ to $l\in \N$
is equal to $1/2^{l-k}$ if $l>k$, and it vanishes otherwise.  The initial distribution on $\N$ is given by
$\nu\{k\}= 1/2^{k+1}$. We let $j_m := k_{m+1}-k_m$ be the {\it jumps} in our random walk.

\begin{lem}
   The coding $x\mapsto \bar k(x)$ establishes a homeomorphism between $\OO$ and $\Sigma$  and a measure-theoretic 
isomorphism between $(\OO,\mu)$ and $(\Sigma,\nu)$.
\end{lem}

We can also consider the random walk that {\it stops on depth $n$}.
This means that we consider the orbit $F^{-t} x$ only until the moment it lands in $B^n$. 
The corresponding  (finite) coding sequence $\{\tl k_m \}_{m=0}^T$ is defined as follows:
$\tl k_m= k_m$ whenever $k_m < n$ ($m=0,1\dots, T-1$), while $\tl k_T =n$.   
(In what follows we will skip ``tilde'' in the notation as long as it would not lead to  confusion.) 

\comm{****
Fix a ``control sequence'' $j_n\in \N$, $n=0,1,\dots$.   
We say that a point $x\in \OO$  is {\it $\bar j$-controlled after a moment $N$}
if   $j_n(x) \leq j_n$ for all $n\geq N$. 
(We will say just ``controlled'' if it is clear which sequence $\bar j$ is meant,
or ``eventually controlled'' if we do not care of the precise moment $N$.)

\begin{cor}
  Let $j_n$ be a control sequence satisfying the following summability condition:
$$
   \sum_{n=0}^\infty \frac 1{2^{j_n}} < \infty. 
$$
Then  a.e. $x\in \OO$ is eventually $\bar j$-controlled. 

More precisely, the probability that a point $x\in \OO$ is 
not controlled after a moment $N$  is at most 
$$
   \eps_n:= \sum_{n=N}^\infty \frac 1{2^{j_n}}\to 0.
$$
\end{cor}

In particular, this Corollary is applicable to {\it linearly growing} jumps $j_n\geq a n$. 
In this case, the probability that  $x$ is not controlled after time $N$ is exponentially small,
namely, it is $O(1/2^{an})$.  
***}

Fix an increasing  {\it control function} $s: \N\ra \Z_+$.   
We say that a sequence $\bar k=\{k_m\}_{m=0}^\infty$  is {\it $s$-controlled after a moment $N$}
if   $j_m  \leq s(k_m)$ for all $k_m\geq N$. 
We say that a point $x\in \OO$  is {\it $s$-controlled after moment N} if its code 
$\bar k(x)$ is such. The set of these points is denoted by $K_N$.

\begin{lem}\label{control criterion}
Under the summability assumption
$$
  \sum_{k=0}^\infty \frac 1{2^{s(k)}} < \infty
$$
we have 
$$
  \nu(K_N)\ge 1-O(\sum_{k=N}^\infty \frac 1{2^{s(k)}}).
$$
\end{lem}

\comm{
\begin{lem}\label{control criterion}
 Almost all points $x\in \OO $ are eventually $s$-controlled
under the following summability assumption:
$$
  \sum_{k=0}^\infty \frac 1{2^{s(k)}} < \infty.
$$
\end{lem}
}

\begin{proof}
It follows immediately from the definition of the random walk, using the monotonicity of the control function, that
$$
\nu(K_N)\ge \prod_{k=N}^\infty (1-\frac 1{2^{s(k)}}),
$$
which implies the Lemma.
\end{proof}

\subsection{Geometric estimates}\label{Geom estimates outline}

Our analysis depends essentially on the geometric control of the renormalizations and changes of variables
established in \cite{CLM}. 

The renormalizations have the following nearly {\it universal shape}: 
\begin{equation}\label{universality formula}
    R^n F = (f_n(x) -\,  b^{2^n}\, a(x)\, y\, (1+ O(\rho^n)), \ x\, ),
\end{equation}
where the $f_n$ converge exponentially fast to the universal unimodal map $f_*$, 
 $a(x)$ is a universal  function, and $\rho\in (0,1)$.

The changes of variables $\Psi_k^l: \Dom(F^l)\ra \Dom(F^k)$ have the following form:
\begin{equation}\label{factoring-outline}
\Psi_k^l = D_k^l \circ (\id + {\bf S}_k^l),
\end{equation}
where 
\begin{equation}\label{reshufflingout}
 D_k^l= 
 \left(
\begin{array}{cc}
1 & t_k\\
0 & 1
\end{array}\right)
\left(
\begin{array}{cc}
(\si^2)^{l-k} & 0\\
0 & (-\si)^{l-k}
\end{array}\right) (1+O(\rho^k)). 
\end{equation}  
is a linear map with $t_k\asymp  b_F^{2^k}$,
while  $\id + {\bf S}_k^l: (x,y)\mapsto (x+S_k^l(x,y), y)$ is a horizontal non-linear map with
$$
   | \di_x S^l_k | = O(1),\quad  |\di_y S^l_k | = O(\bar{\eps}^{2^k}).
$$

\subsection{Regular boxes}
In this section we outline the results of \S \ref{pieces}.

For any $x\in \OO$, we let   $B_n(x)$ be the box $B^n_\om\in \BB^n$
containing $x$ (in particular, $B_n(\tau) = B^n$). 
Let $\BB^n_*= \BB^n\sm \{B^n\}$ 
stand for the family of boxes $B^n_\om$ that do not contain the tip.

Notice that the depth of all points $x$ in  any box   $B\in \BB^n_*$ 
is the same, so it can be assigned to the box itself.
In other words,
$$
  \depth (B) = \sup \{ k: \ B\subset  B^k\}\in   \{0,1,\dots n-1\}. 
$$
Let $\BB^n[l]$, $l<n$,  be the family of all boxes of level $n$ whose depth is $l$.
Note that $\BB^n[l]$ contains $2^{n-l-1}$ boxes. 

  We can view the box $B$  in the renormalization coordinates on  various scales. 
Namely, to view $B$ from scale $k\leq n$ means that we consider its preimage $\bB$
under the (nonlinear) rescaling $\Psi_0^k : \Dom(F_k)\ra B^k$.
The main scale from which $B$ will be viewed is its depth $k$,
so from now on $\bB := (\Psi_0^k)^{-1}(B)$ will stand for the corresponding box 
(see Figure \ref{figregpiece}).  This seemingly minor ingredient plays a crucial role in the estimates.

A box $B$ as above is called {\it regular} if the horizontal and vertical projections of $\bB$
are $K$-comparable, where $K>0$ is a universal constant, to be specified in the main body of the paper. 
In other words,  $\mod \bB$ (the ratio of the the vertical and horizontal sizes of $\bB$) is of order 1. 

%

We will control depth by the control function  
\begin{equation}\label{exp control f-n}
          s(k)= a 2^k -A\quad \mathrm{where}\ a= \frac{\log b}{\log \si},\ A=\frac{\log \alpha}{\log \si},
\end{equation}
with a sufficiently  small  universal $\alpha>0$ to be specified in the main body of the paper.
With this choice, we have:
\begin{equation}\label{safe scales for F_k}
  \alpha\si^{l-k} \geq b^{2^k}.
\end{equation}
Comparing it to (\ref{safe scales for F}) and (\ref{safe scales}), 
we see that the level $l-k$ controlled in this way is safe for the renormalization $F_k$. 

We say that the box $B\in \BB^n[l]$ is {\it  not too deep} in scale $B^k$ if
$$
 l-k \leq  s(k).  
$$

There are a number of constant which have to be chosen appropriately, for example $\alpha$ and $K$. 
In the main body of this paper it will be shown how to choose these constants carefully such that all Lemmas 
and Propositions hold. From now on we will assume in this outline that the constants are chosen appropriately and 
will not mention this matter any more. 

We will number the Lemmas and Propositions in this outline as the corresponding statements in the main body. However,
the version in the outline should be viewed as an informal version of the actual statements.

\bigskip

\noindent
{\bf Proposition 4.1.}
 {\it For all sufficiently big levels  $k$, the following is true. 
If a regular box  $B\in \BB^n_*$, $n>k$,  is not too deep in scale $B^k$ then $G_k(B)$ is regular. }
\bigskip

{\it Outline of the proof}. 
  Let $B\in \BB^n[l]$, $n>l>k$. We should view $B$  from scale $l$, 
i.e., consider the piece $\bB$ of level $n-l$ for the renormalization $F_l$,
 see  Figure \ref{mappiece}.  As the piece  $\tilde{B}=G_k(B)$ has depth $k$,
it should  be viewed from this depth. So, we consider the corresponding piece
   $\tilde{\bB}$ of level $n-k$ for the renormalization $F_k$. 
Then  $\tilde{\bB}=F_k\circ \Psi_k^l(\bB)$.      

Using geometric estimates for factorization (\ref{factoring-outline})  we show that 
$$ 
   \mod \Psi_k^l (\bB) \asymp \si^{l-k} \mod \bB,
$$
provided $\bB$ is regular. So $\Psi_k^l(\bB)$  is  highly stretched in the vertical direction.
The nearly  Universal map  $F_k$, see (\ref{universality formula}), will contract the vertical
size by a factor of order $b^{2^k}<<\si^{l-k}$ since the piece is not too deep.
This implies that the image under $F_k$ is essentially the 
image of the horizontal side. We obtain a piece $\tilde\bB$, which is essentially a curve,  that gets roughly aligned with the parabola,
which makes its modulus of order 1.   
\QED

\subsection{Universal sticks}

Given a box $B\in \BB^n[l]$ of a map $F$, 
let $\OO(B): =  \OO_F \cap B$ be the part of the postcritical set $\OO_F$ contained in $B$. 
Respectively, $\OO(\bB)= \OO_{F_l}\cap \bB$, where $\bB$ is the rescaled box corresponding to $B$.    

We say that a box $B\in \BB^n[l]$ is a $\de$-stick if the postcritical set 
$O(\bB)$ is contained in a diagonal strip $\Pi$ of 
 thickness $\de$, relatively the horizontal size of $\bB$. 
 The minimal thickness is denoted by $\de_{\bB}$. See Figure \ref{wid}. 
 
Let us consider the pieces $B_1$ and $B_2$ of level $n+1$ contained in $B$. The corresponding pieces $\bB_1$ and $\bB_2$ occupy fractions 
$\sigma_{\bB_1}$ and $\sigma_{\bB_2}$ of $\bB$, called {\it scaling ratios}, see Figure \ref{sca}. Let $\sigma^*_{\bB_1}$ and $\sigma^*_{\bB_2}$ be the scaling ratios of the corresponding 
pieces for the degenerate renormalization fixed point $F_*$. Let $\Delta \sigma_{\bB}$ be the maximal difference between the corresponding scaling ratios.

A piece $B\in \BB^n$ is called $\eps$-{\it universal}  if $\de_{\bB}\le \epsilon$ and   $\Delta\sigma_{\bB}\le \epsilon$.

\bigskip

Consider very deep pieces $B\in \BB^n[k]$, with $(1-q_0)\cdot n\le k\le n$, at scale $n-k$.  Then we are watching pieces of 
$\BB^{n-k}(F_{k})$ which can be obtained by following the orbit of $B^{n-k}_v(F_{k})$ for $2^{n-k}$ steps. $F_{k}$ is at a distance 
$O(\rho^{k})$ to the degenerate renormalization fixed point $F_*$. When $q_0>0$ is small, these few iterates, $2^{n-k}=2^{q_0\cdot n}$,  with a map $O(\rho^{(1-q_0)\cdot n})$ close to the renormalization fixed can be well approximated by iterates of the renormalization fixed point. At this scale, one-dimensional  dynamics is a good geometrical model.
We call this the {\it one-dimensional regime}.
   
\bigskip

\noindent
{\bf Proposition 7.2.} {\it There exist $\theta<1$, $0<q_0<q_1$ such that  every piece in $\BB^n[k]$, with
$(1-q_1)\cdot n\le k\le (1-q_0)\cdot n$, is   $O(\rho^{n})$-universal. 
}

\bigskip

We are going to refine Proposition 4.1,  in the sense that we are estimating how $\epsilon$-universality is distorted (or even improved!)
when we apply maps $G_k$ to regular pieces which are not too deep in scale $B^k$.  

\bigskip

\noindent
{\bf Proposition 5.1 and 6.1} 
{\it If $B\in\BB^n[l]$ is regular and not too deep in $B^k$
then 
$$
\delta_{\tilde{\bB}}\le \frac12 \cdot \delta_\bB+O(\sigma^{n-l}),
$$
and
$$
\Delta\sigma_{\tilde{\bB}}= \Delta\sigma_{\bB}+O(\delta_{\bB}+\sigma^{n-l}),
$$
where $\tilde{B}=G_k(B)\in \BB^n[k]$ and $\tilde{\bB}=F_k(\Psi^l_k(\bB))$.}

\bigskip

{\it Outline of the proof.}
We  consecutively estimate, using geometric estimates of \S \ref{Geom estimates outline},  
the relative thickness   of the pieces  $B_\diff=(\id+ {\bf S}^l_k)(\bB)$, 
$B_\aff=D^l_k(B_\diff)$ and $\tl\bB= F_k(B_\aff)$,  see Figure \ref{mapfac}.
The first one is comparable with  the thickness of $\bB$, up to an error of order $\si^{n-l}$,
since the horizontal map $\id + {\bf S}_k^l$ has bounded geometry
(where the error $\si^{n-l}\geq \diam \bB$ comes from the second order correction).

Let us now represent the affine map $D^l_k$ as a composition of the diagonal part  $\La$ and 
the the sheer part  $T$, see  (\ref{reshufflingout}).
The diagonal map $\La$ preserves the horizontal thickness,
so the thickness is only effected by the sheer part  $T$, which has order $t_k\asymp b^{2^k}$.
Using this estimate and that $B$ is not too deep in $B^k$, we show that $\de(B_\aff) =O (\de_{B_\diff})$.   
 
Finally, we show that the map $F_k$, being strongly vertically contracting,  improves thickness
again using that $B$ is not too deep in $B^k$. 

The maps $\Psi^l_k$ do not distort the scaling ratios at all as a consequence of the precise defintion of scaling ratios. The piece $\tl\bB$ is the image under $F_k$ of $B_\aff=\Psi^l_k(\bB)$. This map is exponentially close to the degenerate renormalization fixed point. It will not distort the scaling ratios too much.
\qed

\bigskip

Starting with pieces obtained during the one-dimensional regime, we apply repeatedly the maps $G_k$ as long as the new pieces are not too deep. 
This process is called the {\it pushing-up regime}.

The pieces created by the combined one-dimensional and pushing-up regimes are $O(\rho^n)$-universal.
This can be seen as follows.
Proposition 7.2, states that the pieces from the one-dimensional regime are exponentially universal.
These pieces are the starting pieces of the pushing-up regime. Propositions 5.1 and 6.1, state that the 
error in scaling ratios caused by pushing-up is of order of the sum of the ticknesses observed during the 
pushing-up process. Moreover, the thicknesses are essentially contracted each pushing-up step.  

Unfortunately, the pieces generated by the combination of the one-dimensional and pushing-up regimes, do not 
have a total measure which tends to $1$. In particular, Proposition 8.2 states that asymptotically, these pieces will be missing 
a fraction of the order $O(2^k(b^\gamma)^{2^k})$ of $B^k$, where $\gamma>0$. This is an immediate consequence 
of the fact that during the pushing-up regime we only pushed-up pieces which are not too deep.

The solution to this problem is to stop the pushing-up regime at the level $\kappa(n)\asymp \ln n$.
Then $B^{\kappa(n)}$ will be filled except for an exponential small fraction with $O(\rho^n)$-universal pieces. After level
$\kappa(n)$ we start the {\it brute-force} regime, push-up all pieces without considering whether they are too deep or not.
In other words, just apply the original map $F$ for $2^{\kappa(n)}$ steps. But under these iterates the $O(\rho^n)$-universal 
sticks get spoiled at most by factor $O(C^{\kappa(n)})= O(n^c)$ with some $c>0$. Hence,
they are $O(n^c \rho^n)$-universal sticks, and we still see $O(\theta^n)$-universality, for some $\theta<1$.

Denote the pieces in $\BB^n$ generated by combining these three regimes by $\PP_n$. These pieces are $\theta^n$-universal.

\subsection{Probabilistic  universality}

\comm{
We say that the geometry of $O$ is {\it probabilistically universal}
if for almost all points $x\in O$ the boxes $B_n(x)$ are 
asymptotically universal sticks. 

}
We say that the geometry of $O$ is {\it probabilistically universal} 
if there exists a $\theta\in (0,1)$ such that
the total measure of boxes $B\in \BB^n$ which are  $\theta^n$-universal sticks
is at least $1-O(\theta^n)$. 

\begin{thm}
The geometry of $O$ is probabilistically universal.
\end{thm}

\begin{proof}
Let $n\ge 1$. The pieces in $\PP_n$ are $\theta^n$-universal. Left is to estimate $\mu(\PP_n)$.

The one-dimensional regime deals with the pieces of $\BB^n$ in $B^k$ with $(1-q_1)\cdot n \leq k \leq (1-q_0)\cdot n$. They occupy a fraction 
$1-O(\frac1{2^{(q_1-q_0)\cdot n}})$ of the measure of $B^{(1-q_1)\cdot n}$.
Push them up until $B^0$ without restriction whether they are too deep or not. They will occupy $\BB^n$ except for 
an exponential small fraction. Let $R_n$ be the corresponding set of paths of the random walk. These are the paths which hit the interval
$[(1-q_1)\cdot n, \leq (1-q_0)\cdot n]$ at least once but are not necessarily $s$-controlled. So 
$$
\nu(R_n)=1-O(\frac1{2^{(q_1-q_0)\cdot n}}).
$$

Recall, the set $K_{\kappa(n)}$ consists of the paths which are $s$-controlled after depth $\kappa(n)\asymp \ln n$. Lemma \ref{control criterion} gives
$$
   \nu(K_{\kappa(n)})= 1- O(\sum_{k=\kappa(n)}^\infty \frac 1{2^{s(k)}}) = 1- O\left(\frac {1} {{2}^{a 2^{\kappa(n)}}}\right) =
             1-O(\rho^n )
$$
for some $\rho\in (0,1)$.

Observe, the set of paths corresponding to $\PP_n$ 
is 
$R_n\cap K_{\kappa(n)}$.
Hence,
$$
\mu(\PP_n)=\nu(R_n\cap K_{\kappa(n)})=1-O(\theta^n),
$$
for some $\theta\in (0,1)$.
\end{proof}

\section{Preliminaries}\label{prelim}

A complete discussion of the following definitions and statements can be found in part I and part II, see \cite{CLM}, \cite{LM1}, of this series on renormalization of H\'enon-like maps.

\bigskip

Let $\Omega^h, \Omega^v\subset \mathbb{C}$ be neighborhoods of $[-1,1]\subset \mathbb{R}$ and $\Omega=\Omega^h\times \Omega^v$.  The set
 $\HH_\Omega(\overline{\epsilon})$ consists of maps $F:[-1,1]^2\to [-1,1]^2$ of the following form.
$$
F(x,y)=(f(x)-\epsilon(x,y), x),
$$
where $f:[-1,1]\to [-1,1]$ is a unimodal map which admits a holomorphic extension to $\Omega^h$ and $\epsilon:[-1,1]^2\to \mathbb{R}$ admits  a holomorphic extension to $\Omega$ and finally, $|\epsilon|\le \overline{\epsilon}$. The critical point $c$ of $f$ is non degenerate, $D^2f(c)<0$. A map in  $\HH_\Omega(\overline{\epsilon})$ is called a {\it H\'enon-like} map. Observe that H\'enon-like maps map vertical lines to horizontal lines.

A unimodal map $f:[-1,1]\to [-1,1]$ with critical point $c\in [-1,1]$ is {\it renormalizable} if $f^2:[f^2(c),f^4(c)]\to [f^2(c),f^4(c)]$ is unimodal and 
$[f^2(c),f^4(c)]\cap f([f^2(c),f^4(c)])=\emptyset$. The renormalization of $f$ 
is the affine rescaling of $f^2|([f^2(c),f^4(c)]$, denoted by $Rf$. The domain of $Rf$ is again $[-1,1]$.  The renormalization operator $R$ has a unique fixed 
point $f_*:[-1,1]\to [-1,1]$.  The introduction of \cite{FMP} presents the history of renormalization of unimodal maps and describes the main results.

The {\it scaling factor} of this fixed point $f_*$ is
$$
\sigma=\frac{|[f_*^2(c),f_*^4(c)]|}{|[-1,1]|}.
$$

A H\'enon-like map is renormalizable if there exists a domain $D\subset [-1,1]^2$ such that
$F^2:D\to D$. The construction of the domain $D$ is inspired by renormalization of unimodal maps. In particular, it should be a topological construction.  However, for small $\overline{\epsilon}>0$ the actual domain $A\subset [-1,1]$, used to renormalize as was done in \cite{CLM}, has an analytical definition. The precise definition can be found in \S 3.5 of part I. If the renormalizable H\'enon-like maps is given by  $F(x,y)=(f(x)-\epsilon(x,y))$ then the 
domain $A\subset [-1,1]$, an essentially vertical strip, is bounded by two curves of the form 
$$
f(x)-\epsilon(x,y)=\text{Const}.
$$
These curves are graphs over the $y$-axis with a slope of the order 
$\overline{\epsilon}>0$. 
The domain $A$ satisfies similar combinatorial properties as the domain of renormalization of a unimodal map:
$$
F(A)\cap A=\emptyset,
$$
and
$$
F^2(A)\subset A.
$$
Unfortunately, the restriction $F^2|A$ is not a H\'enon-like map 
as it does not map vertical lines into horizontal lines. 
This is the reason why the coordinated change needed to define the renormalization of $F$ is not an affine map, 
but it rather has the following form. Let
$$
H(x,y)=(f(x))-\epsilon(x,y),y)
$$
and
$$
G=H\circ F^2\circ H^{-1}.
$$
The map $H$ preserves horizontal lines and it is designed in such a way
 that the map $G$ maps vertical lines into horizontal lines. 
Moreover, $G$ is well defined on a rectangle $U\times [-1, 1]$ of 
full height. Here $U\subset [-1,1]$ is an interval of length $2/|s|$ 
with $s<-1$. 
Let us  rescale the domain of $G$ by the $s$-dilation $\Lambda$, 
such that the rescaled domain is of the form $[-1,1]\times V$,
where $V\subset \mathbb{R}$ is an interval of length $2/|s|$. Define the renormalization of $F$ by
$$
RF= \Lambda\circ G\circ \Lambda^{-1}.
$$
Notice that $RF$ is well defined on the rectangle $[-1,1]\times V$.
The coordinate  change $\psi= H^{-1}\circ \La^{-1}$ maps this rectangle 
onto the
topological rectangle  $A$ of full height.

The set of $n$-times renormalizable maps is denoted by $\HH^n_\Omega(\overline{\epsilon})\subset \HH_\Omega(\overline{\epsilon})$. If 
$F\in \HH^n_\Omega(\overline{\epsilon})$ we 
use the notation
$$
F_n=R^nF.
$$
The set of infinitely renormalizable maps is denoted by
$$
\II_\Omega(\overline{\epsilon})=\bigcap_{n\ge 1} 
\HH^n_\Omega(\overline{\epsilon}).
$$

The renormalization operator acting on $\HH^1_\Omega(\overline{\epsilon})$,
$\overline{\epsilon}>0$ small enough, has a unique fixed point
 $F_*\in \II_\Omega(\overline{\epsilon})$. It is the degenerate map 
$$
F_*(x,y)=(f_*(x), x).
$$
This renormalization fixed point is hyperbolic and the stable manifold has codimension one. Moreover,
$$
W^s(F_*)=\II_\Omega(\overline{\epsilon}).
$$

If we want to emphasize that some set, say $A$, 
is associated with a certain map $F$ we 
use notation like $A(F)$.

The coordinate change which conjugates $F_k^2|A(F_k)$ to $F_{k+1}$ is
 denoted 
by
\begin{equation}\label{phi}
          \psi^{k+1}_v=(\Lambda_k \circ H_k)^{-1}: \Domain(F_{k+1})\to A(F_k). 
\end{equation}
Here $H_k$ is the non-affine part of the coordinate change used to define 
$R^{k+1}F$ and $\Lambda_k$ is the dilation by $s_k<-1$.
Now, for $k<n$, let
\begin{equation}\label{Phi}
\Psi^n_k=\psi^{k+1}_v\circ \psi^{k+2}_v\circ \cdots \circ \psi^{n}_v:
\Domain(F_{n})\to A_{n-k}(F_k),
\end{equation}
where
$$
A_k(F)=\Psi^k_0(\Domain(F_k))\cap B.
$$
Notice, that each $A_k\subset [-1,1]$ is of full height and 
$\Psi^k_0$ conjugates $R^kF$ to $F^{2^k}|A_k$.
Furthermore,
$
A_{k+1}\subset A_k
$.

\bigskip

The change of
coordinates conjugating the renormalization $RF$ to $F^2$ is denoted by
  $\psi_v^1 := H^{-1}\circ \La^{-1}$. To describe the attractor of an infinitely renormalizable H\'enon-like map we also need the map $\psi_c^1= F\circ
\psi^1_v$. The subscripts $v$ and $c$ indicate that these maps are
associated to the critical {\it value} and the {\it critical} point,
respectively.

Similarly, let $\psi^2_v$ and $\psi^2_c$ be the corresponding
changes of variable for $RF$, and let
$$
\psi^2_{vv}= \psi^1_v\circ \psi^2_v, \quad \psi^2_{cv}= \psi^1_c\circ
\psi^2_v, \quad \psi^2_{vc}=\psi^1_v\circ\psi^2_c,\quad \psi^2_{cc}=\psi^1_c\circ \psi^2_c.
$$
and, proceeding this way, for any $n\ge 0$, construct $2^{n}$ maps 
$$ \psi^n_w = \psi^1_{w_1}\circ\dots\circ \psi^n_{w_{n}}, \quad
w=(w_1, \dots, w_{n})\in\{v,c\}^{n}. $$ 
The notation $\psi^n_w(F)$ will also be used to emphasize the map
under consideration, and we will let $W=\{v,c\}$ and
$W^n=\{v,c\}^n$ be the $n$-fold Cartesian product. The following Lemma and its proof can be found in \cite[Lemma 5.1]{CLM}.

\begin{lem}\label{contracting}
Let $F\in \II^c_\Om(\bar\eps)$. There exist $C>0$ 
such that for  $w\in W^n$, 
$
\| D\psi^n_w\|\leq  C \sigma^n,
$
$n\ge 1$.
\end{lem}

Let  $F\in \II_\Omega(\overline{\epsilon})$ and consider the domains
$$
B^n_\omega=\Im \psi^n_\omega.
$$
The first return maps to the domains
$$
B^n_{v^n}=\Im \Psi^n_0=\Im \psi^n_{v^n}
$$
correspond to the different renormalizations.
Notice, 
$$
B^{n+1}_{v^{n+1}}\subset B^n_{v^n}.
$$
An infinitely renormalizable H\'enon-like  map has an invariant Cantor set:
$$
\OO_F=\bigcap_{n\ge 1} \bigcup_{i=0}^{2^n-1} F^i(B^n_{v^n})=
\bigcap_{n\ge 1} \bigcup_{\omega\in W^n} B^n_\omega.
$$ 
The dynamics on this Cantor set
 is conjugate to an adding machine. Its unique invariant measure is denoted 
by $\mu$. 
The {\it average  Jacobian}
$$
b_F=\exp\int \log \Jac F d\mu
$$ 
with respect to $\mu$ is an important parameter that
influences the geometry of $\OO_F$, see \cite[Theorem 10.1]{CLM}.

The critical point (and critical value) of a unimodal map plays a crucial role 
in its dynamics. The counterpart of the critical value for infinitely renormalizable H\'enon-like maps is the 
{\it tip}
$$
\{\tau_F\}=\bigcap_{n\ge 1} B^n_{v^n}.
$$

\subsection{One-dimensional maps}\label{1D universal f-s} 
Recall that  $f_*\colon [-1,1] \to [-1,1]$ stands for the one-dimensional
renormalization fixed point 
normalized so that $f_*(c_*)=1$ and $f_*^2(c_*)=-1$, where 
$c_*\in [-1,1]$ is the  critical point of $f_*$. 
The renormalization fixed point $f_*$ has a nested sequence of renormalization cycles $\CC_n$, $n\ge 1$. A cycle consists of the following intervals. The critical point of $f_*$ is $c_*$ and the critical value $v_*=f_*(c_*)$
$$
I^*_j(n)=[f^{j}_*(v_*), f^{j+2^n}_*(v_*)]\in \CC_n,
$$
with $j=0,1,2, \cdots, 2^n-1$. The collection $\CC_n$ consists of pairwise disjoint intervals. Notice, for  $j=0,1,2,\dots, 2^n-2$
$$
f_*(I^*_j(n))=I^*_{j+1}(n),
$$
and
$$
f_*(I^*_{2^n-1}(n))=I^*_0(n).
$$
The interval in $\CC_n$ which contains the critical point is denoted by
$$
U_n=I^*_{2^n-1}(n)\ni c_*.
$$

\bigskip

The {\it nonlinearity} of a $C^2$-diffeomorphism $\phi:I\to \phi(I)\subset \Bbb{R}$, $I\subset \Bbb{R}$, is
\begin{equation}\label{nonlin}
\eta_{\phi}=D\ln D\phi.
\end{equation}
The {\it Distortion} of a diffeomorphism $\phi:I\to J$ between intervals $I,J\subset \Bbb{R}$ is defined as
$$
\text{Dist}(\phi)=\max_{x,y}\log \frac{D\phi(y)}{D\phi(x)}.
$$
If $\eta$ is the nonlinearity of $\phi$ then
\begin{equation}\label{distnonl}
\text{Dist}(\phi)\le \|\eta\|_{C^0}\cdot |I|.
\end{equation}
The distortion of a map does not change if we rescale domain and range.

Given $r>0$. The $r$-neighborhood $T\supset I$ of an interval $I\subset \Bbb{R}$ is the interval such that both components of $T\setminus I$ have length $r|I|$. 

\begin{lem}\label{distortion}  There exist $r>0$ and $D>1$ such that the $r$-neighborhoods $T_j(n)\supset I^*_j(n)$ have the following property. For all $n\ge 1$ the following holds. Let $\zeta_j\in T_j(n)$ then
$$
\prod_{l=k_1}^{k_2-1}\frac{|Df_*(\zeta_j)|}{\frac{|I^*_{j+1}(n)|}{|I^*_j(n)|}}\le D.
$$
with $0\le k_1<k_2<2^n$.
\end{lem}

\begin{proof} The a priori bounds on the cycles $\CC_n$ are described in \cite{MS}, see also
\cite{CMMT}. The a priori bounds state that for some $r>0$ the gap between $I_j(n+1)$ and $I_{j+2^{n+1}}(n)$ satisfies
$$
|I_j(n)\setminus (I_j(n+1)\cup I_{j+2^{n+1}}(n))|\ge 5r\cdot  |I_j(n)|.
$$
Hence, we have  
$T_j(n+1)\cap T_{j+2^{n+1}}(n+1)=\emptyset$ and 
$$
|T_j(n+1)|+|T_{j+2^{n+1}}(n)|\le (1-r)\cdot  |T_j(n)|.
$$
Let $\eta_j(n)$ be the nonlinearity, see (\ref{nonlin}), of the rescaling of $f_*:I^*_j(n)\to I^*_{j+1}(n)$. The rescaling turns domain and range into $[-1,1]$.
Lemma 3.1 in \cite{Ma2} says that
$$
\begin{aligned}
\|\eta_j(n+1)\|_{C_0}&\le \frac{|T_j(n+1)|}{ |T_j(n)|}\cdot \|\eta_j(n)\|_{C_0}, \\
\|\eta_{j+2^{n+1}}(n+1)\|_{C^0}&\le \frac{|T_{j+2^{n+1}}(n+1)|}{ |T_j(n)|}\cdot \|\eta_j(n)\|_{C_0}.
\end{aligned}
$$
Hence, 
$$
\|\eta_j(n+1)\|_{C_0}+\|\eta_{j+2^{n+1}}(n+1)\|_{C^0}\le (1-r)\cdot \|\eta_j(n)\|_{C_0},
$$
for $j=0,1,2,\cdots, 2^{n}-2$.
The a priori bounds also imply a universal bound 
$$
\|\eta_{2^n-1}(n+1)\|_{C_0}\le K.
$$
Inductively, this gives a universal bound
$$
\sum_{j=0}^{2^n-2} \|\eta_{j}(n)\|_{C_0}\le K_0.
$$
Use (\ref{distnonl}) and observe,
$$
\log \frac{|Df_*(\zeta_j)|}{\frac{|I^*_{j+1}(n)|}{|I^*_j(n)|}}=O(\|\eta_j(n)\|_{C_0}).
$$
The Lemma follows.
\end{proof}

\begin{prop}\label{disto} There exists $\rho<1$ such that the following holds. Let  $0<q_0$ and 
$I\in \CC_n$ and $I\subset U_{k}\setminus U_{(1-q_0)\cdot n}$, with $k<(1-q_0)\cdot n$. Let $t_I=2^{k}$ be the first return to 
$U_{k}$. Then for every $j\le t_I$
$$
\text{Dist}(f_*^j|I)=O(\rho^{q_0\cdot n}).
$$
\end{prop}

\begin{proof} Let $s_I\ge t_I$ be the first return time of $I$ to $U_{(1-q_0)\cdot n}$. There exists $J_0\subset U_{k}$ with $I\subset J_0$ such that
$f_*^{s_I}: J_0 \to U_{(1-q_0)\cdot n}$ diffeomorphically, \cite {Ma1}. Let $J_k=f_*^k(J_0)$ and $I_k=f_*^k(I)$. The {\it a priori} bounds on the geometry of the cycles $\CC_n$ imply
$$
\frac{|I_k|}{|J_k|}=O(\rho^{q_0\cdot n}).
$$
This estimate hold because the intervals $J_k$ are in $\CC_{(1-q_0)\cdot n}$ and the intervals $I_k$ are in $\CC_n$.

The nonlinearity of the rescaled map $f_*:J_k\to J_{k+1}$ which has the unit interval as its domain and range, is denoted by $\eta_k$. As in the proof of Lemma \ref{distortion} we obtain
$$
\sum_{k=0}^{s_I-1} \|\eta_k\|_{C^0}\le K_0.
$$ 
The nonlinearity of the rescaled version of the map $f_*:I_k\to I_{k+1}$ which has the unit interval as its domain and range, is denoted by $\eta^I_k$. Lemma 3.1 in \cite{Ma2} says that the nonlinearity of the restriction $f_*:I_k\to I_{k+1}$ of  $f_*:J_k\to J_{k+1}$ satisfies
$$
\|\eta^I_k\|_{C^0}\le \frac{|I_k|}{|J_k|}\cdot \|\eta_k\|_{C^0}.
$$
Hence,
$$
\sum_{k=0}^{s_I-1} \|\eta^I_k\|_{C^0}=O(\rho^{q_0\cdot n}).
$$ 
The distortion of a map $f_*^t:I_k\to I_{k+t}$ is bounded as follows.
$$
\begin{aligned}
\text{Dist}(f^s_*|I_k)&\le \sum_{j=k}^{k+t-1} \text{Dist}(f_*|I_{j})\\
&\le \sum_{j=0}^{s_I-1} \|\eta^I_j\|_{C^0}=O(\rho^{q_0\cdot n}).
\end{aligned}
$$
This finishes the proof of the Lemma.
\end{proof}

\subsection{Geometrical properties of the Cantor attractor}

\begin{thm}[Universality]\label{universality}
For any $F\in \II_\Om(\bar\eps)$ with sufficiently small $\bar\eps$, we have: 
\[
    R^n F = (f_n(x) -\,  b^{2^n}\, a(x)\, y\, (1+ O(\rho^n)), \ x\, ),
\]
where $f_n\to f_*$ exponentially fast, $b$ is the average Jacobian, $\rho\in (0,1)$,
and $a(x)$ is a universal  function. Moreover, $a$ is analytic and 
positive.   
\end{thm}

\begin{cor}\label{bdsFx} There exists  a universal $d_1>0$ such that for $k\ge 1$ large enough \
$$
 \frac{1}{d_1}\le |\frac{\partial F_k}{\partial x}(z)|\le d_1,
$$
for every $z\in B^1_v(F_k)$.
\end{cor}

Let $\tau_n$ be the tip of $F_n=R^nF$ and $\tau^*$ be the tip of $F_*$.

\begin{lem}\label{holmt} There exists $\rho<1$ such the 
conjugations
$$
h_n:\OO_{F_*}\to \OO_{F_n}
$$
with $h_n(\tau_*)=\tau_{n}$ satisfy 
$$
|h_n(z)-z|=O(\rho^n),
$$
for every $z\in \OO_{F_*}$.
\end{lem}

\begin{proof} Choose $z^*\in \OO_{F_*}$ and let $z=h_n(z^*)$. There are unique sequence $w_{n+1},\dots, w_m, \dots$, and $z_n, z_{n+1},\dots, z_m,\dots$, and
$z^*_n,z^*_{n+1},\dots, z^*_m,\dots$ with $w_k\in \{c,v\}$, $z_k\in \OO_{F_k}$, and $z^*_k\in \OO_{F_*}$ such that $z=z_n$, $z^*=z^*_n$ and for $k\ge n$
$$ 
z_k=\psi^{k+1}_{w_{k+1}}(z_{k+1})   
$$
$$ 
z^*_k=(\psi^{k+1}_{w_{k+1}})^*(z^*_{k+1}).   
$$
This follows from the construction of $\OO_F$ in \cite{CLM}. 

Theorem \ref{universality} implies
$$
|\psi^{k+1}_w-(\psi^{k+1}_w)^*|=O(\rho^k)
$$
for some $\rho<1$. The proof of Lemma 5.6 in \cite{CLM} gives that
$(\psi^{k+1}_w)^*=\psi^*_w$ are contractions, $|D\psi^*_w|\le \sigma<1$. Then for $k\ge n$
$$
\begin{aligned}
|z_k-z^*_k|=&|\psi^{k+1}_{w_{k+1}}(z_{k+1})-(\psi^{k+1}_{w_{k+1}})^*(z^*_{k+1})|\\
\le &|\psi^{k+1}_{w_{k+1}}(z_{k+1})-(\psi^{k+1}_{w_{k+1}})^*(z_{k+1})|+\\
&|(\psi^{k+1}_{w_{k+1}})^*(z_{k+1})-(\psi^{k+1}_{w_{k+1}})^*(z^*_{k+1})|\\
\le &O(\rho^k)+\sigma\cdot |z_{k+1}-z^*_{k+1}|
\end{aligned}
$$
Then for every $m>n$ we have
$$
|z_n-z^*_n|\le \sum_{k=n+1}^m O(\rho^k)\cdot \sigma^{k-n-1}+\sigma^{m-n}\cdot |z_m-z^*_m|.
$$
Observe, $|z_m-z^*_m|\le 1$ and the Lemma follows by taking $m>n$ sufficiently large.
\end{proof}

\subsection{Analytical properties of the coordinate changes}

Fix an infinitely renormalizable H\'enon-like map $F\in \II_\Om(\bar\eps)$ 
to which we can apply the results of \cite{CLM} and \cite{LM1}, $\bar\eps>0$ is small enough. For such an $F$, we have a well defined {\it tip}: 
$$
\tau\equiv \tau(F)=\bigcap_{n\ge 0} B^n_{v^n}
$$
Consider the tips of the renormalizations, $\tau_k=\tau(R^k F)$.
To simplify the notations, we will translate these tips to the origin
by letting
$$
 \Psi_k =  \psi^0_v (R^k F)\,  (z + \tau_{k+1}) - \tau_k.  
$$ 
Denote the derivative of the maps $\Psi_k $ at $0$ by $D_k$
and decompose it into the unipotent and diagonal factors: 
 \begin{equation}\label{Dk}
 D_k= \left(
\begin{array}{cc}
1 & t_k\\
0 & 1
\end{array}\right)
\left(
\begin{array}{cc}
\alpha_k & 0\\
0 & \beta_k
\end{array}\right). 
\end{equation}
Let us factor this derivative out from $\Psi_k$: 
$$
\Psi_k = D_k \circ (\id + {\bf s}_k),
$$
where ${\bf s}_k(z) = (s_k(z) , 0) =  O(|z|^2)$ near 0.
Lemma 7.4 in \cite{CLM} states

\begin{lem} \label{smalls}
There exists $\rho<1$ such that for $k\in \Z_+$ the following estimates hold:
\begin{enumerate}


\item [\rm (1)] $\displaystyle 
  \alpha_k=\sigma^2 \cdot (1+O(\rho^k)),\quad \beta_k=-\sigma \cdot (1+O(\rho^k)), \quad
         t_k=O(\bar\eps^{2^k}); 
$

\item [\rm (2)]
$\displaystyle | \di_x s_k | = O(1),\quad  |\di_y s_k| = O(\bar\eps^{2^k}); $  

\item[\rm (3)]
$\displaystyle
  |\partial^2_{xx} s_k |= O(1),\quad
  |\partial^2_{xy} s_k |= O(\bar{\eps}^{2^k}) ,\quad 
  |\partial^2_{yy} s_k| =O(\bar{\eps}^{2^k}).
$
\end{enumerate}

\end{lem}

\begin{lem}\label{tilt} The numbers $t_k$ quantifying the tilt satisfy$$
t_k\asymp  -b_F^{2^k}.
$$
\end{lem}

We will use the following notation
$$
B^{n-k}_{v^{n-k}}(F_k)=\Im \Psi_k^n.
$$
\comment{
The following statements is Corollary 7.5 from \cite{CLM}

\begin{cor}\label{second derivatives} 
 Let $k<n$. For $z \in B_{v^{n-k}}^{n-k}$ we have: 
$$
\left| \partial_x s_k(z)\right|   =O (\sigma^{n-k}), \quad 
\left| \partial_y s_k(z)  \right| =
      O(\bar{\eps}^{2^k}\cdot\sigma^{n-k}).
$$
\end{cor}

\note{did we use this Cor?}
}
Consider the derivatives of the maps $\Psi^n_k$ at the origin:  
$$
D_k^n=D_k\circ D_{k+1}\circ \cdots D_{n-1}.
$$
We can reshuffle this composition and obtain:

 \begin{equation}\label{reshuffling}
 D_k^n= 
 \left(
\begin{array}{cc}
1 & t_k\\
0 & 1
\end{array}\right)
\left(
\begin{array}{cc}
(\si^2)^{n-k} & 0\\
0 & (-\si)^{n-k}
\end{array}\right) (1+O(\rho^k)). 
\end{equation}  
Factoring the derivatives  $D_k^n$ out from $\Psi_k^n$, we obtain: 
\begin{equation}\label{factoring}
\Psi_k^n = D_k^n \circ (\id + {\bf S}_k^n),
\end{equation}
where ${\bf S}_k^n (z) = (S_k^n(z), 0) = O(|z|^2)$ near 0. 

The following Lemma is Lemma 7.6 in \cite{CLM}

\begin{lem}\label{APPsi} For $k<n$,  we have:
\begin{enumerate}

\item [{\rm (1)}]
$ \displaystyle
   | \di_x S^n_k | = O(1),\quad  |\di_y S^n_k | = O(\bar{\eps}^{2^k}); 
$

\item [{\rm (2)}]
$ \displaystyle
|\partial^2_{xx} S^n_k| =O(1), \quad
|\partial^2_{yy} S^n_k| =O(\bar{\eps}^{2^k}), \quad
|\partial^2_{xy} S^n_k| =O(\bar{\eps}^{2^k}\, \si^{n-k}).
$
\end{enumerate}

\end{lem}

\begin{lem}\label{bdsS'} There exists  a universal $d_0>0$ such that for $k\ge 1$ large enough \
$$
 \frac{1}{d_0}\le |\frac{\partial(\id + {\bf S}_k^n)}{\partial x}|\le d_0
$$
\end{lem}

\begin{proof} According to proposition 7.8 in \cite{CLM}, the diffeomorphisms 
$\id + {\bf S}_k^n$ stay wihin a compact family of diffeomorphisms. This gives the upperbound on 
the derivative. The partial derivative $\frac{\partial(\id + {\bf S}_k^n)}{\partial x}$ can not be zero in a point because otherwise the derivative of the diffeomorphism  would become singular in point. This gives the positive lower bound on the partial derivative. 
\end{proof}

\begin{lem}\label{deltapsi} There exists $\rho<1$ such that
$$
|\Psi^n_k-(\Psi_k^n)^*|_{C^0}=O(\rho^k).
$$
\end{lem}

\begin{proof} The proof is a small modification of the proof of 
Lemma \ref{holmt}. Use the same notation: $w_l=v$ for all $l\ge k$.
We have to incorporate the translation which center the maps around the tips.
The estimates in the proof of Lemma \ref{holmt} become
$$
|\Psi^n_k(z)-(\Psi_k^n)^*(z)|
\le O(\rho^k)+\sum_{l=k+1}^n O(\rho^l)\cdot \sigma^{l-k-1}+\sigma^{n-k}\cdot |z_n-z^*_n|,
$$
where $z_n=z+\tau_n$ and $z^*_n=z+\tau^*$.  From Lemma \ref{holmt} we get that
$|z_n-z^*_n|=O(\rho^n)$ and the Lemma follows.
\end{proof}

\subsection{General Notions}

We will use the following general notions and notations throughout the text.

Let $\subset \Bbb{R}^2$ and $Q=[a,a+h]\times [b,b+v]$ be the smallest rectangle containing $X$. Then $h\ge 0$ is the horizontal size of $X$ and $v\ge 0$ the vertical size.

$Q_1\asymp Q_2$ means that $C^{-1}\le Q_1/Q_2\le C$, where $C>0$ is 
an absolute constant or depending on, say $F$. Similarly, we will use $Q_1\gtrsim Q_2$.

\section{Regular Pieces}\label{pieces}

By saying that something depends on the geometry of $F$,
we mean that it depends on the $C^2$-norm of $F$. 
Below, all the constants depend only on the geometry of $F$
unless otherwise is explicitly said.

The tip piece $B^k\equiv B^k_{v^k}$ of level $k\in \N $ contains two pieces of level $k+1$,
the tip one, $B^{k+1}$, and the lateral one,
$$
E^k  =B^{k+1}_{v^{k}c}.
$$
They are illustrated in Figure \ref{figE}, and more schematically, in Figure~\ref{figEsch}.
\begin{figure}[htbp]
\begin{center}
\psfrag{E0}[c][c] [0.7] [0] {\Large $E^0$}
\psfrag{E1}[c][c] [0.7] [0] {\Large $E^1$}
\psfrag{E2}[c][c] [0.7] [0] {\Large $E^2$}
\psfrag{F}[c][c] [0.7] [0] {\Large $F$}
\psfrag{F2}[c][c] [0.7] [0] {\Large $F^2$}
\psfrag{F4}[c][c] [0.7] [0] {\Large $F^4$}
\psfrag{t}[c][c] [0.7] [0] {\Large $\tau$} 
\pichere{0.6}{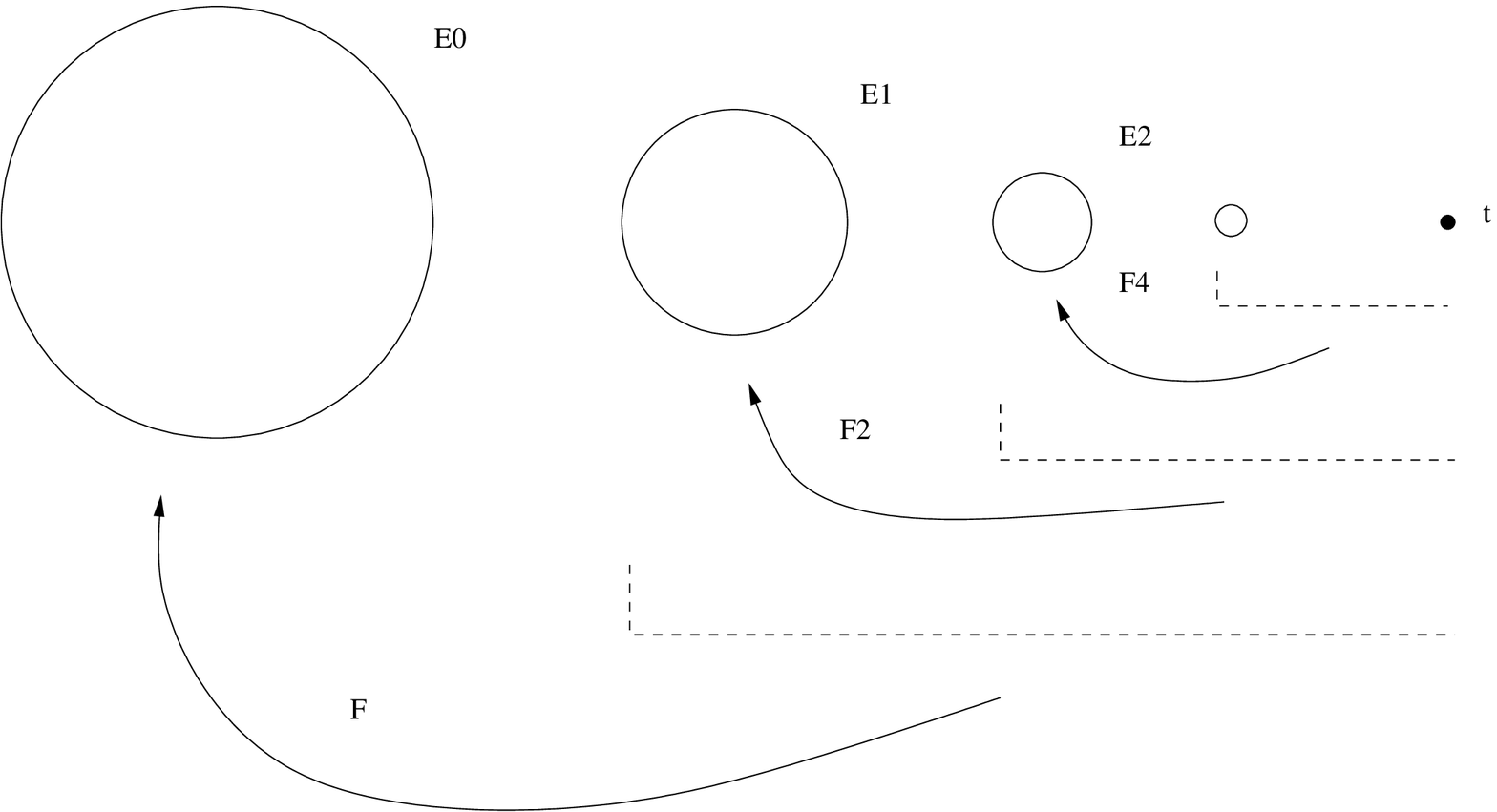} 
\caption{} \label{figEsch}
\end{center}
\end{figure}

For $n\ge k\ge 0$, let
$$
\BB^n[k]\equiv \BB^n(F)[k] =\{B\in \BB^n|\,  B\subset E^k\}.
$$  
We call $k$  the {\it depth} of any piece $B\in \BB^n[k]$. 
A piece $B^n_\omega$ belongs to $\BB^n[k]$ if and only if 
$$
\omega=v^{k}c\omega_{k+2}\dots\omega_n.
$$
Observe 
$$
\mu \left(\bigcup_{B\in \BB^n[k]} B\right) = \mu(E^k) = \frac{1}{2^{k+1}},
$$
where $\mu$ is the invariant measure on $\OO_F$.
%
Let 
$$
G_k = F^{2^k}: \bigcup_{l>k} E^l\to E^k,  \quad k\ge 0. 
$$
Given a piece $B\in \BB^n[k]$, there is a unique sequence 
$$
              k=k_0<k_1<\dots <k_t=n, \quad k_i = k_i(B) 
$$ 
such that
$$
B=G_{k_0}\circ G_{k_1}\circ\cdots \circ G_{k_{t-1}}\circ G_{k_t}(B^n).
$$
To see it, consider the backward orbit $\{F^{-m} B\}$ that brings $B$ to the tip piece  $B^n$.
Let $F^{-m_i}(B)$ be the moments of its closest combinatorial approaches to the tip,
in the sense of the nest $B^0\supset B^1\supset\dots$.  
Then $k_i$ is the depth of $F^{-m_i}(B)$.
Thus,  $F^{-m_i}(B)\in E^{k_i}$, while $F^{-m}(B)\cap B^{k_i}=\emptyset$ for all $m<m_i$, compare with \S \ref{random walk}.  
The pieces 
$$
B^{(i)}:= F^{-m_i}(B) = G_{k_i}\circ \cdots \circ G_{k_{t-1}}\circ G_{k_t}(B^n)\in \BB^n[k_i],
$$
with $i=1,2, \cdots, t$,
are called  {\it predecessors} of $B$.
Let us  view a  piece $B=B^n_{v^{k}c\omega_{k+2}\cdots \omega_n}\in \BB^n[k]$ 
from {\it scale} $k$, i.e., let us consider the following piece $\bB$ of depth $0$ for
the renormalization $F_k \equiv R^k F$: 
\begin{equation}\label{B'}
  \bB=B^{n-k}_{c\omega_{k+2}\cdots \omega_n}(F_k)\in \BB^{n-k}(F_k)[0], \quad \mathrm{so}\ B=\Psi_0^k (\bB),
\end{equation}
see Figure \ref{figregpiece}.

Below, a ``rectangle'' means a rectangle with horizontal and vertical sides.  
Given a piece $B = B^n_\om \in \BB^n$,  let us consider 
the smallest rectangle $Q=Q^n_\om$ containing  $B\cap \OO_F$.
We say that $Q=Q(B)$ is {\it associated with} $B$.

\begin{rem} We are primarily interested in the geometry of the Cantor attractor $\OO_F$.
For this reason we consider  rectangles $Q$ superscribed around $\OO_F\cap B$
rather than the ones superscribed around the actual pieces $B$.
However,  our results apply to the latter rectangles as well.
\end{rem}  

Given $B\in \BB^n[k]$,
let us consider the rectangle $\bQ$ associated to $\bB\in \BB^{n-k}(F_k)$.
Let  $\bh$ and $\bv$  be its  horizontal and  vertical sizes of $\bQ$ respectively.
We also call them the {\it sizes of $B$ viewed from scale $k$}. 
We say that the piece $B\in \BB^n[k]$ is {\it regular} if these sizes are comparable,
or, in other words, if $\mod \bB: = \bh/\bv$ is of order 1: 
\begin{equation}\label{h/v is bounded}
  \frac{1}{C_0}  \le  \mod \bB \le C_0,
\end{equation}
with $C_0=3d_1$, where $d_1>0$ is the bound on $\di F_k/ \di x$  from  
Corollary \ref{bdsFx}
(see~Figure~\ref{figregpiece}).
\begin{figure}[htbp]
\begin{center}
\psfrag{El}[c][c] [0.7] [0] {\Large $E^k$}
\psfrag{B}[c][c] [0.7] [0] {\Large $B$}
\psfrag{B'}[c][c] [0.7] [0] {\Large $\bB$}
\psfrag{F}[c][c] [0.7] [0] {\Large $F$}
\psfrag{Fl}[c][c] [0.7] [0] {\Large $F_k$}
\psfrag{Ps}[c][c] [0.7] [0] {\Large $\Psi^k_0$}
\psfrag{v}[c][c] [0.7] [0] {\Large $\bv$} 
\psfrag{h}[c][c] [0.7] [0] {\Large $\bh$} 
\pichere{0.9}{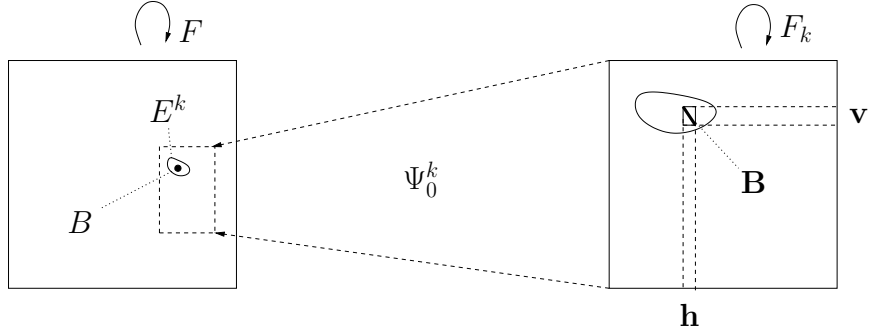} 
\caption{A regular piece} \label{figregpiece}
\end{center}
\end{figure}

Notice that in the degenerate case, $F(x,y)=(f(x),x)$, every piece is regular
since the slope of $f$ in $E^0$ is squeezed in between $d_1$ and $1/d_1$.

Next, we will specify an exponential control function $s(k)=s_\alpha(k)= a2^k-A$, see (\ref{exp control f-n}).
Namely, we let
$$   
 a=  \frac{\ln b}{\ln \sigma}, \quad A = \frac{\ln \alpha}{\ln \sigma},
$$
where $\alpha>0$ is a small parameter.
The actual choice of $\alpha=\alpha(\bar\eps) >0$ depending only on the geometry of $F$ 
will be made in the cause of the paper (see Propositions \ref{regtoreg}, 
\ref{sticktostick}, etc.).

Let $ l(k)= l_\alpha(k) =  s(k) +k.$
If $k\leq l \le l(k)$ we say that the pieces $B\in\BB^n[l]$ are {\it  not too deep in $B^k$}.
The choice of the control function is made so that 
\begin{equation}\label{blalpha}
      b^{2^k}\leq \alpha\, \sigma^{l-k} \quad \mbox{for $l\le l_\alpha(k)$}.
\end{equation}

\begin{prop}\label{regtoreg} There exists $k^*\ge 0$ and $\alpha^*>0$ 
such that for $\alpha<\alpha^*$ and $k\ge k^*$
 the following holds.
If $B\in\BB^n[l]$ is regular and not too deep in  $B^k$, 
$k<l\le l_\alpha(k)$,
then 
$$
     \tilde{B}=G_k(B)\in \BB^n[k]
$$
is regular as well.
\end{prop}

\begin{proof}
We should view $B$  from scale $l$, 
i.e., consider the piece $\bB \in \BB^{n-l}(F_l)[0]$ defined by (\ref{B'}).
As the puzzle piece  $\tilde{B}=F^{2^k}(B)$ has depth $k$,
it should  be viewed from this depth. So, we consider
\begin{equation}\label{tilde B'}
   \tilde{\bB}\in \BB^{n-k}(F_k)[0], \quad \mathrm{such \quad that}\ \Psi_0^k(\tilde{\bB})=\tilde{B}.
\end{equation}
Observe:
$\tilde{\bB}=F_k\circ \Psi_k^l(\bB)$ (see  Figure \ref{mappiece}).
\begin{figure}[htbp]
\begin{center}
\psfrag{Ek}[c][c] [0.7] [0] {\Large $E^k$}
\psfrag{B}[c][c] [0.7] [0] {\Large $B$}
\psfrag{B'}[c][c] [0.7] [0] {\Large $\bB$}
\psfrag{B~}[c][c] [0.7] [0] {\Large $\tilde{B}$}
\psfrag{B~'}[c][c] [0.7] [0] {\Large $\tilde{\bB}$}
\psfrag{F}[c][c] [0.7] [0] {\Large $F$}
\psfrag{Fl}[c][c] [0.7] [0] {\Large $F_l$}
\psfrag{Fk}[c][c] [0.7] [0] {\Large $F_k$}
\psfrag{F2k}[c][c] [0.7] [0] {\Large $F^{2^k}$}
\psfrag{Ps0k}[c][c] [0.7] [0] {\Large $\Psi^k_0$}
\psfrag{Pskl}[c][c] [0.7] [0] {\Large $\Psi^l_k$}
\psfrag{v}[c][c] [0.7] [0] {\Large $v$} 
\psfrag{h}[c][c] [0.7] [0] {\Large $h$} 
\pichere{1.0}{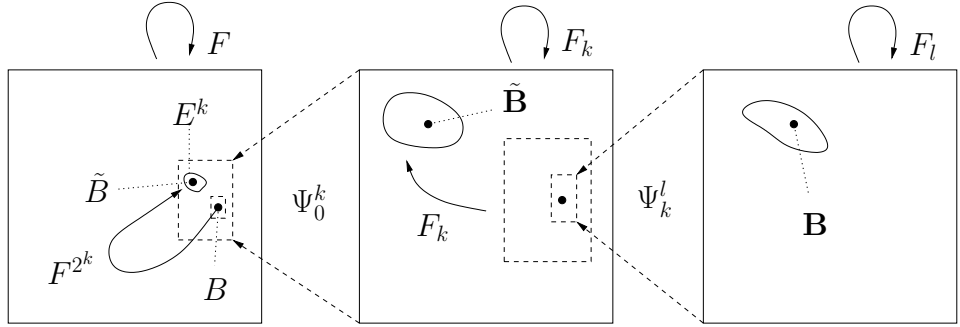} 
\caption{Pieces $B$ and $\tl B$ viewed from appropriate scales.} \label{mappiece}
\end{center}
\end{figure}

As above, let  $(\bh, \bv)$ be the sizes of $\bB$, 
and let $(\tilde{\bh}, \tilde{\bv})$ be the  sizes of $\tilde \bB$. 
Since $B$ is regular, bound (\ref{h/v is bounded}) hold for $\mod \bB = \bh/\bv$. 
We want to show that the same bound hold for  $\mod \tl \bB = \tl \bh/ \tl \bv$.

The map $\Psi_k^l$ factors into a non-linear and an affine part as described in \S \ref{prelim}:
$$
\Psi_k^l = D_k^l \circ (\id + {\bf S}_k^l).
$$
Figure \ref{mapfac} shows details  of this factorization 
for the map $\Psi^l_k$ from  Figure~\ref{mappiece}.
Let $h_{\text{diff}}$ and $v_{\text{diff}}$ be the sizes of the rectangle $Q_{\text{diff}}$ 
associated with the piece $(\id + {\bf S}_k^l)(\bB)$, see Figure \ref{mapfac}.  
Lemmas \ref{APPsi}(1) and \ref{bdsS'}
imply for $k$ big enough:
$$
  h_{\text{diff}} \leq d_0 \cdot \bh +O(\overline{\epsilon}^{2^k})\cdot \bv \leq 2 d_0\cdot \bh,
$$
where the last estimate takes into account (\ref{h/v is bounded}).
Similarly,
\begin{equation}\label{hdifflow}
   h_{\text{diff}}\geq  \frac 1{2d_0} \bh.
\end{equation}
Moreover, since the map $\id + {\bf S}_k^l$ is horizontal, we have
\begin{equation}\label{vdiffup}
v_{\text{diff}}= \bv\leq C_0\cdot \bh .
\end{equation}

\begin{figure}[htbp]
\begin{center}
\psfrag{B'}[c][c] [0.7] [0] {\Large $\bB$}
\psfrag{B~'}[c][c] [0.7] [0] {\Large $\tilde{\bB}$}
\psfrag{Fl}[c][c] [0.7] [0] {\Large $F_l$}
\psfrag{Fk}[c][c] [0.7] [0] {\Large $F_k$}
\psfrag{Fk}[c][c] [0.7] [0] {\Large $F_k$}
\psfrag{D}[c][c] [0.7] [0] {\Large $D^l_k$}
\psfrag{I}[c][c] [0.7] [0] {\Large $\id + {\bf S}_k^l$}
\psfrag{v}[c][c] [0.7] [0] {\Large $v$} 
\psfrag{h}[c][c] [0.7] [0] {\Large $h$}
\psfrag{va}[c][c] [0.7] [0] {\Large $v_{\text{aff}}$} 
\psfrag{ha}[c][c] [0.7] [0] {\Large $h_{\text{aff}}$} 
\psfrag{vd}[c][c] [0.7] [0] {\Large $v_{\text{diff}}$} 
\psfrag{hd}[c][c] [0.7] [0] {\Large $h_{\text{diff}}$}
\psfrag{v~}[c][c] [0.7] [0] {\Large $\tilde{\bv}$} 
\psfrag{h~}[c][c] [0.7] [0] {\Large $\tilde{\bh}$}
\psfrag{v}[c][c] [0.7] [0] {\Large $\bv$} 
\psfrag{h}[c][c] [0.7] [0] {\Large $\bh$}
\pichere{1.0}{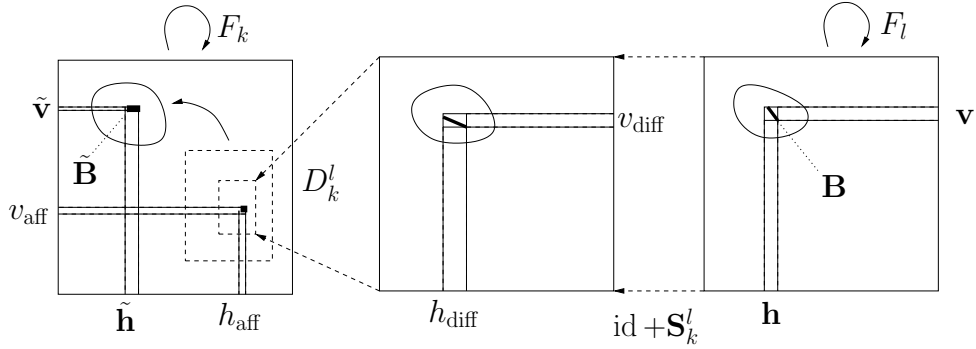} 
\caption{Factorization of the map $\Psi_k^l$ into horizontal and affine parts. } \label{mapfac}
\end{center}
\end{figure}

Let $h_{\text{aff}}$ and $v_{\text{aff}}$ be the sizes of the  rectangle $Q_{\text{aff}}$ 
associated with the piece $B_{\text{aff}}=\Psi_k^l(\bB)=D_k^l\circ(\id + {\bf S}_k^l)(\bB)$
(which is the piece $B$ viewed from scale $k$).  
Incorporating the above estimates into decomposition (\ref{reshuffling})
and using Lemma \ref{tilt}, we obtain for large $k$ (with $s=l-k$) :
$$
\begin{aligned}
h_{\text{aff}}&\le  (h_{\text{diff}} \cdot \sigma^{2s} +  v_{\text{diff}} \cdot  
    | t_k |\cdot \sigma^s)\cdot (1+O(\rho^k)) \\
& \le   [3d_0 \cdot \sigma^s+ O(b^{2^k})]\cdot \sigma^s \cdot \bh . \\
\end{aligned}
$$
Similarly, we obtain a lower bound for $h_{\text{aff}}$:
$$
h_{\text{aff}} \geq  [\frac 1 {3d_0} \cdot \sigma^s - O(b^{2^k})] \cdot \sigma^s \cdot \bh.
$$
If $B$ is not too deep for scale $k$, then $b^{2^k}\leq \alpha \, \si^s$, and we obtain:
\begin{equation}\label{h_aff}
      h_{\text{aff}}\asymp \sigma^{2s}\cdot \bh, 
\end{equation}
as long as $\alpha$ is small enough (depending on the geometry of $F_k$).

Bounds on $v_{\text{aff}}$ are obtained similarly (in fact, easier):
\begin{equation}\label{v_aff}
v_{\text{aff}} = v_{\text{diff}}\cdot \sigma^{l-k}\cdot (1+O(\rho^k)) = 
  \bv \cdot \sigma^s\cdot (1+O(\rho^k))\asymp \bv\cdot \sigma^s.
\end{equation}
%
%
Thus, 
\begin{equation}\label{mod B_aff}
      \mod B_{\text{aff}}=\mod \Psi_k^l(\bB) \asymp \si^s \mod \bB. 
\end{equation}
it gets roughly aligned with the parabola-like curve inside $E^k$,
which makes its modulus of order 1.   
Furthermore, Theorem \ref{universality} and Corollary \ref{bdsFx} imply, 
for $k$ large enough,  the following bounds on the sizes of  $\tilde{\bB}$: 
$$
\begin{aligned}
  \frac 1{2d_1} h_{\text{aff}} -A_0b^{2^k}\cdot v_{\text{aff}}
                   \leq   \tilde{\bh}&
  \le 2d_1\cdot h_{\text{aff}} +A_0b^{2^k}\cdot v_{\text{aff}},\\
               \tilde{\bv}&= h_{\text{aff}},
\end{aligned}
$$
where $A_0>0$ is an upper bound for $ a(x)\,  (1+ O(\rho^k))$
which controls the vertical derivative of $F_k$. 
Hence
$$
       \mod \tilde \bB \leq 2d_1 + \frac{A_0 b^{2^k}} {\mod  \Psi_k^l(\bB)}  
      \leq 2d_1 + \frac{A_0 b^{2^k}} {\si^s \mod \bB}\leq 2d_1+  A_0 C_0 \alpha\leq 3d_1,
$$
as long as $\alpha$ is small enough.

\begin{rem}
  Notice that $\mod \bB$ appears only in the residual term of the last estimate.
The main term ($2 d_1$) depends only on the geometry of $F$,
which makes the bound for $\mod \tl{\bB}$ as good  as that for $\bB$. 
\end{rem}

The lower estimate, $\mod\tilde \bB \geq (3d_1)^{-1}$, is similar. 
\end{proof}

\comm{***************

\begin{lem}\label{Abvh}  There exists $k^*\ge 0$ and $\alpha^*>0$ 

such that for $\alpha<\alpha^*$ and $k\ge k^*$ and $k<l\le l_k(\alpha)$ the following holds.
$$

\frac{1}{2d_1}-A_0b^{2^k}\frac{v_{\text{aff}}}{h_{\text{aff}}}\ge C_0

$$

and 

$$

A_0b^{2^k}\frac{v_{\text{aff}}}{h_{\text{aff}}}=O(\alpha).

$$

\end{lem}

\begin{proof}  For $k\ge 0$ large enough we have

$$

\begin{aligned}

A_0b^{2^k}\frac{v_{\text{aff}}}{h_{\text{aff}}}&\le 

2 A_0b^{2^k}\cdot \frac{v\cdot \sigma^{l-k}}{[[\frac{1}{d_0}  +O(\overline{\epsilon}^{2^k})]\cdot \sigma^{2(l-k)}+ O(\sigma^{l-k}\cdot b^{2^k})]\cdot h}\\

&\le 2 \frac{A_0}{C_0}b^{2^k}\cdot \frac{1}{[[\frac{1}{d_0}  +O(\overline{\epsilon}^{2^k})]\cdot \sigma^{(l-k)}+ O(b^{2^k})]}\\

&\le 2 \frac{A_0}{C_0}\cdot \frac{\frac{ b^{2^k}}{\sigma^{(l-k)}}}{[[\frac{1}{d_0}  +O(\overline{\epsilon}^{2^k})]+ O(\frac{ b^{2^k}}{\sigma^{(l-k)}})]}\\

&=O(\alpha),

\end{aligned}

$$

when 

$$

\frac{ b^{2^k}}{\sigma^{(l-k)}}\le \alpha

$$

is small enough. This holds when $l\le l_k$. The first inequality holds when $\alpha>0$ is small enough.

\end{proof}

This preparation allows us to prove that $\tilde{B}$ is also regular. Namely, for $k\ge 0$ large enough and $\alpha>0$ small enough, Lemma \ref{Abvh} gives

$$

\frac{\tilde{h}}{\tilde{v}}\ge \frac{1}{2d_1}-A_0b^{2^k}\frac{v_{\text{aff}}}{h_{\text{aff}}}

\ge C_0.

$$

Similarly, one shows  $\frac{\tilde{h}}{\tilde{v}}\le 3d_1=\frac{1}{C_0}$.

We finished the proof of Proposition \ref{regtoreg}.

**************}

\section{Sticks}\label{stic}

Let us consider a piece $B\in \BB^n[l]$ and the corresponding piece $\bB\in \BB^{n-l}(F_l)[0]$, 
see (\ref{B'}) and Figure \ref{figregpiece}. In the degenerate case, most   
of the pieces  $\bB\cap \OO_{F_l}$ get squeezed in a narrow strip  around the diagonal
of the associated rectangle $\bQ$. 
We will show that this is also the case for many pieces of H\'enon perturbations.
To this end, let us quantify the thickness of the pieces in question. 

Let us first introduce two  standard  strips of thickness $\de$:  
$$
\Delta_{\delta}^\pm =\{(x,y)\in [0,1]^2\, | \, \,|y \pm  x|\le \frac{\delta}{2}\}
$$
(oriented ``north-west'' and ``north-east'' respectively.)

Given a piece $B\in \BB^n$ and the associated rectangle $Q=Q(B)$, 
let $L: [0,1]^2\ra Q$ be the diagonal affine map. 
Let $\De(B)= L(\De_\de)^\pm$, where:

\ssk\nin
$\bullet$  we select the ``$+$''-sign if $B$ comes from the upper branch of the parabola $x=f(y)$, 
and ``$-$''-sign otherwise. 

\ssk\nin $\bullet$  
 $\de=\de_B$ is selected to be the smallest one such that  $\De(B)\supset B\cap \OO$. 

\ssk This $\de_B$ is called the {\it (relative) thickness} of $B$. 
The horizontal size  $h\de_B$ of $\De(B)$ is called the {\it absolute thickness} of $B$.
$\De(B)$ is called the associated diagonal strip.
We let $\BDe \equiv \Delta_\bB$ and call it the  {\it regular stick} associated with $B$, 
see Figure \ref{wid}.  

\begin{figure}[htbp]
\begin{center}
\psfrag{O}[c][c] [0.7] [0] {\Large $\OO_{F_l}\cap \bB$}
\psfrag{v}[c][c] [0.7] [0] {\Large $\Bv$}  
\psfrag{h}[c][c] [0.7] [0] {\Large $\Bh$}
\psfrag{hd}[c][c] [0.7] [0] {\Large $\Bh\delta_\bB$}
\pichere{0.8}{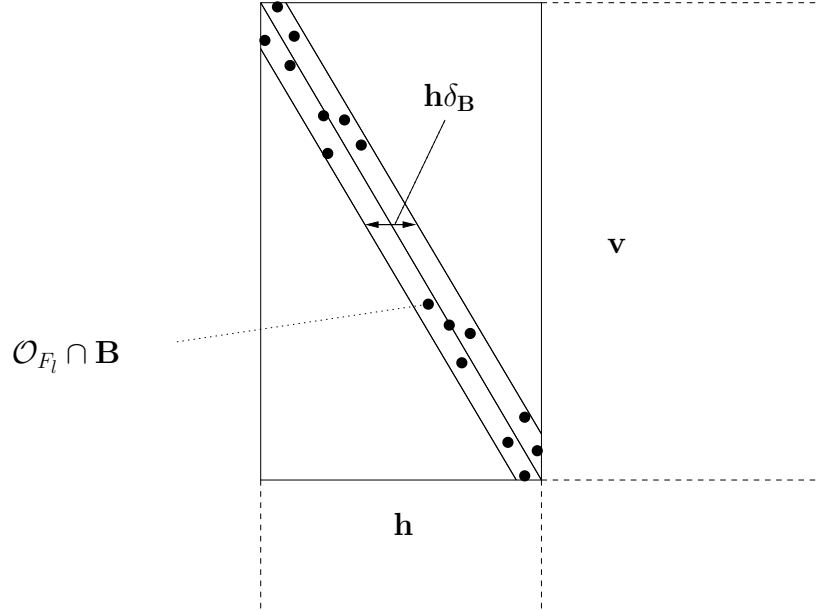} 
\caption{Regular stick} \label{wid}
\end{center}
\end{figure}

\begin{prop}\label{sticktostick}  
There exists $k^*\ge 0$ and $\alpha^*>0$ such that for $\alpha<\alpha^*$ and $k^*\leq k$ the following holds.
If $B\in\BB^n[l]$ is regular and not too deep in $B^k$, $k<l\le l_\alpha(k)$,
then 
$$
\delta_{\tilde{\bB}}\le \frac12 \cdot \delta_\bB+O(\sigma^{n-l}),
$$
where $\tilde{B}=G_k(B)\in \BB^n[k]$ and $\tilde{\bB}=F_k(\Psi^l_k(\bB))$.
\end{prop}

\begin{proof}
We will use the notation from \S \ref{pieces}.  
Let $\Bde\equiv \delta_\bB$,
and let  $\Bw=\Bde \cdot \Bh$ be the absolute thickness  of  $\bB$.  
The relative thickness of $\tilde{\bB}$ is denoted by  $\tilde{\Bde}\equiv \de_{\tilde\bB}$. 
To estimate $\tilde{\Bde}$, we will decompose $\Psi_k^l$ as in \S \ref{pieces}.  
Let $w_{\text{diff}}$ be the absolute thickness of  $B_\diff \equiv (\id + {\bf S}_k^l) (\bB)$. 
Then
\begin{equation}\label{wdiff} 
  w_{\text{diff}} =  O( \Bw +  \Bh\cdot\sigma^{n-l}).
\end{equation}
Indeed, let $\Gamma_y$ be the horizontal section of 
$(\id + {\bf S}_k^l)(\Delta_\bB)$ on height $y$, 
and let $\boldsymbol{\Gamma}_y= (\id + {\bf S}_k^l)^{-1} (\Gamma_y)$. 
Then 
$$
  |\Gamma_y|\leq    |\BGa_y|\cdot \|\id+{\bf S}_k^l\|_{C^1}  = O(\Bw), 
$$ 
where the last estimate follows from Lemma \ref{APPsi}(1). 

Furthermore, let us consider a  boundary curve of 
$(\id + {\bf S}_k^l)(\Delta_\bB)$. 
Its horizontal deviation from any of its tangent lines is bounded by
\begin{equation}\label{curvature}
   \frac 12  \|\id+{\bf S}_k^l\|_{C^2}\cdot (\diam \bB)^2   = O(\si^{n-l})\cdot \Bh , 
\end{equation}
where the last estimate  follows from  Lemma \ref{APPsi} (2), bounded modulus of $\bB$ (\ref{h/v is bounded}), 
and Lemma \ref{contracting}.
Hence 
$$
  w_\diff \leq \max_y|\Gamma_y| + O(\si^{n-l})\cdot \Bh,
$$
 and (\ref{wdiff}) follows.   
Together with (\ref{hdifflow}), it implies:
\begin{equation}\label{de_diff}
    \de_\diff = O(\Bde + \si^{n-l}). 
\end{equation}
\def\diag{{\mathrm{diag}}}
Let us now consider the piece 
$B_\aff\equiv \Psi^l_k(\bB)=D^l_k(B_\diff)$,  see Figure~\ref{mapfac}. 
Let $D^l_k=  T \circ \La$, where $\La=\La^l_k$ and $T=T^l_k$ are the diagonal and sheer parts of $D^l_k$ respectively, see  (\ref{reshuffling}).
Let us consider a box $B_\diag= \La(B_\diff)$,
 and let  $h_\diag=\si^{2(l-k)}h_\diff $ and  $v_\diag=\si^{l-k} v_\diff$ be its horizontal and vertical sizes.  
Since diagonal  affine maps preserve the horizontal thickness,
the thickness is only effected by the sheer part  $T^l_k$, which has order $t_k\asymp b^{2^k}$, 
see Lemma \ref{tilt}, namely: 
\begin{equation}\label{de_aff}
\begin{aligned}
    \de_\aff &\le \delta_{\diff}\cdot \frac{1}{1-\frac{v_{\diag}}{h_{\diag}} \cdot t_k}\\
    &=\delta_{\diff} \cdot \frac{1}{1-\frac{v_{\diff}}{h_{\diff}} \cdot \si^{-(l-k)} \cdot t_k}  \\ 
    &=O(\de_\diff) =  O(\Bde + \si^{n-l}). 
\end{aligned}
\end{equation}
where the passage to the last line comes from  (\ref{blalpha}), (\ref{hdifflow}), (\ref{vdiffup}) and Lemma \ref{tilt}.
Let us also consider the absolute {\it vertical thickness}  $u_{\text{aff}}$ of $B_\aff$,
i.e., the vertical size of the stick $\De(B_\aff)$.  From triangle similarity, we have:
$$
\frac{u_\aff} {v_\aff}=
\frac{w_\aff}{h_\aff} 
$$
So
\begin{equation}\label{u_aff}
u_\aff= \frac {w_\aff}{\mod B_\aff} \asymp  \frac {w_\aff}{\si^s \mod \bB} \asymp \si^{-s} \cdot w_\aff,\quad s=l-k,
\end{equation}
where the last estimate follows from regularity of $B$ while the previous one comes from (\ref{mod B_aff}).

We are now prepared to apply the map $F_k : (x,y)\mapsto (f_k(x)-\epsilon_k(x,y), x)$,
where $\|\eps_k\|_{C^2}= O(2^{b^k})$, see  Theorem \ref{universality}. 
Let $\tilde{\Bw}$ be the absolute thickness of $\tilde{\bB}$. 
By (\ref{h_aff})--(\ref{v_aff}),
the rectangle $Q_\aff$  associated with $B_{\text{aff}}$ has sizes
$$
 v_{\text{aff}}\asymp \sigma^{s}\bv \quad  \mathrm{and}\  
 h_{\text{aff}}\asymp \sigma^{2s}\bh.
$$                                   
Let us use  affine parametrization for the diagonal $Z$ of $B_\aff$:
$$
 x=x_0+t, \quad y=y_0+\frac{C}{\sigma^s}t,\quad 0\le t\le h_{\text{aff}},
$$
where  $x_0,y_0$ is its corner  where the stick $\De_\aff$ begins.
Restrict $F_k$ to this diagonal:
$$
F_k(x(t),y(t))=(A+Bt+E(t), x_0+t),
$$
where $E(t)$ is the second order deviation of $F_k(Z)$ from the straight line. 
We obtain:
$$
\begin{aligned}
E(t)&=O(\|\frac{\partial^2 (f_k-\epsilon_k)}{\partial x^2}\|+
\|\frac{\partial^2 \epsilon_k}{\partial xy}\cdot \sigma^{-s}\|+
\|\frac{\partial^2 \epsilon_k}{\partial y^2}\|\cdot (\sigma^{-s})^2))\cdot h_{\text{aff}}^2\\
&=O(h_{\text{aff}}+b^{2^k}\sigma^{-s}h_{\text{aff}}+
(b^{2^k}\sigma^{-s})\cdot (\sigma^{-s}h_{\text{aff}}))\cdot h_{\text{aff}}.
\end{aligned}
$$
From Lemma \ref{contracting} we have $h_{\text{aff}}=O(\sigma^{n-k})$. 
Hence,
$$
\begin{aligned}
E(t)&=O(\sigma^{n-k}+b^{2^k}\sigma^{-(l-k)+n-k}+
\alpha\cdot \sigma^{-(l-k)+n-k)})\cdot h_{\text{aff}}\\
&=O(\sigma^{n-l})\cdot h_{\text{aff}}
\end{aligned}
$$
where we used that $l$ is not too deep for $k$, i.e.  $b^{2^k}\sigma^{-s}\le \alpha$.
It follows that 
$$
\begin{aligned}
\tilde{\Bw} &= O(\sigma^{n-l}\cdot h_{\text{aff}}+b^{2^k}\cdot u_{\aff})\\
&= O(\sigma^{n-l}\cdot h_{\text{aff}}+b^{2^k}\sigma^{-s}\cdot w_{\aff}) \\
&= O(\sigma^{n-l}\cdot h_{\text{aff}}+ \alpha \cdot w_{\aff}) ,
\end{aligned}
$$
where we used (\ref{u_aff}).

\begin{rem}
This was the moment where the thickness improves.
\end{rem}

From the regularity of $\tilde{\bB}$ we get
$\tilde{\bh}\asymp \tilde{\bv}=h_{\text{aff}}$. Thus,
$$
\begin{aligned}
\tilde\Bde &=O(\sigma^{n-l}+\alpha\cdot \de_\aff )\\
&= O(\alpha\cdot \Bde + \si^{n-l})\\
\end{aligned}
$$
where we used (\ref{de_aff}). 
The Proposition follows
as long as $\alpha$ is sufficiently small.
\end{proof}

\section{Scaling}\label{scaling}

Let $B\in \BB^n[k]$ and $\hat{B}\in\BB^{n-1}[k]$ with $B\subset \hat{B}$. 
Say,
$$
B=B^n_{\omega\nu}\subset \hat{B}=B^{n-1}_{\omega}\subset E^k.
$$
Let $\bB$ and $\hat{\bB}$ be the corresponding rescaled pieces, 
so  $B=\Psi_0^k(\bB)$ and $\hat{B}=\Psi_0^k(\hat{\bB})$. 
The horizontal and vertical sizes of the associated rectangles are called $\Bh,\Bv>0$ and $\hat{\Bh}, \hat{\Bv}>0$ respectively.

The {\it scaling number} of $B$ is 
$$
\sigma_\bB=\frac{\Bv}{\hat{\Bv}}.
$$

\begin{figure}[htbp]
\begin{center}
\psfrag{OB}[c][c] [0.7] [0] {\Large $\OO_{F_k}\cap \bB$}
\psfrag{OhatB}[c][c] [0.7] [0] {\Large $\OO_{F_k}\cap \hat{\bB}$}
\psfrag{hatv}[c][c] [0.7] [0] {\Large $\hat{\Bv}$}
\psfrag{hath}[c][c] [0.7] [0] {\Large $\hat{\Bh}$}
\psfrag{v}[c][c] [0.7] [0] {\Large $\Bv=\sigma_\bB \hat{\Bv}$} 
\psfrag{h}[c][c] [0.7] [0] {\Large $\Bh$} 
\psfrag{hd}[c][c] [0.7] [0] {\Large $\delta_{\hat{\bB}}\hat{\Bh} $}
\pichere{0.8}{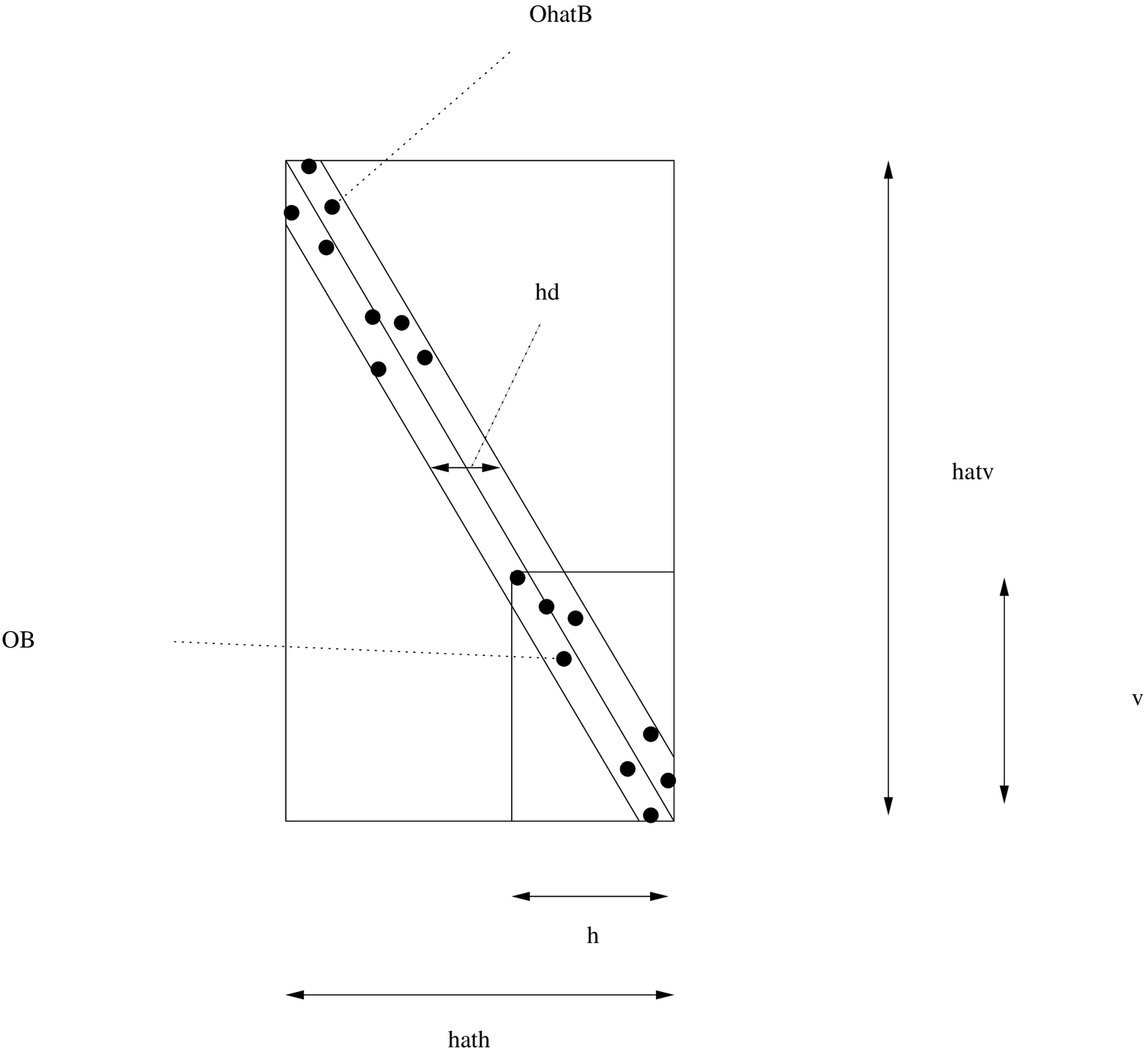} 
\caption{} \label{sca}
\end{center}
\end{figure}

\begin{rem}\label{prosca}
The scaling number can be expressed directly in terms of the original pieces $B$ and $\hat{B}$. 
Indeed, since the diffeomorphism $\Psi^k_0$ is a horizontal map, 
we have $ \sigma_\bB= v/\hat v$, where $v$ and $\hat{v}$ are the vertical sizes of $B$ and $\hat{B}$.
We will use the notation $\sigma_B$ when we refer to the corresponding measurement in the domain of $F$. 
This formal distinction will only play a role in (\ref{deltasigmapro}).
\end{rem}

\begin{rem}
There are many  possible ways to define the scaling number. 
The proof of the probabilistic universality will show that the  relative thickness of most pieces  asymptotically vanishes. 
Because of this, most definitions of the scaling  number become equivalent. 
\end{rem}

For $B= B^n_{\omega\nu}(F)$ as above, 
let $B^*= B^n_{\omega\nu}(F_*)$  be the corresponding degenerate
piece for the renormalization fixed point $F_*$.
The {\it proper scaling} for $B$ is
$$
\sigma^*_{\bB}=\sigma_{B^n_{\omega\nu}(F_*)}.
$$
 The function
$$
\underline{\sigma}: B\mapsto \sigma_\bB
$$
is called the {\it scaling function} of $F$. The universal scaling function 
$\underline{\sigma}^*$ of $F_*$  is injective, as was shown in \cite{BMT}.

\begin{rem}\label{1Dscalingstar} Given a piece $B\in \BB^{n+1}(F^*)$. Let $\hat{B}^*\in \BB^n(F_*)$ which contains $B$. 
For some $\hat{i}<2^n$ we have
$$
\pi_1(\hat{B}^*)=I^*_{\hat{i}}(n)\in \CC_n.
$$
Similarly, $\pi_1(B^*)=I^*_{i}(n+1)\in \CC_{n+1}$, for $i=\hat{i}$ or $i=\hat{i}+2^n$. The scaling ratios $\sigma_\bB$ are 
vertical measurements of pieces. Using that H\'enon like maps take vertical lines to horizontal lines, $y'=x$, we have
$$
\sigma^*_{\bB}=\frac{|I^*_{i-1}(n+1)|}{|I^*_{\hat{i}-1}(n)|}.
$$
\end{rem}

\begin{prop}\label{scapush}
 There exists $k^*\ge 0$ and $\alpha^*>0$ such that for $\alpha<\alpha^*$ and $k\ge k^*$ the following holds.
If a piece $B\in\BB^n[l]$ is regular and  not too deep for $E_k$, i.e. $k<l\le l_\alpha(k)$, 
then 
$$
\sigma_{\tilde{\bB}}= \sigma_\bB + O(\delta_{\hat{\bB}}+\sigma^{n-l}),
$$
where $\tilde{B}=G_k(B)\in \BB^n[k]$ and $B\subset \hat{B}=\Psi^l_0(\hat{\bB})\in \BB^{n-1}[l]$.
\end{prop}

\begin{proof} As above in \S \ref{pieces}, let $h_{\text{aff}}$ stand for the horizontal length of $B_\aff =\Psi_k^l(\bB)$, see Figure \ref{mapfac}. 
We will use the similar notation  $\hat{h}_{\text{aff}}$ and $\hat{w}_{\text{aff}}$ for the corresponding measurements of  
the piece $\hat B_\aff : = \Psi_k^l(\hat{\bB})$.
\begin{figure}[htbp]
\begin{center}
\psfrag{O}[c][c] [0.7] [0] {\Large $\OO_{F_k}\cap \hat{B}_{\text{aff}}$}
\psfrag{hatvaff}[c][c] [0.7] [0]
 {\Large $\hat{v}_{\text{aff}}$}
\psfrag{hathaff}[c][c] [0.7] [0] {\Large $\hat{h}_{\text{aff}}$}
\psfrag{va}[c][c] [0.7] [0] 
{\Large $v_{\text{aff}}=\sigma_\bB\hat{v}_{\text{aff}}$} 
\psfrag{ha}[c][c] [0.7] [0] {\Large $h_{\text{aff}}$} 
\psfrag{sighathaff}[c][c] [0.7] [0] {\Large $\sigma_{\bB}\hat{h}_{\text{aff}}$}
\psfrag{Dsig}[c][c] [0.7] [0]
 {\Large $|\sigma_{\tilde{\bB}}-\sigma_\bB|\hat{h}_{\text{aff}} $}
\psfrag{w}[c][c] [0.7] [0] {\Large $\hat{w}_{\text{aff}} $}
\pichere{0.85}{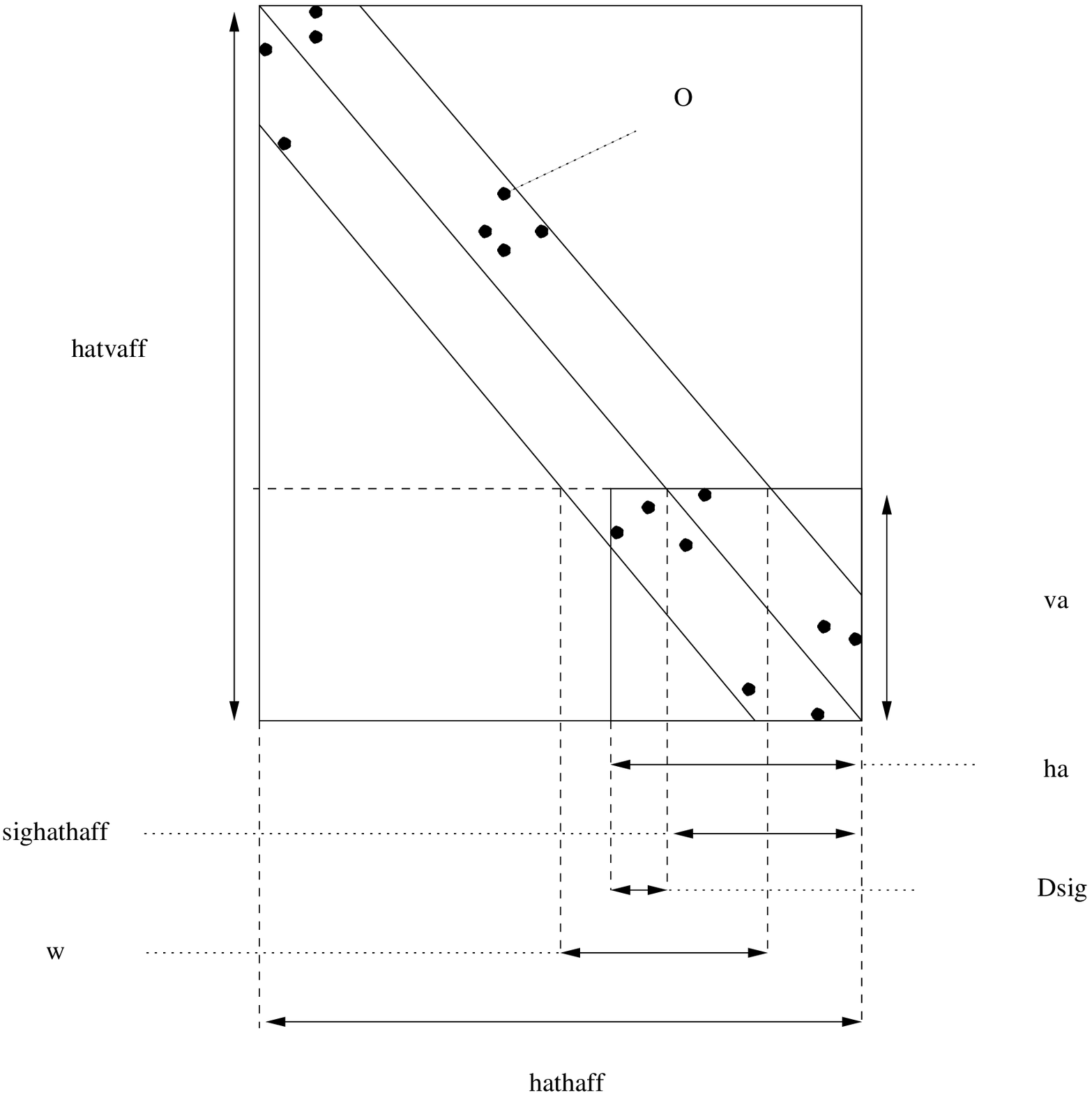} 
\caption{} \label{figsca}
\end{center}
\end{figure}

Since $F_k$ maps vertical lines to horizontal lines, we have
$$
\sigma_{\tilde{\bB}}=\frac{h_{\text{aff}}}{\hat{h}_{\text{aff}}}.
$$
Let $\gamma$ be the angle between the diagonal of $\hat B_\aff$ and the vertical side, so $\tg\gamma = \mod \hat B_\aff$.
Then
$$
      v_\aff \cdot \tg \gamma =  \hat h_\aff\,  \frac{v_\aff}{\hat v_\aff} =  \hat h_\aff\cdot \si_\bB,
$$
Now Figure \ref{figsca} shows:
$$
   |h_\aff - v_\aff \cdot \tg \gamma|\leq  \hat w_\aff. 
$$
Dividing by $\hat h_\aff$ (taking into account the two previous formulas and definition of the  relative thickness
$\hat \de_\aff= \hat w_\aff/\hat h_\aff$), we obtain:
$$
   | \sigma_{\tilde{\bB}} - \sigma_\bB| \leq   \hat{\de}_{\text{aff}}.
$$
Now the Proposition follows from (\ref{de_aff}).
\end{proof}

\section{Universal Sticks}\label{pres}

\subsection{Definition and statement} \label{defstat}

Let us consider a piece $B\in \BB^n$ and the two pieces $B_1, B_2\in \BB^{n+1}$ of level $n+1$  contained in $B$.
Rotate it  to make it horizontal  and then rescale it to horizontal size 1; denote the corresponding linear conformal map by $A$.  
Let  $\delta, \sigma_{B_1}, \sigma_{B_2}\ge 0$ be the smallest numbers such that:
\begin{itemize}
\item[(1)] $A(B\cap \OO_F)\subset [0,1]\times [0,\delta]$,
\item[(2)]$A(B_1\cap \OO_F)\subset [0,\sigma_{B_1}]\times [0,\delta]$,
\item[(3)]$A(B_2\cap \OO_F)\subset [1-\sigma_{B_2},1]\times [0,\delta]$,
\end{itemize}
for the optimal choice of $A$. The numbers $\sigma_{B_1}$, and  $\sigma_{B_2}$ are called {\it scaling factors} of $B_1$ and $B_2$. 

\begin{rem}\label{scaprec} The scaling factor $\sigma_\bB$ of a piece $B$ is a measurement of the corresponding $\bB$. 
The scaling factor $\sigma_B$ of $B$ reveres to measurements of the actual piece in the domain of $F$. The difference between the scaling factors $\sigma_B$ and $\sigma_\bB$ is estimated in Proposition \ref{precision}.
\end{rem}

We say that $B$ is {\it $\epsilon$-universal}  if 
$$
|\sigma_{B_1}-\sigma^*_{\bB_1}|\le \epsilon, \quad
|\sigma_{B_2}-\sigma^*_{\bB_2}|\le \epsilon, \quad
\text{and} \quad
\delta\le \epsilon.
$$
The {\it precision } of the piece $B$ is the smallest $\epsilon>0$ for which $B$ is $\epsilon$-universal.
The optimal $A^{-1}([0,1]\times[0,\delta])$ is called the {\it $\epsilon$-stick} for $B$.  We will revere to the {\it (relative) length} and {\it (relative) height} of such a stick.
Let $\SS^n(\epsilon)\subset \BB^n$ be the collection of $\epsilon$-universal pieces.

\begin{figure}[htbp]
\begin{center}
\psfrag{l}[c][c] [0.7] [0] {\Large $l$}
\psfrag{dl}[c][c] [0.7] [0] {\Large $\delta l$}
\psfrag{s1l}[c][c] [0.7] [0] {\Large $\sigma_1 l$}
\psfrag{s2l}[c][c] [0.7] [0] {\Large $\sigma_2 l$}
\psfrag{s1*}[c][c] [0.7] [0] {\Large $\sigma^*_{\bB_1} l$}
\psfrag{S2*}[c][c] [0.7] [0] {\Large $\sigma^*_{\bB_2} l$}
\psfrag{O}[c][c] [0.7] [0] {\Large $\OO_{F}\cap B_2$}
\pichere{0.7}{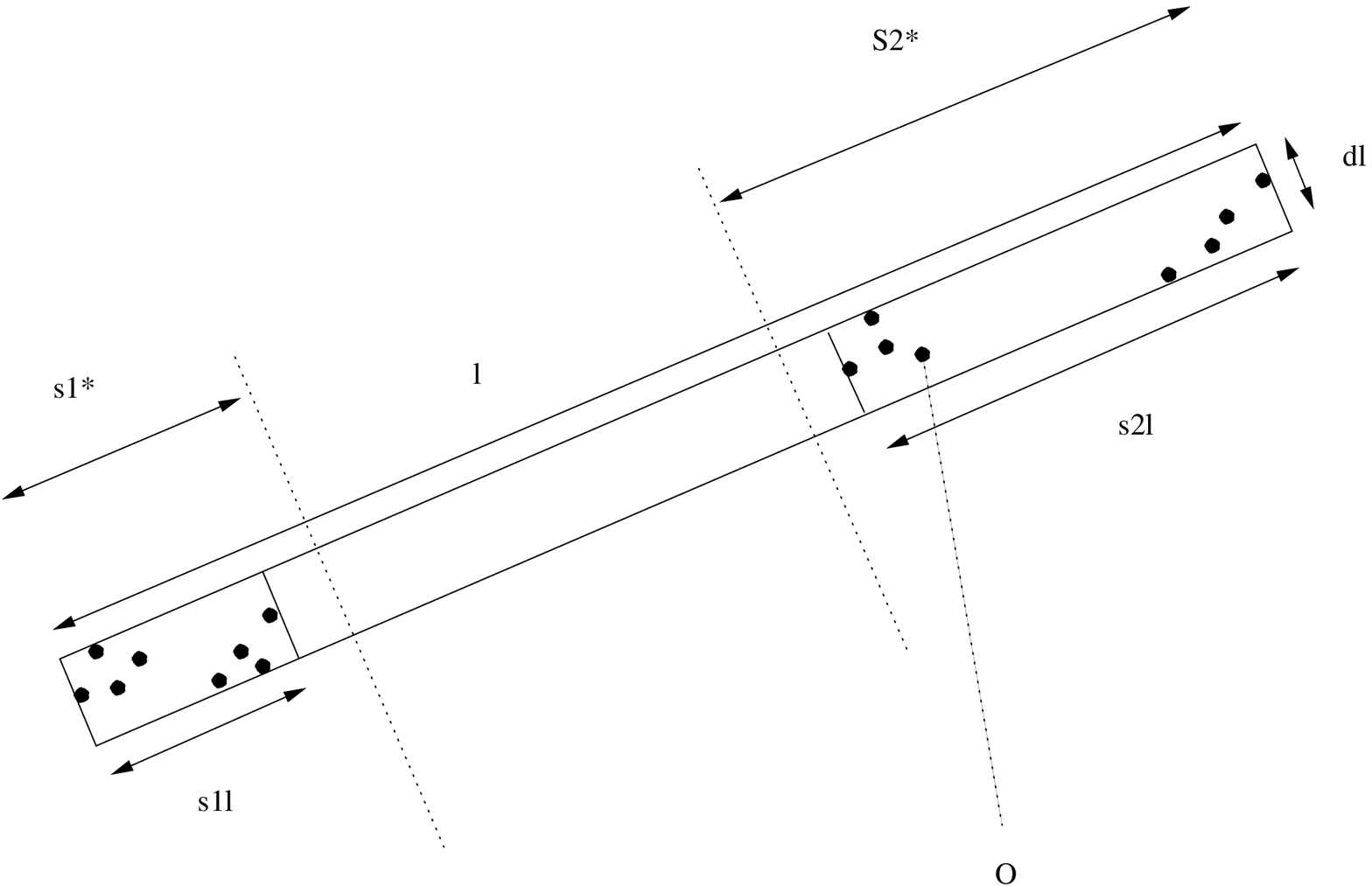} 
\caption{} \label{stick}
\end{center}
\end{figure}

\begin{defn}\label{defunidist} 
The Cantor attractor $\OO_F$ of an infinitely renormalizable H\'enon-like map $F\in \HH_\Omega(\overline{\epsilon})$ 
is {\it probabilistically universal}  if  there is $\theta<1$ such that 
$$
\mu(\SS^n(\theta^n))\ge 1-\theta^n,\quad n\ge 1.
$$
\end{defn}

Now we can formulate the main result of this paper: 

\begin{thm}\label{distuni} 
The Cantor attractor $\OO_F$ is probabilistically universal.
\end{thm}

After careful choices of $\theta<1$, $q_0<q_1$ and $\kappa(n)=-\text{Const}+ \ln n$, one distinguishes three regimes where pieces in $\SS_n(\theta^n)\cap E^k$ are discovered by different techniques. 

The {\it one-dimensional regime}: all the pieces in $\BB^n[k]$ with $(1-q_1)\cdot n\le k\le (1-q_0)\cdot n$ are in $\SS_n(\theta^n)$. These very deep pieces are controlled by the one-dimensional renormalization fixed point: they are perturbed versions the corresponding pieces of $F_*$ and their relative displacements are exponentially small, see Lemma \ref{displacement} and Proposition \ref{asympwidth}. We have to exclude the pieces in $\BB^n[k]$ with $k>(1-q_0)\cdot n$ because they do not have a small thickness. Viewed from their scale $k$, they are relatively large pieces close to the graph of $f_*$. The curvature of the graph of $f_*$ causes this pieces to have a large thickness. 

The {\it pushing-up regime}: the pieces from the one-dimensional regime can be pushed up without being distorted too much, using the Propositions \ref{sticktostick} and \ref{scapush}, as long as they are not too deep. The resulting pieces have exponentially small precision, see Proposition \ref{precision}. In this way one finds pieces in $\SS_n(\theta^n)\cap E^k$ for 
$0\le k<(1-q_1)\cdot n$. Unfortunately, the relative measure of these pieces in $\SS_n(\theta^n)\cap E^k$ obtained by pushing up, is only exponentially close to $1$, for $k\ge \kappa(n)\asymp \ln n$, see Proposition \ref{meas} . That is why the pushing-up regime is restricted to $\kappa(n)\le k< (1-q_1)\cdot n$ where these pieces occupy $E^k$ except for an exponential small relative part.

The {\it brute-force regime}:   the pieces obtained in the one-dimensional and pushing-up regimes are in $B^{\kappa(n)}$. 
They will be spread around by brute-force iteration of the original map until returning.
The time to go from $B^{\kappa(n)} $ and return by iterating the original map is $2^{\kappa(n)}$. 
The depth $\kappa(n)$  is the largest integer such that $2^{\kappa(n)}\le  Kn\ln 1/\theta$. 
The pieces in the one-dimensional and pushing-up regime have exponentially small precision. Each of the
brute-force return steps used to spread around the pieces from the deeper regimes, will distort their exponential precision $\theta^n$, see Proposition \ref{badpart}.  The total distortion along such a return orbit can be bounded by $O(r^{2^{\kappa(n)}})=O(r^{Kn\ln 1/\theta})$, with $r\gtrsim 1/b>>1$. 
However, this distortion can not destroy the exponential precision when $\theta<1$ is chosen close enough to $1$.

The pushing-up regime is split into two parts. Let $\kappa_0(n)$ be the smallest integer such that $l(\kappa_0(n))\ge n$. As long as $\kappa_0(n)\le k<(1-q_1)\cdot n$ the pieces in $\BB^n[l]$, $l>k$, are not too deep and can be pushed up into $E^k$. Indeed, $\kappa_0(n)\asymp \ln n$ is uniquely defined and can not be adjusted. Unfortunately, we can not use $\kappa(n)=\kappa_0(n)$ because the corresponding return time $2^{\kappa_0(n)}$ used to fill the brute-force regime might be too large.  Too large in the sense that it might build up too much distortion, which is of the order $O(r_0^{n})$ for some definite $r_0>1$. We have to choose $\kappa(n)\asymp \ln n $ much smaller than $\kappa_0(n)$ to get an arbitrarily slow growing rate for the distortion during the brute-force regime. The rate should be small enough such that the exponential decaying precision in the deeper regimes can not be destroyed.
In the regime $\kappa(n)\le k<\kappa_0(n)$ we have 
$l(k)<(1-q_1)\cdot n$ which means that we can not push up all previously recovered pieces in $B^n[l]$ with $l>l(k)$. This is responsible for the super-exponential loss  term in Proposition \ref{meas}.

\newpage

\subsection{Universal sticks created in the one-dimensional regime}

\begin{prop}\label{asympwidth} There exist $\rho<1$, $q^*>0$ with the following property. For every $0<q_0<q_1\le q^*$ there exists  $n^*>0$ such that for $n\ge n^*$ and $(1-q_1)\cdot n\le k\le n$ 
\begin{itemize}
\item [(1)] every $B\in \BB^n[k]$ is regular.
\item [(2)]  for every $B\in \BB^{n+1}[k]$ 
$$
|\sigma_\bB-\sigma^*_\bB|=O(\rho^{q_1\cdot n}), 
$$
where $B=\Psi^{k+1}_0(\bB)$.
\item [(3)]  for every $B\in \BB^n[k]$ with $(1-q_1)\cdot n\le k\le (1-q_0)\cdot n$ 
$$
\delta_\bB=O(\rho^{q_0\cdot n}),
$$
where $B=\Psi^k_0(\bB)$.
\end{itemize}
\end{prop}

Choose,
$(1-q_1)\cdot n\le k\le n$
and $B\in \BB^n[k]$. Let $\bB\in \BB^{n-k}(F_k)$ be such that $B=\Psi_0^k(\bB)$.
Let $\tau_n$ be the tip of $F_n$ and $\tau_*$ the tip of $F_*$.  In the next part we will have to compare the maps $\Psi^n_k$ related to $F$ and the maps $(\Psi^n_k)^*$ corresponding to $F_*$.  Let
$$
\bB_0=B_{v^{n-k}}^{n-k}(F_k)=\Psi^n_k(\Dom(F_n))
$$
and
$$
\bB^*_0=B_{v^{n-k}}^{n-k}(F_*)=(\Psi^{n-k}_0)^*(\Dom(F_*)).
$$
where $(\Psi^{n-k}_0)^*$ is the change of coordinates used to construct $R^{n-k}F_*$.
Then $\bB=F_k^j(\bB_0)$ for some $0\le j<2^{n-k}$ and $j$ is odd. Let 
$\bB_j=F_k^j(\bB_0)$ and $\bB^*_j=F_*^j(\bB^*_0)$ for  $0\le j<2^{n-k}$. We will analyse the relative positions of $\bB_j$ and $\bB^*_j$.
Let 
$$
I_j=\pi_1(\bB_j) 
\quad  \mathrm{and} \quad
J_j=\pi_2(\bB_j).
$$
The intervals in the $n^{th}$ cycle of $f_*$ are denoted by $I^*_j(n)$, see 
\S \ref{1D universal f-s}. Observe,
$$
I^*_j\equiv I^*_j(n-k)=\pi_1(\bB^*_j), \quad  0\le  j<2^{n-k}.
$$
and 
$$ 
J^*_j=\pi_2(\bB^*_j)=I^*_{j-1}(n-k), 
\quad  0< j<2^{n-k}.
$$

Consider the conjugations
$$
h_n:\OO_{F_*}\to \OO_{F_n}
$$
with $h_n(\tau_*)=\tau_{n}$. These conjugations allow us to label the points in 
$\OO_{F_n}$. 
Choose, $z^*\in \OO_{F_*}$ and let 
$z=h_n(z^*)$. 
Let $(x_0,y_0)=\Psi^n_k(z)\in \bB_0$ and
$(x^*_0,y^*_0)=(\Psi^n_k)^*(z^*)\in \bB^*_0$. 
The points in the orbits are
$$
(x_j, y_j)=F_k^j(x_0,y_0)
\quad \mathrm{and}\quad 
(x^*_j, y^*_j)=F_*^j(x^*_0,y^*_0),
$$
with  $0\le j<2^{n-k}$. The first estimates will be on the relative displacements
$\frac{\Delta x_j}{|I^*_j|}$ and $\frac{\Delta y_j}{|J^*_j|}$ where
$
\Delta x_j=x_j-x^*_j 
$
and 
$\Delta y_j=y_j-y^*_j$.

\begin{lem}\label{displacement} 
There exist $\rho<1$, $q^*>0$ with the following property. For every $0<q\le q^*$ there exists  $n^*>0$ such that for $n\ge n^*$, $(1-q)\cdot n\le k\le n$, and $0\le j<2^{n-k}$
$$
\frac{|\Delta x_j|}{|I^*_j|}=O(\rho^{q \cdot n}),
\quad\mathrm{and}\quad
\frac{|\Delta y_j|}{|J^*_j|}=O(\rho^{q \cdot n}).
$$
\end{lem}

\begin{proof} Recall, $y_{j+1}=x_{j}$. Hence,
$$
\frac{|\Delta y_{j+1}|}{|J^*_{j+1}|}= \frac{|\Delta x_{j}|}{|I^*_{j}|},
$$
we only have to estimate the displacements $\Delta x_j$ and $\Delta y_0$.
 Since, $F_k\to F_*$ exponentially fast controlled by some $\rho<1$, see Theorem \ref{universality}, we have  
$$
\begin{aligned}
x_{j+1}&=f_*(x_j)+ O(\rho^k)\\
&=f_*(x^*_j)+Df_*(\zeta_j)\Delta x_j +O(\rho^k).
\end{aligned}
$$
Hence,
$$
\Delta x_{j+1}= Df_*(\zeta_j)\Delta x_j +O(\rho^k).
$$
 There exists $K>1$ such that

\begin{equation}\label{delatxyest}
\frac{|\Delta x_{j+1}|}{|I^*_{j+1}|}\le  \frac{Df_*(\zeta_j)}{\frac{|I^*_{j+1}|}{|I^*_j|}}   \cdot \frac{ |\Delta x_j|}{|I^*_j|} +
K\frac{\rho^k}{\rho_0^{n-k}},
\end{equation}
where we used the {\it a priori} bounds: $|I^*_{j+1}|\ge \rho_0^{n-k}$ for some $\rho_0<1$. 

We will use (\ref{delatxyest}) repeatedly but to do so we first need to estimate $|\Delta x_0|$. 
Let $\Delta z=z-z^*$ and use  the Lemmas \ref{deltapsi}, \ref{contracting}, and  \ref{holmt} in the following estimate
$$
\begin{aligned}
|(x_0,y_0)-(x^*_0, y^*_0)|&\le |\Psi^n_k(z)-(\Psi_k^n)^*(z^*)|\\
&\le |\Psi^n_k-(\Psi_k^n)^*|+|(\Psi^n_k)^*(z)-(\Psi^n_k)^*(z^*)|\\
&\le O(\rho^k)+ |D(\Psi^n_k)^*|\cdot |\Delta z|\\
&=O(\rho^k+ \sigma^{n-k}\cdot \rho^n )\\
&=O(\rho^k).
\end{aligned}
$$ 
Thus,
\begin{equation}\label{deltax0}
 \frac{ |\Delta x_0|}{|I^*_0|}=O(\frac{\rho^k}{\rho_0^{n-k}} )
\end{equation}
and
\begin{equation}\label{deltay0}
 \frac{ |\Delta y_0|}{|J^*_0|}=O(\frac{\rho^k}{\rho_0^{n-k}}).
\end{equation}

Let $r>0$ and $D>1$ be given as in Lemma \ref{distortion} and $K>1$ as defined above.
For $q>0$ small enough and $n\ge 1$ large enough we have
\begin{equation}\label{deltax0}
 \frac{ |\Delta x_0|}{|I^*_0|}=O(\frac{\rho^k}{\rho_0^{n-k}})=O((\frac{\rho^{1-q}}{\rho_0^{q}})^n)=O(\rho^{q\cdot n})\le \frac{r}{2D}.
\end{equation}
and
\begin{equation}\label{deltax02}
DK(\frac{2}{\rho_0})^{n-k}\cdot \rho^k=O((\frac{\rho^{1-q}}{(\rho_0/2)^{q}})^n)=O(\rho^{q\cdot n})\le \frac{r}{2}.
\end{equation}

One has to be careful when applying (\ref{delatxyest}) repeatedly.  The points $\zeta_j$ should not be too far from $I^*_j$ to be able to control distortion. 

\begin{clm} For $q>0$ small enough and $n>1$ large enough
$$
 \frac{ |\Delta x_j|}{|I^*_j|}\le DK(\frac{2}{\rho_0})^{n-k}\cdot \rho^k+D \frac{ |\Delta x_0|}{|I^*_0|},
$$
for $0 \le j<2^{n-k}$.
\end{clm}

\begin{proof} 
The proof is by induction: the statement holds for $j=0$ because $D>1$.  Suppose it holds 
up to $j<2^{n-k}-1$. The $r-$neighborhoods $U_l(n)\supset I^*_l$ were introduced in 
Lemma \ref{distortion}. The induction hypothesis together with (\ref{deltax0}) and (\ref{deltax02}) imply that
$$
\zeta_l\in U_l(n-k)
$$
for $l\le j$.  Now repeatedly apply (\ref{delatxyest}) and Lemma \ref{distortion} to get
$$
\begin{aligned}
 \frac{ |\Delta x_{j+1}|}{|I^*_{j+1}|}&\le  \sum_{l=1}^{j+1} (\prod_{k=l}^{j} \frac{Df_*(\zeta_k)}{\frac{|I^*_{k+1}|}{|I^*_k|}} )\cdot K\frac{ \rho^k}{\rho_0^{n-k}}+ (\prod_{k=0}^{j} \frac{Df_*(\zeta_k)}{\frac{|I^*_{k+1}|}{|I^*_k|}} )\cdot\frac{ |\Delta x_0|}{|I^*_0|}\\
 &\le (j+1) D\index{\cite{\footnote{\cite{\footnote{\pageref{\ref{•}}}}}}}K\frac{ \rho^k}{\rho_0^{n-k}}+D \frac{ |\Delta x_0|}{|I^*_0|}\\
&\le  DK(\frac{2}{\rho_0})^{n-k}\cdot \rho^k+D \frac{ |\Delta x_0|}{|I^*_0|}. 
\end{aligned}
$$
This estimate finishes the induction step.
\end{proof}

Now incorporate the estimates (\ref{deltax0}), (\ref{deltax02}) in the Claim and together with (\ref{deltay0}), Lemma \ref{displacement} follows.
\end{proof}

\noindent
{\it Proof of Proposition \ref{asympwidth}.} 
Let $(1-q_1)\cdot n\le k\le n$ and assume that the conditions of Lemma \ref{displacement} are satisfied. Choose $B\in \BB^n[k]$. Let $\bB\in \BB^{n-k}(F_k)$ be such that $B=\Psi_0^k(\bB)$, say $\bB=\bB_j$ with $0< j<2^{n-k}$ odd. 

The pieces $\bB^*_j\in \BB^{n-k}(F_*)$, $0< j<2^{n-k}$ odd,  are curves on the graph of $f_*$ contained in $B^1_c(F_*)$, 
that is, they have a bounded slope. This bounded slope implies that 
$$
|I^*_j|\asymp |J^*_j|.
$$ 
This bound and Lemma \ref{displacement} imply that the Hausdorff distance between $\bB_j$ and $\bB^*_j$ is 
$O(\rho^{q_0\cdot n}\cdot |I^*_j|)$. We get that 
 $B_j=\Psi_0^k(\bB_j)$ is regular, which  proves Proposition \ref{asympwidth}(1).

Let $B\in \BB^{n+1}[k]$, say  $B=\Psi_0^k(\bB)$ with
$\bB\in \BB^{n-k+1}(F_k)$ and $\bB\subset \bB_j\in \BB^{n-k}(F_k)$, for some $0< j<2^{n-k}$.
Recall that the scaling ratio of $B\in \BB^{n+1}[k]$ is a measurement in vertical direction in the domain of $F_k$. 
The relative displacement of every point $z^*\in \OO_{F_*}$ is estimated in Lemma \ref{displacement}. These bounds imply
$$
|\sigma_\bB-\sigma_{\bB^*}|=O(\rho^{q_0\cdot n}).
$$
This finishes the proof of Proposition \ref{asympwidth}(2).

To control the thickness associated to $B\in \BB^n[k]$ we have to restrict ourselves to 
$(1-q_1)\cdot n\le k\le (1-q_0)\cdot n$.
The piece $\bB\equiv \bB_j$, which determines the relative thickness of $B=\Psi_0^k(\bB)$ has a Hausdorff distance  $O(\rho^{q_0\cdot n}\cdot |I^*_j|)$ to $\bB^*_j$, Lemma \ref{displacement}.  This piece  $\bB^*_j$ is a curve
 in the graph of $f_*$ contained in $B^1_c(F_*)$. 
This curve has a bounded slope. Hence, its relative thickness is
proportional to its diameter, 
which is of the order $\sigma^{n-k}\le \sigma^{q_0\cdot n}$,  see
Lemmas \ref{contracting}. 
The control of the Hausdorff distance and the small relative thickness of $\bB^*_j$ implies
$$
\delta_\bB=O(\rho^{q_0\cdot n})
$$
We finished the proof of Proposition \ref{asympwidth}(3).
\qed

\subsection{Universal sticks created in the pushing-up regime}

\begin{defn}\label{controlled}
Given $0<q_0<q_1$, the collection $\PP_n(k; q_0,q_1)$ of {\it $(q_0,q_1)$-controlled} pieces
consists of $B\in \BB^n[k]$ with the following property. If $B^{(i)}$, $i=0,1,2,\cdots,t$, are the predecessors of $B=B^{(0)}$ with
$$
k=k_{0}(B)<k_1(B)<k_{2}(B)<\dots< k_{t-1}(B)<k_t(B)<n.
$$
then
\begin{itemize}
\item[(1)] $k_{i+1}\le l(k_{i})$, $i=0,1,2,3,\dots, t-1$,
\item[(2)] there exists $0\le s\le t$ such that
$ 
(1-q_1)\cdot n\le k_s(B)\le (1-q_0)\cdot n,
$
and
\item[(3)] 
$  k_{s-1}(B)\le (1-q_1)\cdot n.$
\end{itemize}
\end{defn}

\begin{rem}\label{topcontrolled} The definition of controlled pieces is a combinatorial definition. It does not depend on $F$ but only on the average Jacobian $b_F$ which is a 
topological invariant, \cite{LM1}. If $B$ is a $(q_0,q_1)$-controlled piece of $F$ then the corresponding piece $B^*$ is  $(q_0,q_1)$-controlled piece of $F_*$. 
\end{rem}

The definition of controlled pieces implies
\begin{equation}\label{PP}
\bigcup_{k<l\le l(k)} G_k(\PP_n(l; q_0,q_1))=\PP_n(k; q_0,q_1).
\end{equation}

Proposition \ref{asympwidth} introduced the constants $\rho<1$, and $q^*>0$. The constants $\alpha^*>0$ and $k^*>0$ are the optimal choice given by the Propositions \ref{regtoreg}, \ref{sticktostick} and \ref{scapush}. 
Now Proposition \ref{regtoreg} and Proposition \ref{asympwidth}(1) imply

\begin{lem}\label{regPP} 
Let $\alpha<\alpha^*$. For every $q^*>q_1>q_0>0$  there exists $n^*\ge 1$ such that 
every $B\in \PP_n(k; q_0,q_1)$ and all its predecessors are regular when $n\ge n^*$ and $k\ge k^*$.
\end{lem}

\begin{lem}\label{wdeltas} 
Let $\alpha<\alpha^*$. For every $q^*>q_1>q_0>0$  there exists $n^*\ge 1$ such that 
 for every $\hat{B}\in \PP_n(k;q_0,q_1)$ and $B\in \BB^{n+1}[k]$ with $B\subset \hat{B}$
$$
\delta_{\hat{\bB}}=O(\rho^{q_0\cdot n})
$$
and
$$
|\sigma_\bB-\sigma^*_\bB|=O(\rho^{q_0\cdot n})
$$
when $n\ge n^*$ and $k\ge k^*$.
\end{lem}

\begin{proof}
Let us call the predecessors of $\hat{B}$ and $B$ 
$$
B^{(i)}\subset \hat{B}^{(i)},
$$
$i=0,1,2,\dots, t$.
 Let $k_i=k_i(\hat{B})=k_i(B)$ and 
$\delta_i$ the relative thickness of $\hat{\bB}^{(i)}$, where $\hat{B}^{(i)}=\Psi^{k_i}_0(\hat{\bB}^{(i)})$, and $\sigma_i=\sigma_{\bB^{(i)}}$,  the scaling number of $B^{(i)}\Psi^{k_i}_0(\bB^{(i)})$, $i=0,1,2,\dots, t$. Observe, the piece $B$ 
might have  one predecessor more than $\hat{B}$.

 Apply 
 Propositions \ref{sticktostick} and \ref{scapush}.  In particular, 
\begin{equation}\label{iterw}
\delta_{i-1}\le \frac12 \delta_i+O(\sigma^{n-k_{i}})
\end{equation}
and 
\begin{equation}\label{iters}
|\sigma_{i-1}-\sigma_i|=O(\delta_i +\sigma^{n-k_i}  )
\end{equation}
for $i=1,2,\dots,t$.

Iterating estimate (\ref{iterw}) we obtain
\begin{equation}\label{sumdelta}
\begin{aligned}
\sum_{i=0}^s \delta_i
&\le 2\delta_s+O(\sigma^{n-k_s})\\
&= O(\rho^{q_0\cdot n})+O(\sigma^{q_0\cdot n}),
\end{aligned}
\end{equation}
where we used Proposition \ref{asympwidth}(3) and property (2) of 
Definition \ref{controlled}. We may assume $\sigma<\rho<1$.
The first estimate of the Lemma follows: 
$$
\delta_{\hat{\bB}}=\delta_0\le \sum_{i=0}^s \delta_i=O(\rho^{q_0\cdot n}).
$$
To establish the second estimate of the Proposition, first observe that
$$
\begin{aligned}
\sigma_{\bB^{(0)}}&=\sigma_{\bB^{(s)}}+
\sum_{i=0}^{s-1}( \sigma_{\bB^{(i)}}-\sigma_{\bB^{(i+1)}}).
\end{aligned}
$$
Hence, by using (\ref{iters}) and (\ref{sumdelta}),
$$
\begin{aligned}
|\sigma_{B^{(0)}}-\sigma_{B^{(s)}}|&\le 
\sum_{i=0}^{s-1}| \sigma_{B^{(i)}}-\sigma_{B^{(i+1)}}|\\
&= O(\sum_{i=1}^s( \delta_i+ \sigma^{n-k_{i}}))\\
&=O(\rho^{q_0\cdot n}+\sigma^{n-k_s})
=O(\rho^{q_0\cdot n}).
\end{aligned}
$$
If $B\in \PP_n(k,q_0, q_1)$ and $B^*$ is the corresponding piece of
$F_*$,
 then $B^*$ is also controlled. Namely, each $l(k)=\infty$ because  $b_{F_*}=0$. Hence, we have the same estimate for the proper scaling
$$
|\sigma^*_{\bB^{(0)}}-\sigma^*_{\bB^{(s)}}|
=O(\rho^{q_0\cdot n}).
$$
This finishes the proof.  Namely, $B^{(s)}\in \BB^{n+1}[k_s]$ with $(1-q_1)\cdot n\le k_s\le n$ and we can apply 
Proposition \ref{asympwidth}(2),
$$
\begin{aligned}
|\sigma_\bB-\sigma^*_{\bB}|&=|\sigma_{\bB^{(0)}}-\sigma^*_{\bB^{(0)}}|\\
&\le |\sigma_{\bB^{(0)}}-\sigma_{\bB^{(s)}}|+
 |\sigma_{\bB^{(s)}}-\sigma^*_{\bB^{(s)}}|+|\sigma^*_{\bB^{(s)}}-\sigma^*_{\bB^{(0)}}|\\
&=O(\rho^{q_0\cdot n}).
\end{aligned}
$$
\end{proof}

The measurements of the pieces, such as scaling and thickness, are geometrical quantities observed when viewing a piece from its scale, they are geometrical measurements of $\bB$ and not $B$ itself. The next Proposition states that the actual pieces $B$ inherit exponentially small estimates for their precision. The Proposition is also a preparation for the brute-force regime which concerns iteration of the original map.

\begin{prop}\label{precision} Let $\alpha<\alpha^*$. For every $q^*>q_1>q_0>0$  there exists $n^*\ge 1$ such that 
$$
\PP_n(k;q_0,q_1)\subset \SS_n(O(\rho^{q_0\cdot n}))
$$
when $n\ge n^*$ and $k\ge k^*$.
\end{prop}

The estimates in the proof of this Proposition are like the estimates used to prove the Propositions \ref{regtoreg}, \ref{sticktostick},
 and \ref{scapush}.

\begin{proof} 
Let $\hat{B}\in \PP_n(k;q_0,q_1)$ and $B\in\BB^{n+1}[k]$ with 
$B\subset \hat{B}$. 
Let $\bB$ and $\hat{\bB}$ be such that $B=\Psi_0^k(\bB)$ and
 $\hat{B}=\Psi_0^k(\hat{\bB})$. The horizontal and vertical size of the smallest rectangle which contains $\hat{\bB}$ are $\bh,\bv>0$. Let $\Bde>0$ be the relative thickness of $\hat{\bB}$, the absolute thickness of  
$\hat{\bB}$ is $\bw=\Bde\cdot \bh$.   From Lemma \ref{contracting} we get
$$
\bh,\bv =O(\sigma^{n-k}).
$$
Moreover, the regularity of $\hat{B}$ gives
$$
\bh\asymp \bv.
$$
The situation allows to apply Lemma \ref{wdeltas} :
\begin{equation}\label{deltasigmaestimate}
\Bde=O(\rho^{q_0\cdot n})
\quad \mathrm{and} \quad
|\sigma_\bB-\sigma^*_{\bB}|=O(\rho^{q_0\cdot n}).
\end{equation}
We have to show that $\hat{B}\cap \OO_{F}=\Psi^k_0(\hat{\bB}\cap
\OO_{F_k})$ is contained in a $O(\rho^{q_0\cdot n})$-stick. As before
we will decompose $\Psi_0^k$ into its diffeomorphic part 
$(\id+{\bf S}_0^k)$ and its affine part. Let 
$h_{\text{diff}}, v_{\text{diff}}>0$ be the horizontal and vertical size of the smallest rectangle containing the image of $\hat{\bB}$ under  $(\id+{\bf S}_0^k)$ and 
$w_{\text{diff}}>0$ the absolute thickness of its stick and $\sigma_{\text{diff}}>0$ the scaling factor of the image of $\bB$ under the same diffeomorphism. Then we have
$$
\begin{aligned}
\sigma_{\text{diff}}&=\sigma_\bB,\\
v_{\text{diff}}&=\bv
\end{aligned}
$$
and, by recalling (\ref{wdiff}),
$$
\begin{aligned}
w_{\text{diff}}&=O( \bw +\sigma^{n-k} \cdot \bh ),\\
h_{\text{diff}}&\asymp \bh.
\end{aligned}
$$
The last two estimates rely on $\bv\asymp \bh$. The term $\bh \cdot \sigma^{n-k}$ reflects the distortion of  $(\id+{\bf S}_0^k)$ on $\hat{\bB}$ determined by the 
diameter of $\hat{\bB}$ which is of the order $\sigma^{n-k}$. The next step is to apply the affine part of $\Psi_0^k$. Denote the measurements after this step by $h_{\text{aff}}, v_{\text{aff}}, w_{\text{aff}}, \sigma_{\text{aff}}>0$ resp. 
Equation (\ref{reshuffling}) and Lemma \ref{tilt} yield
\begin{equation}\label{oldaffestimates}
\begin{aligned}
w_{\text{aff}}&\asymp \sigma^{2k}w_{\text{diff}},\\
\sigma_{\text{aff}}&=\sigma_{\text{diff}}=\sigma_\bB,\\
h_{\text{aff}}&\asymp \sigma^{2k} h_{\text{diff}}+ \sigma^kv_{\text{diff}},\\
\end{aligned}
\end{equation}
Use the above estimates in the following 
\begin{equation}\label{waffhaff}
\frac{w_{\text{aff}}}{h_\text{aff}}=O(\frac{\sigma^{2k} \cdot [\bw+\sigma^{n-k}\cdot \bh ]}
{\sigma^{2k} \cdot \bh +\sigma^{k}\cdot \bv})= O(\sigma^k\cdot \Bde +\sigma^{n})=O(\rho^{q_0\cdot n}).
\end{equation}

\begin{figure}[htbp]
\begin{center}
\psfrag{vaff}[c][c] [0.7] [0] {\Large $v_{\text{aff}}$}
\psfrag{haff}[c][c] [0.7] [0] {\Large $h_{\text{aff}}$}
\psfrag{saffl}[c][c] [0.7] [0] {\Large $\sigma_{\text{aff}} l'$}
\psfrag{waff}[c][c] [0.7] [0] {\Large $w_{\text{aff}}$}
\psfrag{dl}[c][c] [0.7] [0] {\Large $\Delta l'$}
\psfrag{l}[c][c] [0.7] [0] {\Large $l'$}
\psfrag{w'}[c][c] [0.7] [0] {\Large $w'$}
\psfrag{s'l}[c][c] [0.7] [0] {\Large $\sigma' l'$}
\psfrag{sb}[c][c] [0.7] [0] {\Large $\sigma_{\bB} v_{\text{aff}}$}
\psfrag{BOF}[c][c] [0.7] [0] {\Large $B \cap \OO_F$}
\pichere{0.9}{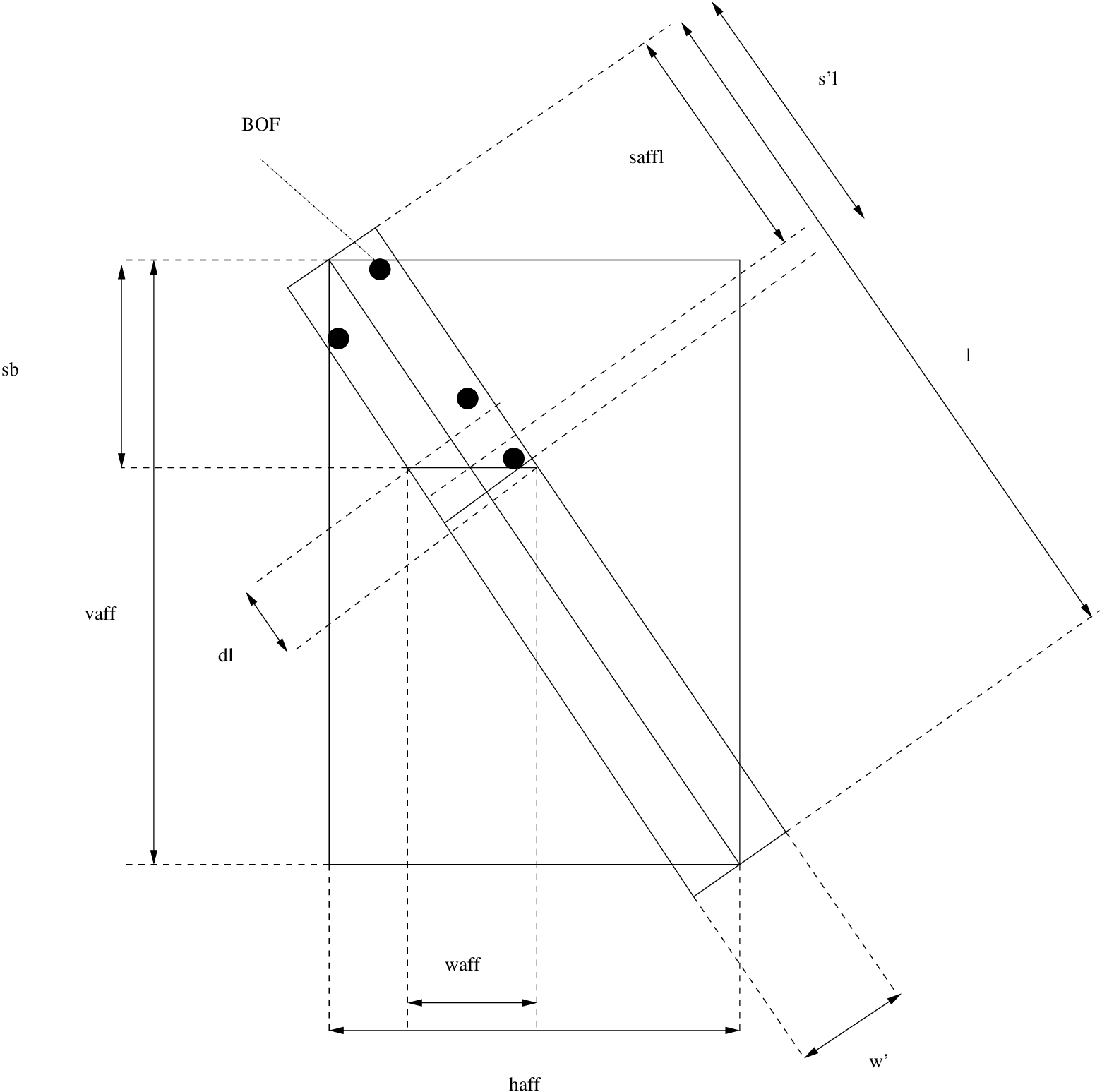} 
\caption{} \label{actualstick}
\end{center}
\end{figure}

Consider the smallest conformal image of a rectangle aligned along the diagonal of the rectangle containing $\hat{B}=\Psi^k_0(\hat{\bB})$,
see Figure \ref{actualstick}. The precision of $\hat{B}$ will be better than the precision based on the measurements of this approximation of the 
stick. Let $l'>0$ be the length, $w'>0$ be the absolute thickness and $\sigma'>0$ be the scaling factor of $B\subset \hat{B}$ within this rectangle. Then
\begin{equation}\label{l'}
l'=\sqrt{h_{\text{aff}}^2+v^2_{\text{aff}}},
\end{equation}
and 
\begin{equation}\label{w'}
w'\le w_{\text{aff}}.
\end{equation}
First we will estimate the precision of $\sigma'$. Let $\gamma$ be the angle between the diagonal of the rectangle and the horizontal.  Observe,
$$
\cos \gamma=\frac{h_{\text{aff}}}{\sqrt{h_{\text{aff}}^2+v^2_{\text{aff}}}},
$$
see Figure \ref{actualstick}. The projection $\Delta l'$ of the horizontal interval of length $w_{\text{aff}}$ onto the diagonal has length
$$
\Delta l'=w_{\text{aff}} \cdot \cos \gamma.
$$
Observe,
$$
|\sigma' \cdot l'-\sigma_{\text{aff}}\cdot l'|\le \Delta l'=w_{\text{aff}} \cdot \frac{h_{\text{aff}}}{\sqrt{h_{\text{aff}}^2+v^2_{\text{aff}}}}.
$$
Then, by using (\ref{waffhaff}) and (\ref{l'}), 
\begin{equation}\label{delsigma}
|\sigma' -\sigma_{\text{aff}}|\le \frac{w_{\text{aff}} }{h_{\text{aff}} }\cdot \frac{h^2_{\text{aff}}}{h_{\text{aff}}^2+v^2_{\text{aff}}}\le \frac{w_{\text{aff}} }{h_{\text{aff}} }=
 O(\rho^{q_0\cdot n}).
\end{equation}
Use (\ref{deltasigmaestimate}), (\ref{oldaffestimates}), and (\ref{delsigma}) to estimate the precision of $\sigma'$
\begin{equation}\label{delsig}
|\sigma'-\sigma_{\bB}^*|\le | \sigma'-\sigma_{\text{aff}}|+|\sigma_{\text{aff}}-\sigma_{\bB}^*|= O(\rho^{q_0\cdot n}).
\end{equation}
The estimate (\ref{w'}) says that the height of the stick containing $\hat{B}$ is at most $w_{\text{aff}}$. 
The relative height is estimated by
\begin{equation}\label{relthick}
\frac{w'}
              {l'}\le \frac{w_{\text{aff}}}
                                {\sqrt{h_{\text{aff}}^2+v^2_{\text{aff}}}}\le 
\frac{w_{\text{aff}}}
       {h_{\text{aff}}}=O(\rho^{q_0\cdot n}),
\end{equation}
where we used (\ref{waffhaff}) and (\ref{l'}).
The estimates (\ref{delsig}) and (\ref{relthick}) confirm that
$
\hat{B}\in \SS_n(\rho^{q_0\cdot n})),
$
which finishes the proof of the Proposition.
\end{proof}

\subsection{Universal sticks created in the brute-force regime}

\begin{prop}\label{badpart} There exists $\epsilon^*>0$, and $q^*>0$ such that the following holds.
Let $\epsilon<\epsilon^*$, and $0<q_0<q_1<q^*$ then there exists $n^*\ge 1$ such that
 if for $0\le j< 2^{(1-q_1)\cdot n}$
$$
F^j(B)\in\SS_n(\epsilon),
$$
 with
$B\in \BB^n[k]$, $(1-q_1)\cdot n\le k\le (1-q_0)\cdot n$, and $n\ge n^*$, then
$$
F^{j+1}(B)\in \SS_n(O(\epsilon+\rho^{q_0\cdot n})).
$$
\end{prop}

\begin{proof}
 Choose $\hat{B}\in \BB^n[k]$ 
with $(1-q_1)\cdot n\le k\le (1-q_0)\cdot n$ and $B\in \BB^{n+1}$ with $B\subset \hat{B}$. The iterates under the original map are denoted by
$B_j=F^j(B)$ and $\hat{B}_j=F^j(\hat{B})$, $j\le 2^{(1-q_1)\cdot n}$.
Assume that for some  $j\le 2^{(1-q_1)\cdot n}$
$$
\hat{B}_j\in \SS_n(\epsilon).
$$
The piece $\hat{B}_j$ is contained in an $\epsilon$-stick.
 Say $\hat{B}_j\cap \OO_F$ is contained in a rectangle of length $l>0$ and height
 $w\le \epsilon l$. The smaller rectangle which contains $B_j\cap \OO_F$ has length $\sigma_{j} l$, where $\sigma_j=\sigma_{B_j}$ and 
 $|\sigma_j-\sigma^*_{\bB_j}|\le \epsilon$.  Notice that we have to estimate the scaling factor $\sigma_{B_j}$ and not $\sigma_{\bB_j}$, compare remark \ref{prosca}.

Apply $F$ to this rectangle. The stick which contains $\hat{B}_{j+1}$ has length $l'>0$ and height $w'>0$. The relevant scaling factor of $B_{j+1}$ is $\sigma_{j+1}=\sigma_{B_{j+1}}$.

Choose, $M, m>0$ such that 
$$
m |v| \le |DF(x,y)v|\le M |v|.
$$
This is possible because $F$ is a diffeomorphism onto its image. However, $m=O(b)$.
Let $K>0$ be the maximum norm of the Hessian of $F$. The diameter of $\hat{B}_j\cap \OO_F$, which is proportional to $l$,  is of the order $\sigma^n$, see Lemma \ref{contracting}. We can estimate the sizes $l', w'$ and $\sigma'$ by applying the derivative of $F$ and correcting for 
distortion which is bounded by $K l^2$. Let $D$ be the absolute value of the directional derivative of $F$ in the direction of the rectangle containing $\hat{B}_j$, measured in a corner of the rectangle.
Then
$$
\begin{aligned}
l'&\ge Dl-2K l^2-2M w,\\
w'&\le M w+2K l^2,\\
\end{aligned}
$$
Observe,
$$
|\sigma_{j+1} \cdot l'-D \cdot \sigma_{j} \cdot l|\le 2M w+2K l^2.
$$
Let us first estimate the relative height of the stick of $\hat{B}_{j+1}$. Use $w\le \epsilon l$,
\begin{equation}\label{epsw}
\begin{aligned}
\frac{w'}{l'}&\le \frac{ M\epsilon l+2K l^2}{ ml-2K l^2-2M\epsilon l}\\
&\le \frac{M}{m-2K l-2M\epsilon }\cdot \epsilon+2\frac{K}{m-2K l-2M\epsilon } \cdot l\\
&=O(\epsilon+\sigma^n)=O(\epsilon+\rho^{q_0\cdot n}),
\end{aligned}
\end{equation}
when $\epsilon<\epsilon^*$, $q_0< q^*_1$ small enough, and $n\ge n^*$ large enough. Similarly, 
\begin{equation}\label{deltasigma}
|\sigma_{j+1}-\sigma_{j}|=O(\epsilon+\rho^{q_0\cdot n}).
\end{equation}

Use remark \ref{1Dscalingstar}  and apply Proposition \ref{disto} to get 
\begin{equation}\label{deltasigmapro}
|\sigma^*_{\bB_s}-\sigma^*_{\bB}|=O(\rho^{q_0\cdot n}),
\end{equation}
with $0\le s< 2^{(1-q_1)\cdot n}$.

We need to estimate the scaling factor $\sigma_{j+1}$ of $B_{j+1}$.
 Use (\ref{deltasigma}) and (\ref{deltasigmapro}) 
and and the notation  $\sigma^*_j=\sigma^*_{\bB_j}$. Then 
\begin{equation}\label{epssig}
\begin{aligned}
|\sigma_{j+1}-\sigma^*_{j+1}|&\le
 |\sigma_{j+1}-\sigma_j|+|\sigma_j-\sigma^*_{j}|+|\sigma^*_j-\sigma^*_{j+1}|\\
&\le O(\epsilon+\rho^{q_0\cdot n})+\epsilon+O(\rho^{q_0\cdot n})\\
&=O(\epsilon+\rho^{q_0\cdot n}),
 \end{aligned}
\end{equation}
for $\epsilon\le \epsilon^*$, $0<q_0<q^*$ small enough and $n\ge n^*$ large enough. The estimates (\ref{epsw}) and (\ref{epssig}) together finish the
proof.
\end{proof}

\section{Probabilistic Universality}\label{conv}

In this section we are going to estimate the measure of the pieces created in the three regimes, see Proposition \ref{Pn1mintheta}.
Let $\alpha=\alpha^*, \epsilon^*>0$, and $0< q^*_1<1/3$ small and $k^*\ge 1$ large enough to allow the use of the Propositions \ref{precision}, and \ref{badpart}.   

\bigskip

For each $n\ge 1$, let $\kappa_0(n)\asymp \ln n$ be the smallest integer such that
$$
l(\kappa_0(n))\equiv2^{\kappa_0(n)} \cdot \frac{\ln b}{\ln \sigma}-\frac{\ln \alpha}{\ln \sigma}+\kappa_0(n)\ge n.
$$
For $n\ge 1$ large enough we have 
\begin{equation}\label{kapn}
\kappa_0(n)\le \frac{\ln n}{\ln 2}.
\end{equation}


\begin{lem}\label{pushupregime} Given $q_0<q_1$. There exists $n^*\ge 1$ such that for $n\ge n^*$ and
$\kappa_0(n)\le k<(1-q_0)\cdot n$,
$$
\mu(\PP_n(k; q_0,q_1))\ge [1-\frac{1}{2^{(q_1-q_0)\cdot n+1}}] \cdot \mu(E^k).
$$
\end{lem}

\begin{proof} Let $\beta_n(k; q_0,q_1)=\mu(E^k\setminus\PP_n(k; q_0,q_1) )$ be the measure of the uncontrolled pieces.
The construction implies immediately
\begin{equation}\label{betamuEk}
\beta_n(k; q_0,q_1)=\mu(E^k), \quad (1-q_0)\cdot n< k\le n,
\end{equation}
 and 
\begin{equation}\label{beta0}
\beta_n(k; q_0,q_1)=0, \quad (1-q_1)\cdot n\le k\le (1-q_0)\cdot n,
\end{equation}
every piece in the one-dimensional regime is controlled. The Lemma holds for $(1-q_1)\cdot n\le k\le (1-q_0)\cdot n$.
This implies that the fraction of the uncontrolled part in $\cup_{l\ge (1-q_1)\cdot n} E^l$ is
\begin{equation}\label{qq}
\frac{\sum_{l=(1-q_1)\cdot n}^n\beta_n(l; q_0,q_1)} 
{\mu(B^{(1-q_1)\cdot n})}\le \frac{1}{2^{(q_1-q_0)\cdot n+1}}.
\end{equation}
Observe,
$$
\begin{aligned}
l((1-q_1)\cdot n-1)&= 2^{(1-q_1)\cdot n-1}\cdot \frac{\ln b}{\ln \sigma}-
\frac{\ln \alpha}{\ln \sigma}+(1-q_1)\cdot n-1\\
&\gtrsim 2^{(1-q_1)\cdot n-1}\ge n,
\end{aligned}
$$
holds when $n\ge 1$ is large enough. All pieces in $\BB^n[k]$, with $k\ge (1-q_1)\cdot n$ are not too deep for level 
$(1-q_1)\cdot n-1$. Hence, equation (\ref{PP}) reduces to 
$$
\PP_n((1-q_1)\cdot n-1; q_0,q_1)=\bigcup_{(1-q_1)\cdot n\le l\le n} G_{(1-q_1)\cdot n-1}(\PP_n(l; q_0,q_1)).
$$
Hence, using (\ref{qq}),
$$
\begin{aligned}
\beta_n((1-q_1)\cdot n-1; q_0,q_1)&=\sum_{l=(1-q_1)\cdot n}^n\beta_n(l; q_0,q_1)\\
&\le \frac{1}{2^{(q_1-q_0)\cdot n+1}}\cdot\mu(B^{(1-q_1)\cdot n})\\
&=\frac{1}{2^{(q_1-q_0)\cdot n+1}}\cdot \mu(E^{(1-q_1)\cdot n-1}).
\end{aligned}
$$
Now we finish the proof by induction. The Lemma is proved for $k=(1-q_1)\cdot n-1$.  Assume the Lemma holds from $(1-q_1)\cdot n-1$ down 
to $k+1\le (1-q_1)\cdot n-1$. Because $k\ge \kappa_0(n)$ we have $l(k)\ge n$.  Hence, again by using 
(\ref{PP}), (\ref{betamuEk}),  (\ref{beta0}),  and $\mu(E^l)=\frac{1}{2^{l+1}}$, $l\ge 0$, we get
$$
\begin{aligned}
\mu(\PP_n(k; q_0,q_1))&=\mu(\bigcup_{l=k+1}^n G_k(\PP_n(l; q_0,q_1)))\\
&= \sum_{l=k+1}^{ (1-q_1)\cdot n-1} \mu(\PP_n(l; q_0,q_1)))+
\sum_{l=(1-q_1)\cdot n}^{(1-q_0)\cdot n} \mu(E^l)\\
&\ge (1-\frac{1}{2^{(q_1-q_0)\cdot n+1}})\cdot [
\sum_{l=k+1}^{ (1-q_1)\cdot n-1}  \mu(E^{l})+  \frac{1}{2^{(1-q_1)\cdot n}}]\\
&= (1-\frac{1}{2^{(q_1-q_0)\cdot n+1}})\cdot [
\sum_{l=k+1}^{ (1-q_1)\cdot n-1}  \mu(E^{l})+\sum_{l=(1-q_1)\cdot n}^{\infty} \mu(E^l)]\\
&=(1-\frac{1}{2^{(q_1-q_0)\cdot n+1}})\cdot \mu(E^k).
\end{aligned}
$$
\end{proof}

\begin{prop}\label{meas}   Given $q_0<q_1<q^*_1$. There exists $n^*\ge 1$ such that for $n\ge n^*$ and $k\le (1-q_0)\cdot n$
$$
\mu(\PP_n(k; q_0,q_1))\ge [1- \frac{1}{2^{(q_1-q_0)\cdot n+1}}
-2^{\frac{\ln \alpha\sigma}{\ln \sigma}}\sum_{l= k}^{\infty} 2^l (b^\gamma)^{2^l}]\cdot 
\mu(E^k),
$$
where $\gamma=- \frac{\ln 2}{\ln \sigma}\in (0,1)$.
\end{prop}

\begin{proof}
According to Lemma \ref{pushupregime}, the Proposition holds for $\kappa_0(n)\le k\le (1-q_0)\cdot n$. 
 The proof for the lower values of $k<\kappa_0(n)$ is by induction. 
 Assume by induction 
$$
\beta_n(k; q_0,q_1))\le [\frac{1}{2^{(q_1-q_0)\cdot n+1}}
+2^{\frac{\ln \alpha\sigma}{\ln \sigma}}
\sum_{l= k}^{\kappa_0(n)-1} 2^l (b^\gamma)^{2^l}]\cdot \mu(E^k),
$$
which holds for  $k=\kappa_0(n)$. Suppose it holds from $\kappa_0(n)$ down to 
$k+1\le \kappa_0(n)$.
Observe,
$$
\frac{1}{2^{l(k)}}=2^{\frac{\ln \alpha}{\ln \sigma}}\cdot 2^{-[\frac{k}{2^k}+\frac{\ln b}{\ln \sigma}]\cdot 2^k}
\le 2^{\frac{\ln \alpha}{\ln \sigma}}\cdot 2^{-\frac{\ln b}{\ln \sigma}\cdot 2^k}.
$$
Hence,
\begin{equation}\label{2lk}  
\frac{1}{2^{l(k)}}\le   2^{\frac{\ln \alpha}{\ln \sigma}}\cdot (b^\gamma)^{2^k}.
\end{equation}
Use (\ref{kapn}) and observe,
$$
\begin{aligned}
l_{\kappa_0(n)-1}&=2^{\kappa_0(n)-1}\cdot \frac{\ln b}{\ln \sigma}-
\frac{\ln \alpha}{\ln \sigma}+\kappa_0(n)-1\\
&=\frac12 (n+\frac{\ln \alpha}{\ln \sigma}-\kappa_0(n))
-\frac{\ln \alpha}{\ln \sigma}+\kappa_0(n)-1\\
&\le \frac12 n(1+\frac{\kappa_0(n)}{n})+O(1)\\
&\le  \frac12 n(1+\frac{\ln n}{n\ln 2})+O(1)\\
&<(1-q_1)\cdot n,
\end{aligned}
$$
holds when $n\ge n^*$ large enough because $q^*_1<\frac13$.
Hence, for $n\ge 1$ large enough, we have
\begin{equation}\label{lk}
l(k)\le l(\kappa_0(n)-1)< (1-q_1)\cdot n.
\end{equation}
Use (\ref{PP}), the induction hypothesis, (\ref{2lk}), and (\ref{lk}) in the following 
 estimates.
$$
\begin{aligned}
\beta_n(k; q_0,q_1))&\le \sum_{l=k+1}^{l(k)} \beta_n(l; q_0,q_1)+
\mu(B^{l(k)+1})\\
&\le  [\frac{1}{2^{(q_1-q_0)\cdot n+1}}
+2^{\frac{\ln \alpha\sigma }{\ln \sigma}}
\sum_{l= k+1}^{\kappa_0(n)-1} 2^l (b^\gamma)^{2^l}]\cdot \sum_{l=k+1}^{l(k)}\mu(E^l)
+\frac{1}{2^{l(k)}}\\
&\le 
 [\frac{1}{2^{(q_1-q_0)\cdot n+1}}
+2^{\frac{\ln \alpha\sigma }{\ln \sigma}}
\sum_{l= k+1}^{\kappa_0(n)-1} 2^l (b^\gamma)^{2^l}]\cdot \mu(E^k)+
 2^{\frac{\ln \alpha}{\ln \sigma}} (b^\gamma)^{2^k}\\
&= [\frac{1}{2^{(q_1-q_0)\cdot n+1}}
+2^{\frac{\ln \alpha\sigma}{\ln \sigma}}
\sum_{l= k}^{\kappa_0(n)-1} 2^l (b^\gamma)^{2^l}]\cdot \mu(E^k),
\end{aligned}
$$
where the last equality uses $\mu(E^k)=\frac{1}{2^{k+1}}$.
\end{proof}

For each $K>0$ and $\theta<1$,  let $\kappa(n)$ be the largest integer such that
$$
2^{\kappa(n)}\le Kn\ln 1/\theta.
$$
 
\begin{lem}\label{kn} There exists $K>0$ such that for every $\theta<1$ there exists $n^*\ge 1$ such that $\kappa(n)\ge k^*$ for $n\ge n^*$ and 
$$
2^{\frac{\ln \alpha\sigma}{\ln \sigma}}\sum_{l= \kappa(n)}^{\infty} 2^l (b^\gamma)^{2^l}\le \frac13 \theta^n.
$$
\end{lem}

\begin{proof} Observe,
$$
\sum_{l= \kappa(n)}^{\infty} 2^l (b^\gamma)^{2^l}=O(2^{\kappa(n)}(b^\gamma)^{2^{\kappa(n)}}).
$$
To achieve the property of the Lemma it suffices to satisfy
$$
\ln 2^{\kappa(n)}+ 2^{\kappa(n)}\ln b^\gamma +O(1)\le n\ln \theta.
$$
In turn, this holds when
$$
n\ln 1/\theta \cdot [\frac12 K \ln b^\gamma +1]+O(1)\le 0.
$$
This holds for large $n\ge 1$ when $K>0$ is chosen large enough.
\end{proof}

In the sequel we will fix $K>0$ according to the previous Lemma.
For each $Q>0$ and $\theta<1$,  define $q_0$ by
$$
q_0=Q\ln 1/\theta.
$$
and 
$$
q_1=[Q+\frac{3}{2 \ln 2}]\cdot \ln 1/\theta.
$$
\begin{lem}\label{q0q1} For every $\theta<1$ there exists $n^*\ge 1$ such that for $Q>0$ and $n\ge n^*$
$$
\frac{1}{2^{(q_1-q_0)\cdot n +1}}\le \frac13 \theta^n.
$$
\end{lem}

The brute-force regime consists of iterates of $\bigcup_{l=\kappa(n)}^{(1-q_0)\cdot n}    \PP_n(l;q_0,q_1) $  up to just one step before the moment of return to $B^{\kappa(n)}\equiv \bigcup_{l=\kappa(n)}^\infty E^l$. The return uses exactly  $2^{\kappa(n)}$ steps. Thus we obtain for each choice $Q>0$ and $\theta<1$, the collection
\begin{equation}\label{PPn} 
\PP_n=\bigcup_{j=0}^{2^{\kappa(n)} -1}F^j(  \bigcup_{l=\kappa(n)}^{(1-q_0)\cdot n} 
\PP_n(l;q_0,q_1)         )
\end{equation}

\begin{prop}\label{muPnSn}  There exist $Q>0$ and $\theta^*<1$ such that the following holds. For $\theta^*\le \theta<1$ there exists $n^*\ge 1$ such that for $n\ge n^*$
$$
\PP_n \subset \SS_n(\theta^n).
$$
\end{prop}

\begin{proof} Take $B\in \bigcup_{l=\kappa(n)}^{(1-q_0)\cdot n}    \PP_n(l;q_0,q_1) $. According to 
Proposition \ref{precision}  there exists $C>0$ such that
\begin{equation}\label{BinS}
B\in \SS_n(C\rho^{q_0\cdot n}),
\end{equation}
when $\theta<1$ close enough to $1$ and $n\ge 1$ large enough (Recall that $q_0$ depends on $\theta$).
Now consider an image $F^j(B)$ with $j\le 2^{\kappa(n)}-1<2^{(1-q_1)\cdot n}$. Denote its precision by $\epsilon_j$.   This is a piece in the brute-force regime.
If $\theta<1$ close enough to $1$ and $n\ge 1$ large enough we can apply Proposition \ref{badpart}: there exists $r>1$ such that if 
$\epsilon_j\le \epsilon^*$ then
\begin{equation}\label{repeateps}
\epsilon_{j+1}\le r\cdot (\epsilon_j+\rho^{q_0\cdot n}).
\end{equation}
Choose $Q>0$ large enough such that
$$
Q\ln \rho +K\ln r+\frac32\le 0.
$$
This choice implies
\begin{equation}\label{Crbound}
\rho^{q_0\cdot n}   \cdot  r^{2^{\kappa(n)}}\le (\theta^{\frac32})^n.
\end{equation}
Now we can repeatedly apply (\ref{repeateps}): for $n\ge 1$ large enough and $0\le j<2^{\kappa(n)}$
$$
\begin{aligned}
\epsilon_j&\le  C\rho^{q_0\cdot n} \cdot r^j
+ \rho^{q_0\cdot n}\cdot \sum_{i=0}^{j-1} r^{j-i}\\
&\le (C+\frac{r}{r-1})\cdot \rho^{q_0\cdot n}  \cdot r^{2^{\kappa(n)}}\le \theta^n\le \epsilon^*.
\end{aligned}
$$
Every piece in $\PP_n$ is $\theta^n$-universal.
\end{proof}

In the sequel we will fixed $Q>0$ according to the previous Proposition.

\begin{prop}\label{Pn1mintheta} There exists $\theta^*<1$ such that the following holds. For $\theta^*\le \theta<1$ there exists $n^*\ge 1$ such that for $n\ge n^*$
$$
\mu(\PP_n)\ge 1-\theta^n.
$$
\end{prop}

\begin{proof} For $\theta<1$ close enough to $1$ we have
$$
\frac{1}{2^{\frac12- Q\ln 1/\theta}}\le \theta.
$$
Hence, for $n\ge 1$ large enough
\begin{equation}\label{QKtheta}
\frac{Kn\ln 1/\theta}{2^{(1-Q\ln 1/\theta)\cdot n+1}}\le \frac13\cdot \theta^n.
\end{equation}

For $\theta<1$ close enough to $1$, and $n\ge 1$ large enough we can apply Proposition \ref{meas}, Lemmas \ref{kn}, \ref{q0q1}, and (\ref{QKtheta})
to obtain
$$
\begin{aligned}
\mu( \PP_n  )&= 
2^{\kappa(n)}\cdot \mu(\bigcup_{l=\kappa(n)}^{(1-q_0)\cdot n}    \PP_n(l;q_0,q_1)   )\\
&\ge 2^{\kappa(n)}\cdot(1-\frac23\theta^n)\cdot \sum_{l=\kappa(n)}^{(1-q_0)\cdot n} \mu(E^l)\\
&=(1-\frac23\theta^n)\cdot (1-\frac{2^{\kappa(n)}}{2^{(1-q_0)\cdot n+1}})\\
&\ge (1-\frac23\theta^n)\cdot (1-\frac{Kn\ln 1/\theta}{2^{(1-Q\ln 1/\theta)\cdot n+1}})\\
&\ge 1-\theta^n.
\end{aligned}
$$
\end{proof}

The Propositions \ref{Pn1mintheta} and \ref{muPnSn} confirm probabilistic universality, Theorem \ref{distuni}.

\section{Recovery}

The pieces in $\BB^n$ which are contained in $\theta^n$-sticks can be determined by pure combinatorial methods. In \cite{CLM}, it has been shown that there are pieces which are not contained in $\theta^n$-sticks. Probabilistic universality says that these {\it bad} spots will be filled on deeper levels with pieces contained in sticks with exponential precision. This recovery process has a combinatorial description. 
              
A piece $B\in \BB^n$ has an associated word 
$\omega=w_1w_2 \dots w_n$, with letters $w_k\in\{c,v\}$, such that
$$
B=\Im \psi^1_{w_1}\circ  \psi^2_{w_2}\circ \dots  \psi^n_{w_n}
$$
where $\psi^k_{v}$ is the non-affine rescaling used to renormalize $R^kF$, and to obtain $R^{k+1}F$ and $\psi^k_{c}=R^kF\circ \psi^k_{v}$. If $B_1, B_2\in \BB^{n+1}$ are the two pieces contained in $B$ then the associated words for $B_1$ and $B_2$ are $wc$ and $wv$.  This discussion defines a homeomorphism
$$
w:\OO_F\to \{c,v\}^{\Bbb{N}}.
$$
The relation between the $k_i(B)$, $i=0,1,2,\dots, t$, which define the predecessors of $B\in \BB^n$ and the word $\omega=w_1w_2\dots w_n$ is as follows. If 
$i\in \{k_0(B),k_1(B),\dots, k_t(B)\}$ then 
$w_i=c$, otherwise $w_i=v$.

\comm{
In the next section, see (\ref{PPn}), we will introduce collections $\PP_n\subset \BB_n$ of pieces
which are contained in $\theta^n$-sticks, and prove for some $\theta<1$,
\begin{equation}\label{PnSn}
\PP_n\subset \SS_n(\theta^n).
\end{equation}
These collections turn out to have a large measure,
\begin{equation}\label{mPn}
\mu(\PP_n)\ge 1-\theta^n.
\end{equation}
The selection process for the pieces which are in $\PP_n$, see (\ref{PPn}),  is purely topological. 
The influence of the actual map  is limited by the value of its average Jacobian which is a topological invariant, see \cite{LM}. 

Let $k_1(n)$ be defined by
$$
2^{k_1(n)}=A_1(1+n\ln \frac{1}{\theta}),
$$
and 
$$
q_0(n)=A_2(1+n\ln \frac{1}{\theta}),
$$
$$
q_1(n)=(A_2+A_3)(1+n\ln \frac{1}{\theta}).
$$ 
The precise values of $\theta<1$, and $A_1, A_2, A_3>0$ will be discussed in the next section.  These values are universal, they do not depend on the actual map. The value of $\theta$ will be close to $1$. This means that $q_0(n)$ and $q_1(n)$ are small fractions of $n$ and 
\begin{equation}\label{k1n}
k_1(n)=(1+o(1))\frac{\ln n}{\ln 2}.
\end{equation}
}

In the previous section we constructed the collection  $\PP_n\subset \SS_n(\theta^n)$, see (\ref{PPn}).
The word $\omega=w_1w_2\dots w_n$ of a piece $B\in \PP_n$ is characterized by
\begin{itemize}
\item [(1)] If $k\ge \kappa(n)$ and $w_k=c$ then there exists $k<i\le l(k)$ with $w_i=c$.
\item [(2)] There exits $n-q_1\cdot n\le k\le n-q_0\cdot n$ with $w_k=c$.
\end{itemize}

\begin{rem} Recall, $q_0$, $q_1$, and the function $l(k)$,  depend only on the average Jacobian, which is a topological invariant, see \cite{LM1}. The characterization of the pieces in $\PP_n$ is purely topological.
\end{rem}

\begin{defn}\label{controlpt} A point $x\in \OO_F$ is {\it eventually controlled} if there exists $N_x\ge 1$ such that for all $n\ge N_x$ there exists 
$n-q_1\cdot n\le k\le n-q_0\cdot n$
with
$$
w_k=c,
$$
where $w(x)=w_1w_2w_3\dots$. The collection of eventually controlled points is denoted by $C_F\subset \OO_F$.
\end{defn}

\begin{lem}\label{controlPP} The set of eventually controlled points satisfies $\mu(C_F)=1$ and 
$$
C_F=\bigcup_{N\ge 1} \bigcap_{n\ge N} \PP_n.
$$ 
\end{lem}

\begin{proof} There exists $k^*\ge 1$ such that $(1-q_1)\cdot l(k)>k$ for $k\ge k^*$. Let $x\in C_F$. Choose $n\ge  1$ large enough such that 
$n\ge \kappa(n)\ge N_x$ and $\kappa(n)\ge k^*$. The piece $B_n(x)\in \BB^n$ contains $x$. Then $B_n(x)$ satisfies property (2).

Choose $k\ge \kappa(n)$. Then $l(k)> (1-q_1)\cdot l(k)> k\ge \kappa(n)\ge N_x$. Hence, there exists $w_i=c$ with
$(1-q_1)\cdot l(k)\le i\le (1-q_0)\cdot l(k)$. Now, $i\ge (1-q_1)\cdot l(k) >k$. Moreover, $i\le (1-q_0)\cdot l(k)<l(k)$. The piece $B_n(x)$ satisfies property (1). We proved,
\begin{equation}\label{xP}
x\in \bigcap_{\kappa(n)\ge \max \{N_x,k^*\} }\PP_n.
\end{equation}
Choose $x\in \bigcap_{n\ge N} \PP_n$. Then property (2) implies that for 
every $n\ge N$   there exists 
$n-q_1\cdot n\le k\le n-q_0\cdot n$
with
$$
w_k=c.
$$
We proved that 
$
 \bigcap_{n\ge N} \PP_n\subset C_F,
$
for $N\ge 1$.
The statement on the measure of $C_F$ follows from Proposition \ref{Pn1mintheta}.
This finishes the proof of Lemma \ref{controlPP}
\end{proof}

The recovery process can be described by using Proposition  \ref{muPnSn} and (\ref{xP})

\begin{prop}\label{recov} If $x\in \OO_F$ is controlled then and $\kappa(n)\ge N_x$ then
$
B_n(x)\in \SS_n(\theta^n).
$
\end{prop}

\begin{rem} Given a conjugation $h:\OO_{F_1}\to \OO_{F_2}$ then $b_{F_1}=b_{F_2}$, see \cite{LM1}, and 
$h(C_{F_1})=C_{F_2}$. The set of controlled points is a topological invariant.
\end{rem}

\begin{figure}[htbp]
\begin{center}
\psfrag{n}[c][c] [0.7] [0] {\Large $n$}
\psfrag{0}[c][c] [0.7] [0] {\Large $0$}
\psfrag{k}[c][c] [0.7] [0] {\Large $N_x$}
\psfrag{ki}[c][c] [0.7] [0] {\Large $k_i$} 
\psfrag{Nx}[c][c] [0.7] [0] {\Large$\kappa(n)$ }

\psfrag{lk}[c][c] [0.7] [0] {\Large $l_{k_i}$}
\psfrag{q0}[c][c] [0.7] [0] {\Large $n-q_0\cdot n$}
\psfrag{q1}[c][c] [0.7] [0] {\Large $n-q_1\cdot n$}
\pichere{0.9}{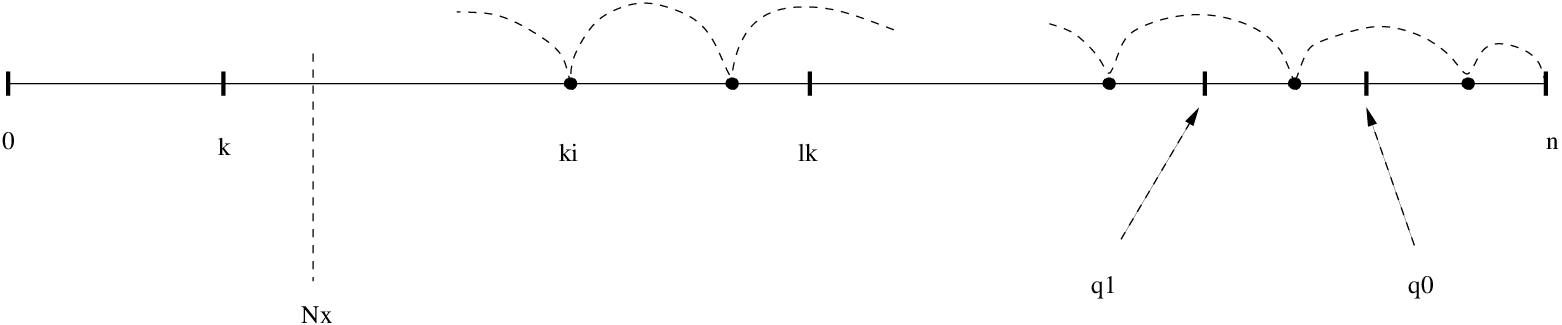} 
\caption{} \label{kqqn}
\end{center}
\end{figure}
\section{Probabilistic Rigidity}\label{haus}

The geometry of large parts of $\OO_F$ resemble that of the geometry of $\OO_{F_*}$, 
see Theorem \ref{distuni}, probabilistic universality.
The large parts are  
\begin{equation}\label{XN}
X_N=\bigcap_{k\ge N}\SS_k(\theta^k),
\end{equation}
where $\theta<1$ is given by Theorem \ref{distuni}, with 
$$
\mu(X_N)\ge 1-O(\theta^N).
$$
Let 
$$
X=\bigcup_{N\ge 1} X_N
$$
and note  $\mu(X)=1$.

As a consequence of a result from \cite{CLM} we known that there is no continuous line field on $\OO_F$ consisting of tangent lines to $\OO_F$. However, the first step towards describing the geometry of $\OO_F$ will be the construction of tangent lines to $\OO_F$ in all points of $X\subset\OO_F$.
Choose $N\ge 1$ and define for $n\ge N$
$$
T_n:X_N\to \Bbb{P}^1
$$
as follows. Let $x\in X_N$ and let $B_n(x)\in \BB^n$, $n\ge N$, be the piece with $x\in B_n(x)$. The part $\OO_F\cap B_n(x)$ is contained in a $\theta^n$-stick  see Figure \ref{hol}. The direction of the longest edge of this stick is denoted by $T_n(x)\in \Bbb{P}^1$.

The {\it a priori} bounds give that the scaling $\sigma_1$ of $B_{n+1}(x)$ is strictly away from 
zero. Namely,
$
\sigma_1=\sigma_{B_{n+1}(x)}\ge\sigma^*_{B_{n+1}(x)}-\theta^n\ge a>0. 
$
\begin{figure}[htbp]
\begin{center}

\psfrag{l}[c][c] [0.7] [0] {\Large $l$}
\psfrag{sl}[c][c] [0.7] [0] {\Large $\sigma_1 l$} 
\psfrag{Odl}[c][c] [0.7] [0] {\Large $O(\theta^nl)$}

\psfrag{Othe}[c][c] [0.7] [0] {\Large $O(\theta^n)$}

\pichere{0.7}{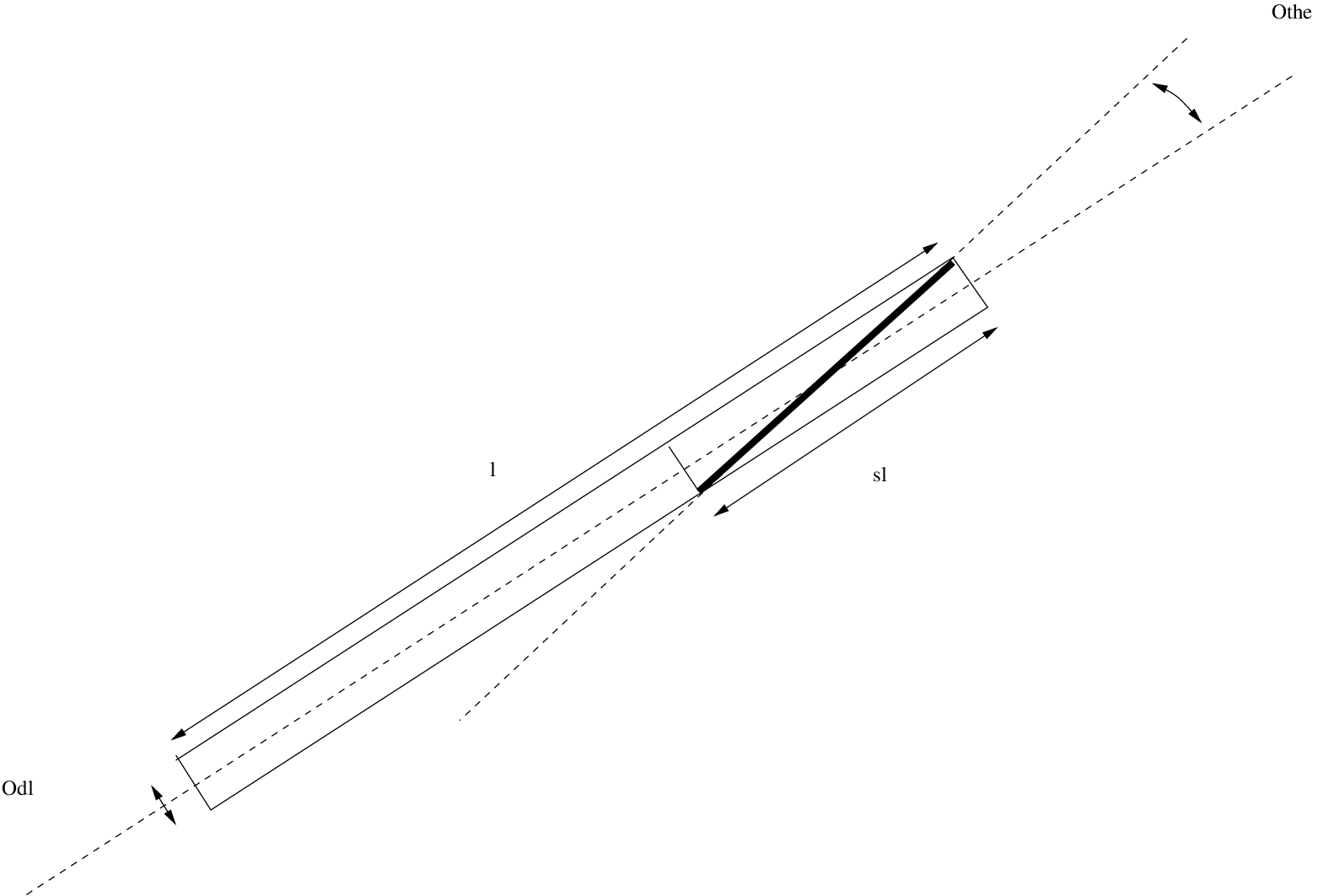} 
\caption{} \label{hol}
\end{center}
\end{figure}
The angle between $T_n(x)$ and $T_{n+1}(x)$ is of the order $\theta^n$, see Figure \ref{hol}. The piecewise constant functions $T_n$ form a Cauchy sequence,
\begin{equation}\label{Cbeta}
\text{dist}(T_{n+1}(x),T_n(x))=O(\theta^n).
\end{equation}
for $n\ge N$ and $x\in X_N$.
The limit is denoted by
$$
T=\lim_{n\to\infty}T_{n} : X_N\to \Bbb{P}^1.
$$
The construction implies that we get in fact a map
$$
T:X\to \Bbb{P}^1.
$$
The actual line through $x\in X\subset \OO_F$ with direction $T(x)$ is denoted by $T_x\subset \Bbb{R}^2$.

\begin{defn}\label{microh}
The Cantor set $\OO_F$ is {\it almost everywhere $(1+\beta)$-differentiable} 
if for each $N\ge 1$  there exists $C_N>0$ such that
$$
 dist(x, T_{x_0})\le C_N|x-x_0|^{1+\beta}
$$
when $x\in \OO_F$, $x_0\in X_N$.

The tangent line field of $\OO_F$ is {\it  weakly  $\beta$-H\"older } if for each $N\ge 1$ there exists $C_N>0$ such that
$$
\text{dist}(T(x_0),T(x_1))\le C_N|x_0-x_1|^\beta,
$$
with $x_0,x_1\in X_N$. 
\end{defn}

\comm{
\begin{defn}\label{microh}
The line field $T$  on $X$ is {\it almost everywhere $\beta$-H\"older} 
if for each $N\ge 1$ the restriction $T|X_N$
 is $\beta$-H\"older
$$
\text{dist}(T(x_0),T(x_1))\le C_N|x_0-x_1|^\beta,
$$
with $x_0,x_1\in X_N$. 

The line field $T$ consists of  {\it   $\beta$-H\"older tangent lines to $\OO_F$} if for each $N\ge 1$ there exists $C_N>0$ such that
$$
 dist(x, T_{x_0})\le C_N|x-x_0|^{1+\beta}
$$
when $x\in \OO_F$, $x_0\in X_N$.
\end{defn}
}

\begin{rem}  
 The objects we consider  have H\"older estimates on the growing sets $X_N$. Although, the increasing sequence of sets $X_1\subset X_2\subset X_3\subset \cdots$ is intrinsically related to the notion of being  almost everywhere H\"older we will suppress it in the notation, instead of using 
{\it almost everywhere H\"older with respect to the sequence $\{X_N\}$}.
\end{rem}

\begin{thm}\label{lines}  The Cantor set $\OO_F$ is almost everywhere $(1+\beta)$-differentiable, where $\beta>0$ is  universal. The tangent line field is weakly $\beta$-H\"older.
\end{thm}

\begin{proof} 
Choose $N\ge 1$.
Let 
$$
d_N=\min_{B\in \BB^N} \text{diam}(B\cap \OO_F)>0.
$$ 
Choose, $x_0,x_1\in X_N$. We will find a uniform H\"older estimate for the function $T|X_N$ in these two points.
Let $n\ge 1$ such that $x_1\in B_n(x_0)$ and $x_1\notin B_{n+1}(x_0)$. To prove a H\"older estimate we may assume that $n\ge N$.
 The {\it a priori} bounds for the Cantor set of the one-dimensional map $f_*$ and the probabilistic universality of $\OO_F$ observed in the sets $X_N$, see (\ref{XN}),  give a $\rho<1$ such that
$$
|x_1-x_0|\ge \rho^{n-N}\cdot d_N.
$$ 
Estimate (\ref{Cbeta}) implies
$$
\begin{aligned}
\text{dist}(T(x_1), T(x_0))&\le \text{dist}(T(x_1), T_{n}(x_1))+
\text{dist}(T_{n}(x_0), T(x_0))\\
&= O(\theta^n)\\
&\le C_N|x_1-x_0|^\beta.
\end{aligned}
$$
where $C_N=O(\frac{\theta^N}{(d_N)^\beta})$ and $\beta>0$ is such that 
\begin{equation}\label{beta}
\rho^\beta=\theta.
\end{equation}
The estimate only holds when $x_0$ and $x_1$ are in the same piece of $\BB^N$. To get a global estimate we might have to increase the constant to obtain
$$
\text{dist}(T(x_1), T(x_0))\le C_N |x_1-x_0|^\beta,
$$
for any pair $x_0,x_1\in X_N$.

Choose $x\in \OO_F$ to prove that $T_{x_0}$, $x_0\in X_N$, is a $\beta-$H\"older tangent 
line to $\OO_F$.
Again let
$n\ge 1$ such that $x\in B_n(x_0)$ and $x\notin B_{n+1}(x_0)$. The distance between $x_0$ and $x$ is bounded from below when $n<N$. To find the H\"older estimate for the distance between $x$ and $T_{x_0}$ we may assume that $n\ge N$.  
Recall, $\text{dist}(T(x_0), T_{n}(x_0))=O(\theta^n)$ and 
$|x-x_0|\ge \rho^{n-N}\cdot d_N$.
Denote the length of the stick which contains $\OO_F\cap B_n(x_0)$ by $l>0$. 
The a priori bounds imply
$$
l=O(|x-x_0|).
$$
Then
\begin{equation}\label{distxT}
\begin{aligned}
\text{dist}(x, T_{x_0})&= O(\theta^n)\cdot l\\
&=O((\rho^n)^\beta |x-x_0|)\\
&\le C_N|x-x_0|^{1+\beta}.
\end{aligned}
\end{equation}
This estimate holds when $x_0, x$ are in the same piece of $\BB^N$. We might have to increase the constant $C_N$ to get a global H\"older estimate.
\end{proof}

In \cite{CLM} it has been shown that the Cantor attractors $\OO_F$, with $b_F>0$, can not be part of a smooth curve. 

\begin{thm}\label{curve} Each set $X_N\subset \OO_F$ is contained in a 
$C^{1+\beta}$-curve.
\end{thm}

\begin{proof} The proof will not use the specific structure of the set $X_N$ described by the pieces in $\BB^n$. 
The proof holds for every closed set in the plane with tangents line to each point with H\"older dependence on the point.

We will construct a $C^{1+\beta}$-curve through every set $X_N\cap B$ with $B\in \BB^{N+K}$ and $K\ge 0$ large enough. 
This suffices to prove the Theorem.

Choose $B\in \BB^{N+K}$ with $X_N\cap B\ne \emptyset$. For each $x_0\in X_N\cap B$ consider the cusps
$$
S_{x_0}=\{x\in B| \text{dist}(x, T_{x_0})< C_N|x-x_0|^{1+\beta}\}.
$$
Note
$
X_N\cap B\subset S_{x_0}.
$
Thus 
$$
S\equiv \bigcap_{x\in X_N}S_{x}\supset X_N\cap B.
$$
Fix $K\ge 0$ large enough such that each $S_x\setminus \{x\}$ has two components.
This defines already an order on $X_N\cap B$. Write 
$$
S_x\setminus\{x\}=S^+_x\cup S^-_x,
$$
where $S^\pm_x$ are the connected components. We may assume that the assignment of connected components 
preserves the order in the following sense.
If $x_1\in S^+_{x_0}$ then
$$
S^+_{x_1}\cap X_N\subset S^+_{x_0}.
$$

A point $x\in X_N$ is a boundary 
point of $X_N$ if
$S^+_x\cap X_N=\emptyset$ or  $S^-_x\cap X_N=\emptyset$.
A connected component $G\subset S\setminus X_N$, see Figure \ref{gap}, is called a gap of $X_N$. For every gap there exist two boundary points $x_0, x_1\in X_N$ such that
$$
G\subset S^+_{x_0}\cap S^-_{x_1}.
$$
\begin{figure}[htbp]
\begin{center}
\psfrag{G}[c][c] [0.7] [0] {\Large $G$}
\psfrag{gG}[c][c] [0.7] [0] {\Large $\gamma_G$}
\psfrag{g}[c][c] [0.7] [0] {\Large $\gamma$}
\psfrag{S+}[c][c] [0.7] [0] {\Large $S^+_{x}$}
\psfrag{S-}[c][c] [0.7] [0] {\Large $S^-_{x}$}
\psfrag{x}[c][c] [0.7] [0] {\Large $x$}
\psfrag{A}[c][c] [0.7] [0] {\Large $a_0$}
\psfrag{x0}[c][c] [0.7] [0] {\Large $x_0$}
\psfrag{x1}[c][c] [0.7] [0] {\Large $x_1$}
\psfrag{Tx}[c][c] [0.7] [0] {\Large $T_{x}$}
\psfrag{B}[c][c] [0.7] [0] {\Large $a_1$} 
\pichere{1.0}{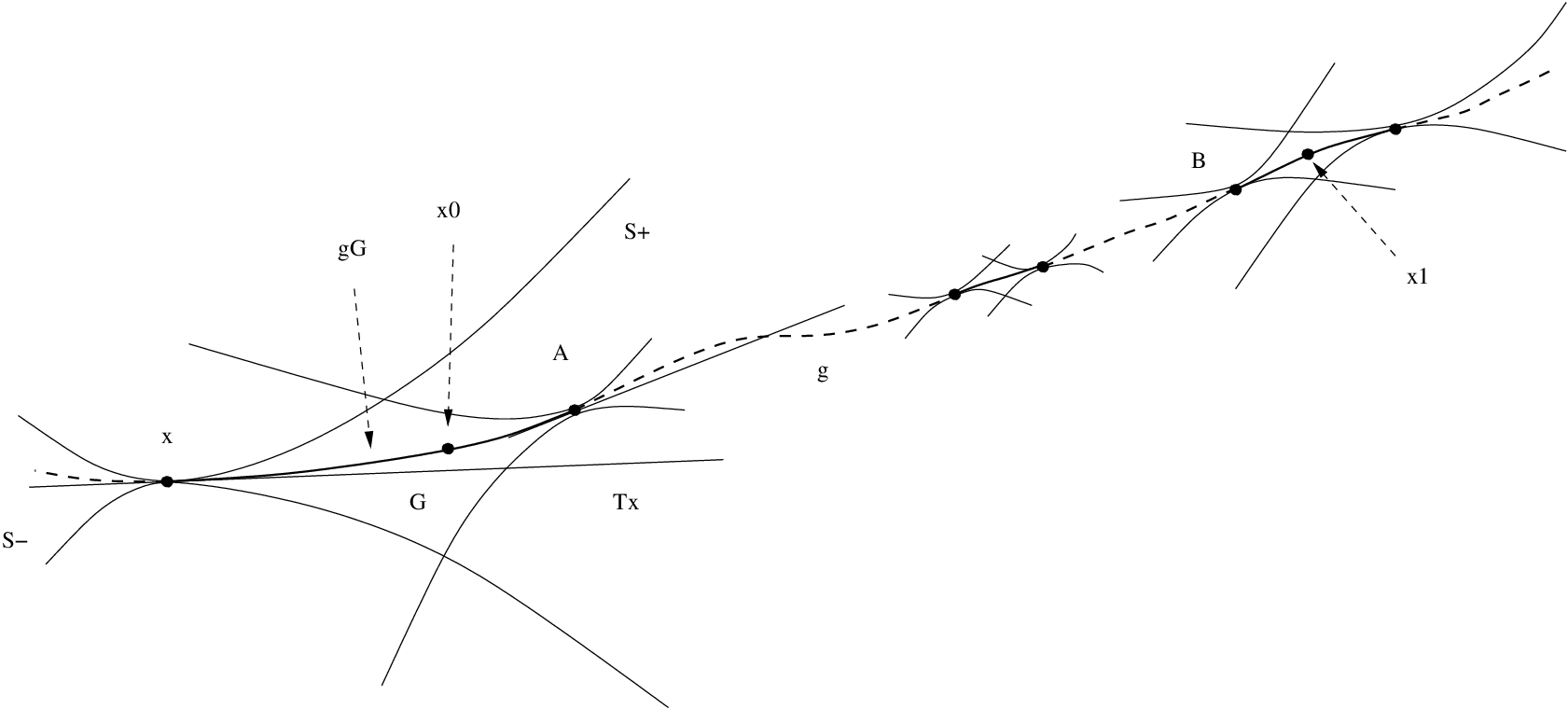} 
\caption{} \label{gap}
\end{center}
\end{figure}
Consider a gap between two boundary points $x_0$ and $x_1$ and the graph
over the tangent line $T_{x_0}$ of a cubic polynomial $\gamma_G$ which passes through $x_0$ and $x_1$ and is tangent to the tangent lines $T_{x_0}$ and $T_{x_1}$. Denote the graph of $\gamma_G$ also by $\gamma_G$. A calculation shows that
$$
|\gamma_G|_0\le 7 C_N|x_0-x_1|^{1+\beta}
$$
and if $D\gamma_G(x)\in \Bbb{P}^1$ is the direction of the tangent line to the graph $\gamma_G$ at a point $x\in \gamma_G$ then
$$
|D\gamma_G(y)-D\gamma_G(x)|\le 21 C_N|y-x|^\beta.
$$
In particular, the distance between the tangent directions along the curve and 
the direction at the boundary points shrink to zero as the diameter of the gap shrinks. This implies that the closure of the union of the curves $\gamma_G$
$$
\gamma=X_N \cup \bigcup_{G} \gamma_G
$$
is a $C^1$ curve. 

Left is to show that the tangent direction $D\gamma$ is $C^\beta$. Choose $x_0, x_1\in \gamma$. Let $a_0\in \gamma\cap  X_N$ be the closest 
point to $x_0$ on the line segment between $x_0$ and $x_1$. Similarly, let $a_1$ be the closest point to $x_1$. If $x_0\in G$ then $a_0$ is a boundary point of the gap $G$, See Figure \ref{gap}. For $K\ge 0$ large enough, the distances between these points are, up to a factor close to $1$, equal to the corresponding distances of the projections of these points to the tangent line through $a_0$.   We may assume that  $|x_1-a_1|,|a_1-a_0|,|a_0-x_0|\le 2|x_1-x_0|$. 
 Then 
$$
\begin{aligned}
|D\gamma(x_1)-D\gamma(x_0)|&\le
 C_N\cdot \{21|x_1-a_1|^\beta+ |a_1-a_0|^\beta+ 21|a_0-x_0|^\beta\}\\
&\le 86 C_N|x_1-x_0|^\beta.
\end{aligned}
$$
The curve $\gamma$ is $C^{1+\beta}$ and contains $X_N\cap B$.
\end{proof}

The following Theorem is an answer to a question posed by J.C. Yoccoz.

\begin{thm} The Cantor attractor $\OO_F$ is contained in a rectifiable curve without self-intersections.
\end{thm}

\begin{proof} Let $F_n:[0,1]^2\to [0,1]^2$ be the $n^{th}$-renormalization of $F$. 
The piece $B^1_v(F_n)\subset \Dom(F_n)$ is strip bounded between two horizontal line segments and $B^1_c(F_n)\subset \Dom(F_n)$ is strip bounded between two vertical line segments. Let $\gamma_n$ be a collection of three line segments which connects the two pieces and each piece with the horizontal boundaries of $\Dom(F_n)=[0,1]^2$, see Figure \ref{gn}.

\begin{figure}[htbp]
\begin{center}
\psfrag{F}[c][c] [0.7] [0] {\Large $F$}
\psfrag{Fk}[c][c] [0.7] [0] {\Large $F_k$}
\psfrag{Fk+1}[c][c] [0.7] [0] {\Large $F_{k+1}$}
\psfrag{Fn}[c][c] [0.7] [0] {\Large $F_n$}
\psfrag{gnk}[c][c] [0.7] [0] {\Large $g^n_k$}
\psfrag{gn}[c][c] [0.7] [0] {\Large $\gamma_n$}
\psfrag{gk}[c][c] [0.7] [0] {\Large $\gamma_k$}
\psfrag{Gnk+1}[c][c] [0.7] [0] {\Large $\Gamma^n_{k+1}$}
\psfrag{Gn0}[c][c] [0.7] [0] {\Large $\Gamma^n_{0}$}
\psfrag{Gnn}[c][c] [0.7] [0] {\Large $\Gamma^n_{n}$}
\psfrag{P}[c][c] [0.7] [0] {\Large $\psi^k_v$}
\pichere{1.0}{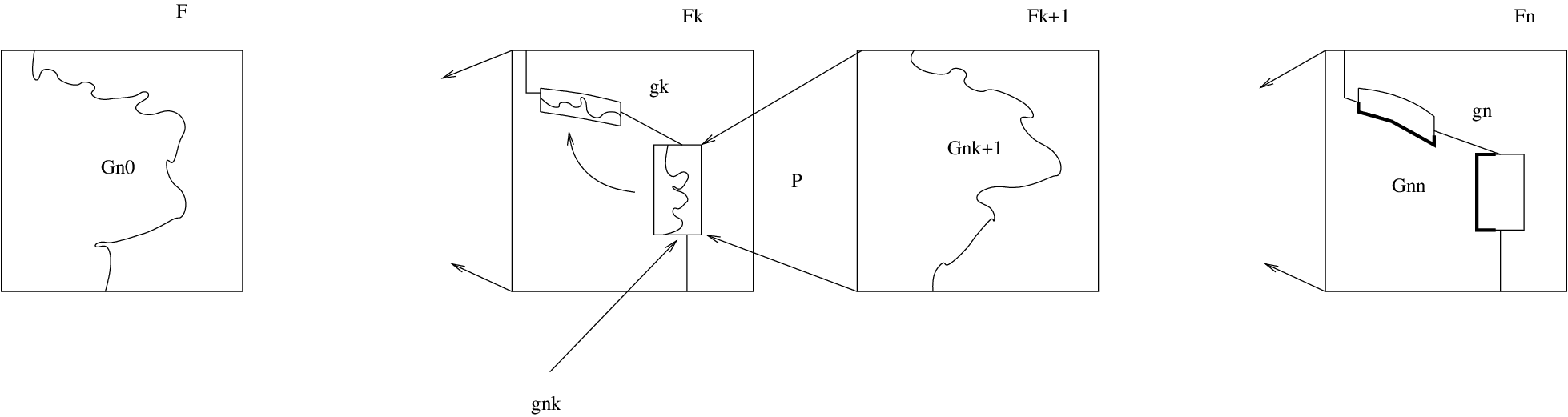} 
\caption{} \label{gn}
\end{center}
\end{figure}

For each $n\ge 1$ we will construct inductively a curve $\Gamma^n$ in the domain of $F$ which passes through all pieces $B\in \BB^n$ of the $n^{th}$-cycle of $F$. Let $\Gamma^n_n$ consists of $\gamma_n$ and curves in the boundaries of $B^1_v(F_n)$ and $B^1_c(F_n)$ connecting the end points of $\gamma_n$, see Figure \ref{gn}.

Suppose $\Gamma^n_{k+1}$ is defined and its end point are in the two horizontal boundary part of the domain of $F_{k+1}$, see Figure \ref{gn}. Let $\Gamma^n_k$ be the curve connecting the top and bottom of the domain of $F_k$ consists of the curves
$$
\Gamma^n_k=\psi^k_v(\Gamma^n_{k+1})\cup \psi^k_c(\Gamma^n_{k+1})\cup \gamma_k\cup g^n_k,
$$ 
where $g^n_k$ consists of the two shortest horizontal line segments connecting the endpoints of $ \psi^k_v(\Gamma^n_{k+1})$ with the end points of $\gamma_k$ and the two vertical line segments connecting the endpoints of $ \psi^k_v(\Gamma^n_{k+1})$ with the end points of $\gamma_k$, see Figure \ref{gn}. Let $\Gamma^n=\Gamma^n_0$.

The curve $\Gamma^{n+1}$ is obtained from $\Gamma^n$ by changing it inside the pieces of $\BB^n$. Hence, 
$$
\Gamma^{n+1}\setminus \BB^n=\Gamma^n \setminus \BB^n.
$$
This refinement process induces natural parametrizations of the curves $\Gamma^n$ where the parametrization of $\Gamma^{n+1}$ is obtained from the one of $\Gamma^n$ by only adjusting only inside the pieces of $\BB^{n}$. In each piece $B\in \BB^n$, the curve $\Gamma^{n+1}$ is partitioned into five sub-curves, see Figure \ref{gn}. The refinement of the parametrization of $\Gamma^n$ spends equal time in each of these five sub-curves. The diameter of the pieces in $\BB^n$ decay exponentially fast, $\sup_{B\in\BB_n} \diam(B)=O(\sigma^n)$. The construction and this decay imply that the parametrization have a uniform H\"older bound. This bound allows us to take a limit.  Let $\Gamma$ be the limiting H\"older curve. It contains $\OO_F$.

The maps $\psi^k_v$ and $\psi^k_c$ are contracting distance by at least $\frac{1}{2.5}$, for $k\ge 1$ large enough, see Lemma \ref{contracting}. 
Denote the length of $\Gamma^n_k$ by $|\Gamma^n_k|$. Then,
$$
\begin{aligned}
|\Gamma^n_k|&\le \frac{2}{2.5} \cdot |\Gamma^n_{k+1}|+|\gamma_k|+|g^n_k|\\
&\le \frac{2}{2.5} \cdot |\Gamma^n_{k+1}|+4.
\end{aligned}
$$
The curves $\Gamma^n_k$ have a bounded length. In particular, the limiting curve $\Gamma$ is rectifiable. 

Outside the pieces $B\in \BB^n$ the curve $\Gamma$ coincides with $\Gamma^n$ which consists of non-intersecting curves. A self-intersection has to be a point  $x\in \OO_F$. Let $B_n(x)\in \BB^n$ the piece which contains this self-intersection. The interval of parameter values which correspond to points in $B_n(x)$ is an interval of length $O(1/5^n)$. This means that the parametrization is injective. There are no self-intersections.
\end{proof}

\begin{rem} The curve $\Gamma$ for the degenerate maps follows the same combinatorial construction as for a non-degenerate maps. This implies that the order of the pieces $B\in \BB^n$ in the curve $\Gamma$ is the same order as  observed in 
one-dimensional maps.
\end{rem}

\begin{rem} The relative height (or thickness) of a piece $B\in \BB^n$ coincides with the number $\beta(B)\le 1$ introduced by P. Jones. In \cite{J}, Jones characterizes sets which are contained in rectifiable curves. A set $\OO$ is contained in a rectifiable curve if and only if its diadic covers $\BB^n$ satisfy the summability condition
$$
\sum_{n\ge 1} \sum_{B\in \BB^n} \beta^2(B)\cdot \diam{B}<\infty.
$$
In the present case of $\OO_F$, one can use the dynamical covers $\BB^n$ instead of the diadic ones.
Since $\diam(B)=O(\sigma^n)$ with $2\sigma<1$, the set $\OO_F$ satisfies the summability condition with respect to these covers. The diameter of the pieces decay fast enough so that we do not have to consider actual geometrical information of the pieces:  the bound $\beta(B)\le 1$ suffices. For completeness we include a direct proof for rectifiability using the strongly contracting rescalings $\psi^k_c$ and $\psi^k_v$. 

The sets $X_N$ have better geometrical properties. The relative height (or thickness) of the pieces covering $X_N$ and the corresponding numbers $\beta(B)$ decay exponentially fast. This is responsible for the smooth curves containing these sets.
\end{rem}

 The {\it tangent bundle} over $\OO_F$ is defined by 
$$
TX=\{(x,v)\in X\times \Bbb{R}^2| v\in T_x\}.
$$
If $Y\subset X$ then the tangent bundle over $Y$ is denoted by
$$
TY=\{(x,v)\in TX| x\in Y\}.
$$
We identify $T_x\subset \Bbb{R}^2$, $\{x\}\times T(x)\subset TX$ with the 
{\it tangent space } at $x\in X\subset \OO_F$. 
Let
$
\pi_x:\Bbb{R}^2\to T_x
$
be the orthogonal projection.

\bigskip

Let $Y\subset \OO_{F_1}$. A map $h:Y\to h(Y)\subset \OO_{F_2}$
 is differentiable at $x_0\in Y$ if
 $x_0$ and $h(x_0)$ have a tangent line,
and there exists a linear $Dh(x_0): T_{x_0}\to T_{h(x_0)}$ such that for $x\in Y$
$$
h(x)=h(x_0)+Dh(x_0)(\pi_{x_0}(x)-x_0)+o(|x-x_0|).
$$
We will identify $Dh(x_0)$ with a number.

A bijection $h:X\to h(X)\subset \OO_{F_2}$ is {\it almost everywhere a
$(1+\beta)$-diffeomorphism}  
if for each $N\ge 1$ the restriction $h|X_N$ is differentiable at each $x\in X_N$ and
$$
Dh: TX_N\to Th(X_N)
$$ 
and its inverse are $\beta$-H\"older homeomorphisms.

\bigskip

Let $\OO_{F_*}$ be the Cantor attractor of the  fixed point of renormalization, the degenerate map $F_*$. Its invariant measure is denoted by $\mu_*$.
In \cite{LM1} it has been shown that every conjugation
which extends to a homeomorphism between neighborhoods of $\OO_F$ and 
$\OO_{F_*}$ respects the orbits of the tips. We will only consider conjugations
$$
h:\OO_F\to \OO_{F_*}
$$
with  $h(\tau_{F})=\tau_{F_*}$.

\begin{defn}\label{defdistributionalrigidity} The attractor $\OO_F$ of an infinitely renormalizable H\'enon map $F\in\HH_{\Omega}(\overline{\epsilon})$ is {\it probabilistically rigid}  if there exists $\beta>0$ 
such that the restriction
$
h:X\to h(X)
$
of the conjugation $h:\OO_F\to \OO_{F_*}$, is almost everywhere a $(1+\beta)$-diffeomorphism. 
\end{defn}

\begin{thm}\label{distributionalrigidity} The Cantor attractor $\OO_F$ is probabilistically rigid. 
\end{thm}

\begin{proof} Fix $N\ge 1$ and choose  $B^0\in \SS_{N}(\theta^N)$ which intersects $X_N$. Consider the stick which contains $B^0$. 
Call one of the long edges of this stick the bottom and choose an orientation of this line segment.
It suffices to show the differentiability of the conjugation restricted to such a piece. 

We will construct a curve containing $X_N\cap B^0$. This curve will be the closure of a countable collection of pairwise disjoint line segments.
These line segments are called gaps. This piecewise affine curve is better  adapted to the problem at hand than the curve of 
Theorem \ref{curve}. Let
$$
\XX_N(k)=\{B\in \SS_k(\theta^k)| B\cap X_N\ne \emptyset \text{ and } B\subset B^0\}.
$$
Given $B\in \XX_N(k)$. Let $\delta>0$ be the relative height of the stick of $B$ and $\sigma_1, \sigma_2>0$ the scaling factors of the two 
pieces $B_1, B_2\in \BB^{k+1}$ contained in $B$. The stick of $B$ has three parts. Two rectangles of relative length $\sigma_1$ and $\sigma_2$ containing respectively $B_1$ and $B_2$ and the the complement within the stick. This last part does not intersect $X_N$. It could be that one of the other parts also does not intersect $X_N$. At least one of the parts does intersect $X_N$. Let $E$ be the union of the parts which do not intersect $X_N$ and $H_-$ and $H_+$ be the vertical boundaries of $E$, see Figure \ref{gapseg}. 

\begin{figure}[htbp]
\begin{center}
\psfrag{l}[c][c] [0.7] [0] {\Large $l$}
\psfrag{dl}[c][c] [0.7] [0] {\Large $\delta\cdot l$}
\psfrag{H-}[c][c] [0.7] [0] {\Large $H_-$}
\psfrag{H+}[c][c] [0.7] [0] {\Large $H_+$}
\psfrag{x-}[c][c] [0.7] [0] {\Large $x_-$}
\psfrag{x+}[c][c] [0.7] [0] {\Large $x_+$}
\psfrag{s1l}[c][c] [0.7] [0] {\Large $\sigma_1\cdot l$}
\psfrag{s2l}[c][c] [0.7] [0] {\Large $\sigma_2\cdot l$}
\psfrag{GB}[c][c] [0.7] [0] {\Large $G_B$}
\pichere{0.9}{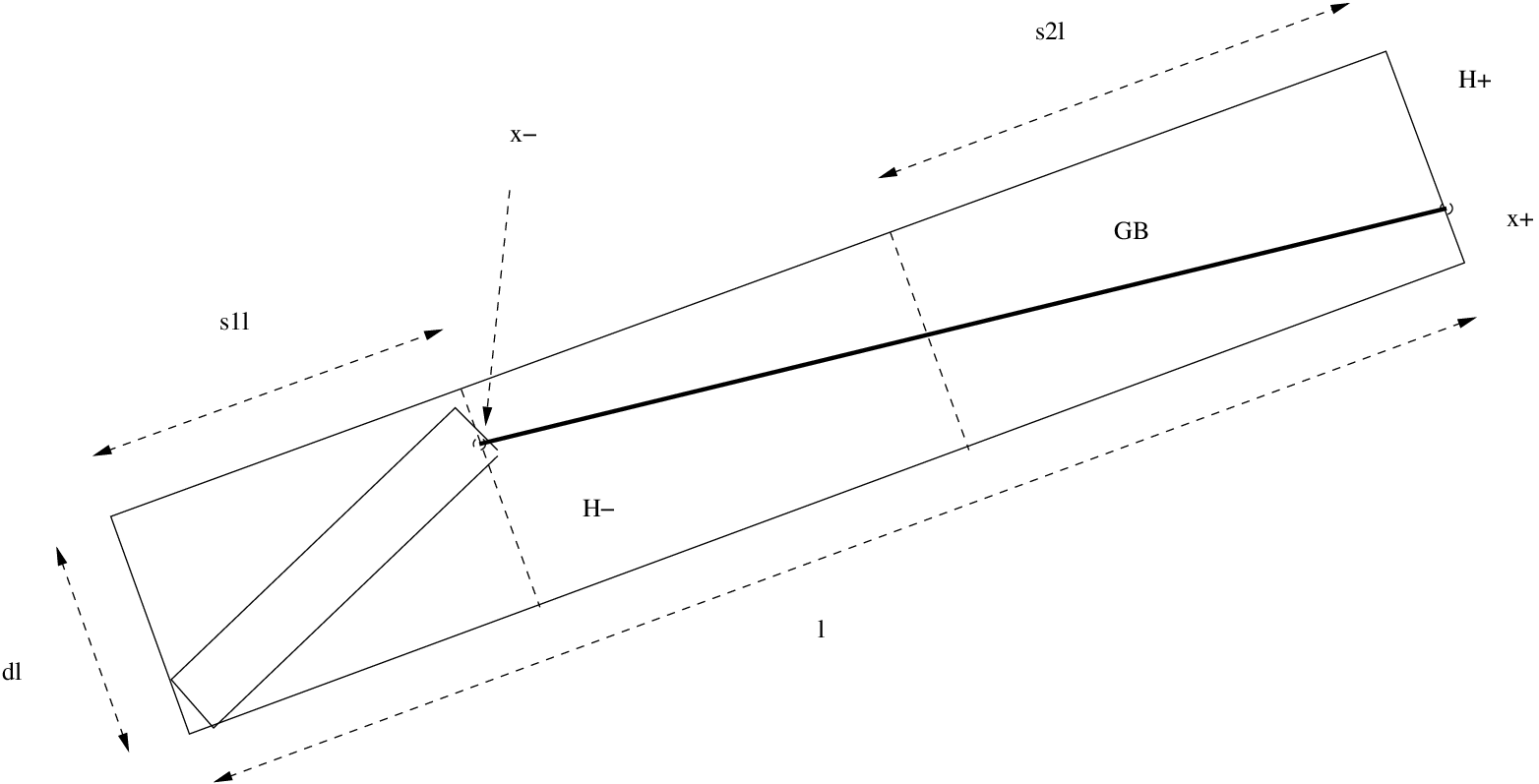} 
\caption{} \label{gapseg}
\end{center}
\end{figure}

The {\it gap} of $B$ will be a line segment $G_B$ connecting $H_-$ with $H_+$. Let $B_l\in \XX_N(l)$ which intersect $H_+$, $l=k, \dots, L$.  Choose 
$$
x^+_B\in H_+\cap \OO_F\cap \bigcap_{l=k}^L B_l.
$$
The point $x^+_B$ is uniquely defined when $L=\infty$. In fact, it will be a point of $X_N$. When $L<\infty$ we have some freedom choosing $x^+_B$. Choose it to be the closest point to the bottom of $B_0$.
Similarly, choose a point $x^-_B\in H_-$.  The gap of $B$, denoted by $G_B$, is the line segment  $(x^-_B, x^+_B)$.

The length of a gap is defined by 
$$
|G_B|\equiv |x^+_B-x^-_B|.
$$

\begin{rem} \label{endpt} The gaps are pairwise disjoint. For $B_1\in \XX_N(k+1)$ and $B\in \XX_N(k)$ it might happen that $G_{B_1}$ and 
$G_B$ have a common endpoint.  The angle between the gap $G_B$ and the bottom of $B\in \XX_N(k)$ is of order $\theta^k$. This is a consequence of $\delta=O(\theta^k)$ and the {\it a priori} bounds on $\sigma_1$ and $\sigma_2$.
\end{rem}

There is a natural order on $X_N\cap B^0$ and the collection of gaps. It coincides with the order of the projections of $X_N$ and 
the gaps onto the bottom of $B^0$.  Let us define the order between some $x\in X_N\cap B^0$ and a gap $G_{B_1}$. Let $k\ge N$ be maximal such that there is $B\in \XX_N(k)$ with $x\in B$ and $G_{B_1}\cap B\ne\emptyset$. 
The stick of $B$ has three parts as described above. Observe, $x$ and $G_{B_1}$ cannot be in the same part of the stick of $B$.
The angle of the axis of $B$ with the bottom of $B^0$ is of order $\theta^N$. This defines an order on the three parts of this stick.  Accordingly, this defines whether  $x>G_{B_1}$, or, $x<G_{B_1}$.

The {\it gap}-distance between $x,y\in X_N\cap B^0$ is
$$
|x-y|_g=\sum_{x<G_B<y} |G_B|.
$$
The gaps between $x,y\in X_N\cap B^0$ form a curve 
$$
[x,y]_g\equiv \overline{\bigcup_{x<G_B<y} G_B}.
$$
It is a graph over the tangent line  of $x$.

\begin{clm}\label{distg} If $x,y\in B\cap X_N$ with $B\in \XX_N(k)$
then
$$
\frac{|x-y|_g}{|x-y|}=1+O(\theta^k).
$$
\end{clm}

\begin{proof} Let $\pi_x$ be the projection onto the tangent line $T_x$ of $x$. Then
\begin{equation}\label{distg1}
|x-\pi_x(y)|=\sum_{x<G_{B'}<y} |\pi_x(G_{B'})|.
\end{equation}
The angle between each gap $G_{B'}$ between $x$ and $y$, and the tangent line of $x$ is of order $\theta^k$, see (\ref{Cbeta}) and remark \ref{endpt}. This implies that
\begin{equation}\label{distg2}
\frac{|\pi_x(G_{B'})|}{|G_{B'}|}=1+O(\theta^k).
\end{equation}
The Cantor set $\OO_F$ is almost everywhere differentiable, see Theorem \ref{lines}. In particular, use (\ref{distxT}) to obtain
\begin{equation}\label{distg3}
\frac{|x-\pi_x(y)|}{|x-y|}=1+O(\theta^k).
\end{equation}
The estimates (\ref{distg1}), (\ref{distg2}), and (\ref{distg3}) prove the Claim.
\end{proof}

Given a piece $B$ of $F$, the corresponding piece of $F_*$ is denoted by $B^*=h(B)$.

\begin{clm}\label{GG*}  Let $B_l\in \XX_N(l)$ with $B_l\subset B_k\in \XX_N(k)$. Then
$$
\ln\frac{|G_{B_l}|} {|G_{B_k}|}\cdot                        
    \frac{|G_{B^*_k}|}  {|G_{B^*_l}|} =O(\theta^k).
$$
\end{clm}

\begin{proof} The Claim holds for $l=k+1$ because the relevant pieces are in $\SS_k(\theta^k)$ and 
$\SS_{k+1}(\theta^{k+1})$.  
 In general, there is a unique sequence of pieces $B_j\in \XX_N(j)$, $k\le j\le l$ with
 $B_l\subset B_{l-1}\subset \dots B_{k+1}\subset B_k$. Then
 $$
 \ln\frac{|G_{B_l}|} {|G_{B_k}|}\cdot                        
      \frac{|G_{B^*_k}|}{|G_{B^*_l}|} 
     =\sum_{j=k}^{l-1}  \ln\frac{|G_{B_{j+1}}|} {|G_{B_{j}}|} \cdot                        
      \frac{|G_{B^*_{j}}|}  {|G_{B^*_{j+1}}|} =\sum_{j=k}^{l-1}  O(\theta^j)=O(\theta^k).
 $$
\end{proof}

\begin{clm}\label{xyz} Let $x,y,z\in X_N\cap B$ with $B\in \XX_N(k)$ and $x^*,y^*,z^*\in h(X_N)$ the corresponding images under $h$. Then
$$
\ln \frac{|x-y|_g}{|x-z|_g}\cdot \frac{|x^*-z^*|_g}{|x^*-y^*|_g}=O(\theta^k).
$$
\end{clm}

\begin{proof} Claim \ref{GG*} gives for every piece $\tilde{B}\subset B$
$$
|G_{\tilde{B}}|=  |G_{\tilde{B}^*}|\cdot   \frac{|G_{B}|} {|G_{B^*}|}\cdot (1+O(\theta^k)).
$$
This implies
$$
\begin{aligned}
\frac{|x-y|_g}{|x-z|_g}&=\frac{\sum_{x<\tilde{B}<y} |G_{\tilde{B}}|} {\sum_{x<\tilde{B}<z} |G_{\tilde{B}}|}\\
&=\frac{\sum_{x^*<\tilde{B}^*<y^*} |G_{\tilde{B}^*}|} {\sum_{x^*<\tilde{B}^*<z^*} |G_{\tilde{B}^*}|}\cdot (1+O(\theta^k))\\
&=\frac{|x^*-y^*|_g}{|x^*-z^*|_g}\cdot (1+O(\theta^k)).
\end{aligned}
$$
This finishes the proof of the Claim.
\end{proof}

A reformulation of this Claim is the following. Let $x,y,z\in X_N\cap B$ with $B\in \XX_N(k)$. Then
\begin{equation}\label{deltaDh}
|\ln \frac{|h(y)-h(x)|_g}{|y-x|_g}-\ln \frac{|h(z)-h(x)|_g}{|z-x|_g}|=O(\theta^k).
\end{equation}
This implies that for $x,y\in X_N\cap B^0$ the following limit exists.
$$
Dh(x)=\lim_{y\to x} \frac{|h(y)-h(x)|_g}{|y-x|_g}.
$$
Moreover, the limit depends continuously on $x$.

\begin{clm}\label{hHolder}  There exists a universal $\beta>0$, independent of $N$, such that $Dh:X_N\to \Bbb{R}$ is $\beta$-H\"older.
\end{clm}

\begin{proof}
Choose $x_0, x\in X_N\cap B^0$ to prove a H\"older estimate for $\ln Dh$. Let $k\ge N$ be maximal such that $x\in B_k(x_0)$.  
Observe, as before in the proof of Theorem \ref{lines}, 
$$
|x-x_0|\ge \rho^{k-N}\cdot \text{diam}(B^0)
$$
where $\rho<1$. Choose $\beta>0$ such that $\rho^\beta=\theta$.
Then
\begin{equation}\label{deltaxtheta}
\theta^k=O(|x-x_0|^\beta).
\end{equation}
Hence, using (\ref{deltaDh}) and (\ref{deltaxtheta}),
$$
|\ln Dh(x)-\ln Dh(x_0)|=O(\theta^k)=O(|x-x_0|^\beta).
$$
This suffices to show the H\"older bound for $Dh$.
\end{proof}

We will identify $Dh(x)$ with a linear map $Dh(x):T_x\to T_{h(x)}$.  
The positive function $Dh$ is bounded. This bound, (\ref{deltaDh}), and Claim \ref{distg}, imply that for $x,x_0\in X_N$
\begin{equation}\label{dhdx}
|h(x)-h(x_0)|=O(|x-x_0|).
\end{equation}

\begin{clm}\label{Deltah} For $x,y\in X_N\cap B^0$ 
$$
|h(y)-h(x)|=Dh(x)\cdot |x-y|\cdot (1+O(|x-y|^\beta)).
$$
\end{clm}

\begin{proof} Let $k\ge N$ be maximal such that $y\in B_k(x)$. 
Apply Claim \ref{distg},  (\ref{deltaDh}), and (\ref{deltaxtheta}),  in the following estimate
$$
\begin{aligned}
|h(y)-h(x)|&=\frac{|h(y)-h(x)|}{|h(y)-h(x)|_g}\cdot |h(y)-h(x)|_g\\
&=(1+O(\theta^k))\cdot Dh(x)\cdot |y-x|_g\\
&=(1+O(|y-x|^\beta))\cdot Dh(x)\cdot |y-x|.
\end{aligned}
$$
\end{proof}

Now we are prepared to show the differentiability of $h$. Choose $x,x_0\in X_N\cap B^0$. Let $k\ge N$ be maximal such that $x\in B_k(x_0)$. 
Let $\Delta=Dh(x_0)(\pi_{x_0}(x)-x_0)\in T_{h(x_0)}$. Claim \ref{Deltah}, (\ref{distg3}), and (\ref{deltaxtheta}), imply
\begin{equation}\label{delta}
|\Delta|=|h(x)-h(x_0)|\cdot (1+O(|x-x_0|^\beta)). 
\end{equation}
Let $J=\pi_{h(x_0)}(h(x))-h(x_0)\in T_{h(x_0)}$ and $V=h(x)-\pi_{h(x_0)}(h(x))$. The image $h(\OO_F)$ is contained in a smooth curve, the image of the degenerate map $F_*$. Hence,
\begin{equation}\label{JJ}
\begin{aligned}
|J|&=|h(x)-h(x_0)|\cdot (1+O(|h(x)-h(x_0)|^2))\\
&=|h(x)-h(x_0)|\cdot (1+O(|x-x_0|^\beta))
\end{aligned}
\end{equation}
and
\begin{equation}\label{V}
|V|=O(|h(x)-h(x_0)|^2).
\end{equation}
Apply (\ref{delta}), (\ref{JJ}), (\ref{V}), and (\ref{dhdx}), in the following estimate
$$
\begin{aligned}
h(x)&=h(x_0)+\Delta+(J-\Delta)+V\\
&=h(x_0)+\Delta+O(|h(x)-h(x_0)|\cdot |x-x_0|^\beta)+O(|h(x)-h(x_0)|^2)\\
&=h(x_0)+Dh(x_0)(\pi_{x_0}(x)-x_0)+O(|x-x_0|^{1+\beta}) .
\end{aligned}
$$
This finishes the proof of the differentiability and the Theorem.
\end{proof}

\begin{rem} The conjugation  $h:\OO_F \to \OO_{F_*}$ satisfies
$$
h(x)=h(x_0)+Dh(x_0)(\pi_{x_0}(x)-x_0)+O(|x-x_0|^{1+\beta})
$$
in almost every point $x_0\in \OO_F$. Observe, that the H\"older exponent is universal. The H\"older constant tends to infinity when $h$ is restricted to larger and larger sets $X_N$, when $N\to \infty$.
\end{rem}

The Cantor attractor $\OO_F$ has two characteristic exponents, \cite {O}.  
One is zero the other is $\ln b_F$, see \cite{CLM}. The function $T:X\to \Bbb{P}^1$ constructed before defines a measurable line field, with respect to $\mu$, on $\OO_F$. 

\begin{prop}\label{lambda=0} The line field 
$$
T:\OO_F \to \Bbb{P}^1
$$
is the invariant line field of zero characteristic  exponent.
\end{prop}

\begin{proof}
For each point $x_0\in X$ we have, see Theorem \ref{lines},
$$
dist(x, T_{x_0})\le C_{x_0}|x-x_0|^{1+\beta}
$$
with $x\in \OO_F$. The map $F$ is a diffeomorphism which preserves $\OO_F$. 
Hence,
$$
dist(x, DF(x_0)T_{x_0})=O(|x-F(x_0)|^{1+\beta})
$$
with $x\in \OO_F$. For almost every $x_0\in X$ we have $F(x_0)\in X$. Hence, $T$ is an invariant line field, i.e. for almost every $x_0\in \OO_F$ we have
$$
DF(x_0)T_{x_0}=T_{F(x_0)}.
$$
The map $F$ has only two invariant lines fields, the two characteristic directions, \cite{O}. Left is to show that $T(x)$ corresponds to the zero exponent.  

Choose $N\ge 1$. For almost every $x_0\in X_N$ there are $t_n\to \infty $ such 
that
$$
F^{t_n}(x_0)\in X_N.
$$
This is because the ergodic measure $\mu$ assigns positive measure to $X_N$. 
Let $v\in T_{x_0}$ and $v_*\in T_{h(x_0)}$ be unit vectors. Apply the chain rule
$$
\begin{aligned}
|DF^{t_n}(x_0)v|&= |Dh^{-1}(F_*(h(x_0)))|\cdot |DF_*^{t_n}(h(x_0)) Dh(x_0)v| \cdot |Dh(x_0)|\\
&\asymp |DF_*^{t_n}(h(x_0)) v_*|.
\end{aligned}
$$
Observe, $v_*\in T_{h(x_0)}$ which is a tangent line to the graph of $f_*$. The degenerate H\'enon map $F_*$ has zero exponential contraction along this curve. Hence,
$$
\lim_{t\to\infty} \frac{1}{t}\ln |DF^t(x_0)v|=
\lim_{n\to\infty} \frac{1}{t_n}\ln |DF^{t_n}(x_0)v|=0
$$
On a set of full measure in $X_N$ there is no exponential contraction along the direction $T(x)$. The line field $T$ has exponent zero. 
\end{proof}

The Hausdorff dimension of a measure $\mu$ on a metric space $\OO$ is defined as
$$
HD_{\mu}(\OO)=\inf_{\mu(X)=1} HD(X).
$$

\begin{thm}\label{HD} The Hausdorff dimension of the invariant measure is universal
$$
HD_\mu(\OO_F)=HD_{\mu_*}(\OO_{F_*}).
$$
\end{thm}

\begin{proof} Let $h:\OO_F\to \OO_{F_*}$ be a conjugation which exchanges the orbits of the tips. According to Theorem \ref{distributionalrigidity} there are sets $X_N\subset \OO_F$ with $\mu(X_N)\ge 1-O(\theta^N)$ and on which $h$ is a 
$(1+\beta)$-diffeomorphism. The continuity of the derivative gives upper and lower bounds of the derivative. This implies
$$
HD(h(X_N))=HD(X_N).
$$
Hence, for 
$
X=\bigcup_{N\ge 1} X_N
$
and every $Z\subset \OO_F$
$$
HD(h(X\cap Z))=HD(X\cap Z).
$$
Let $Z_N\subset \OO_F$ with $\mu(Z_N)=1$ and $\lim_{N\to \infty} HD(Z_N)=HD_\mu(\OO_F)$ then
$$
\begin{aligned}
HD_\mu(\OO_F)&\ge \lim_{N\to\infty} HD(Z_N\cap X)\\
&=\lim_{N\to\infty} HD(h(Z_N\cap X))\\
&\ge HD_{\mu_*}(\OO_{F_*}),
\end{aligned}
$$
where the last inequality holds because $\mu_*(h(Z_N\cap X))=\mu(Z_N\cap X)=1$.
The opposite inequality $HD_{\mu_*}(\OO_{F_*})\ge HD_\mu(\OO_F)$ is obtained in the same way.
\end{proof}

\begin{rem} We can identify the Hausdorff dimension of the measure on the Cantor attractor. Namely,
$$
HD_\mu(\OO_F)=\frac{\ln 2 }{\int \ln | Dr_*| d\mu_*}.
$$
where $r_*$ is the analytic  expanding one dimensional map constructed such that $ \pi_1(\OO_{F_*})$ is its invariant Cantor set, see for example \cite{BMT} and references therein. The measure $\mu_*$ is the projected measure from 
$\OO_{F_*}$.
\end{rem}

\section*{Appendix: Open Problems}\label{problems}

Let us finish with some questions related to the previous discussion.
\bigskip

\noindent
\underline{Problem I:}
 The collections $\PP_n$, see (\ref{PPn}), of good pieces that we have  constructed are determined by the average Jacobian of the map. 
Observe that $\SS_n(\theta^n)$ might be slightly larger than  $\PP_n$. 
It was suggested by Feigenbaum's experiment, mentioned in the introduction, that 
the statistics of the remaining  {\it bad} pieces,
 might be governed by some universality law. 
 This problem is also related to one of the open problems in \cite{CLM} on the regularity of the conjugation $h:\OO_F\to \OO_G$ when $b_F=b_G$.
\bigskip

\noindent
\underline{Problem II:} Do wandering domains exist? This question was already formulated in \cite{LM1}. It is included again because its solution might be obtained by using the techniques developed in this paper.
\bigskip

\section*{Nomenclature}   
\begin{itemize}
\item[$b_F$] average Jacobian, \S \ref{prelim}
\item[$B^n_\omega$] a piece of the $n^{th}$-renormalization cycle, \S \ref{prelim}
\item[$\BB^n$] collection of pieces in the $n^{th}$-renormalization cycle, \S \ref{prelim}
\item[${\BB^n[k]}$] pieces of $\BB^n$ in $E^k$, \S \ref{pieces}
\item[$B_n(x)$] the piece in $\BB^n$ containing $x\in \OO_F$
\item[$\bB$] the piece $B$ viewed from its proper scale, \S \ref{pieces}, Figure \ref{figregpiece}
\item[$\text{Dist}(\phi)$] Distortion, (\ref{distnonl})
\item[$D_k$] derivative of $\psi^k_v$ at the tip, (\ref{Dk})
\item[$\delta_B$] thickness of $B$, \S \ref{stic}
\item[$\Delta_B$] absolute thickness of $B$, \S \ref{stic}
\item[$E^k$] part of a dynamical partition, \S \ref{pieces}, Figure \ref{figE}, \ref{figEsch}
\item[$f_*$] unimodal renormalization fixed point, \S \ref{prelim}
\item[$G_k$] return map related to the partition by $E^k$, \S \ref{pieces}, Figure \ref{figE}, \ref{figEsch}
\item[$k_i(B)$] depth of the $i^{th}$-predecessor of $B$, \S \ref{pieces}, Definition \ref{controlled}
\item[$\kappa_0(n)$] minimal depth to safely push-up, \S \ref{conv}, and \S \ref{defstat}
\item[$\kappa(n)$] upper bound of the brute-force regime, \S \ref{conv}, and Lemma \ref{kn}
\item[$l(k)$] maximal allowable depth, \S \ref{pieces}
\item[$\eta_\phi$] Nonlinearity, (\ref{nonlin})
\item[$\OO_F$] invariant Cantor set of $F$, \S \ref{prelim}
\item[$\psi^k_{c,v}$] coordinate changes related to the renormalization $R(R^kF)$, (\ref{phi})
\item[$\psi^{n}_\omega$] coordinate change, \S \ref{prelim}
\item[$\Psi^{n}_k$] coordinate change relating $R^{n-k}(R^kF)$ to $R^n$,  (\ref{Phi})
\item[$\PP_n(k; q_0,q_1)$] collection of $q_0,q_1$-controlled pieces, Definition \ref{controlled}
\item[$\PP_n$] pieces obtained by applying the three regimes, \S \ref{conv}, and (\ref{PPn})
\item[$q_0,q_1$] boundary one-dimensional regime, \S \ref{conv}, and Lemma \ref{q0q1}
\item[$\sigma$] scaling factor of the unimodal renormalization fixed point, \S \ref{prelim}
\item[$\sigma_B$] scaling factor of $B$, \S \ref{scaling}
\item[$\SS^n(\epsilon)$] collection of pieces in $\BB^n$ with $\epsilon$ precision, \S \ref{scaling}
\item[$t_k$] tilt of the derivative of $\psi^k_v$ at the tip, (\ref{Dk})
\item[$T$] tangent line field to $\OO_F$, \S \ref{haus}
\item[$\tau_F$] tip, \S \ref{prelim}
\item[$X$] the differentiable part of $\OO_F$, \S \ref{haus}
\end{itemize}

\end{document}